\documentclass[a4paper, 11pt, reqno]{amsart}

\usepackage{amssymb}
\usepackage{amsthm}
\usepackage{fullpage}
\usepackage{graphicx}
\usepackage{subfig}

\usepackage{amsmath, mathtools}
\usepackage{empheq}
\usepackage{xypic}
\usepackage{color}
\usepackage{float}
\usepackage{relsize}

\usepackage{vmargin}

\usepackage[colorlinks,citecolor=blue,filecolor=black,linkcolor=blue,urlcolor=black]{hyperref}
\usepackage[all,cmtip]{xy}

\numberwithin{equation}{section}
\numberwithin{equation}{subsection}

\newcommand{\C}{\ensuremath{\mathbb C}}

\newcommand{\R}{\ensuremath{\mathbb R}}

\theoremstyle{plain}
\newtheorem{thm}{Theorem}[section]
\newtheorem*{thm*}{Theorem}
\newtheorem{lem}[thm]{Lemma}
\newtheorem{prop}[thm]{Proposition}
\newtheorem{cor}[thm]{Corollary}
\newtheorem*{cor*}{Corollary}
\newtheorem*{prop*}{Proposition}

\newtheorem*{lemma*}{Lemma}
\newtheorem*{claim*}{Claim}

\theoremstyle{definition}
\newtheorem{defn}[thm]{Definition}

\newtheorem*{exmp*}{Example}
\newtheorem*{defn*}{Definition}
\newtheorem*{rem*}{Remark}
\newtheorem*{note*}{Note}

\begin{document}

\title{Higher order Seiberg-Witten functionals and their associated gradient flows}

\author{Hemanth Saratchandran}
\address{Beijing International Center for Mathematical Research, No. 5 Yiheyuan Road, Haidian
District, Beijing 100871, China P.R.}
\email{hemanth.saratchandran@gmail.com}
\date{\today}

\maketitle 
\parskip=0.2cm

\begin{abstract}
We define functionals generalising the Seiberg-Witten functional on closed $spin^c$ manifolds, involving
higher order derivatives of the curvature form and spinor field. We then consider their associated gradient flows 
and,
using a gauge fixing technique, are able to prove short time existence for the flows. We then prove
energy estimates along the flow, and establish local $L^2$-derivative estimates. These are then
used to show long time existence of the flow in sub-critical dimensions. In the critical dimension, we are able
to show that long time existence is obstructed by an $L^{k+2}$ curvature concentration phenomenon.
\end{abstract}

\tableofcontents

\section{Introduction}
In their investigations into the gauge theory of 4-manifolds, N. Seiberg and E. Witten introduced a set of
equations \cite{seiberg_1}, \cite{seiberg_2}, 
now known as the Seiberg-Witten equations, which they then used 
to construct new differential invariants of 4-manifolds. The invariants defined by Seiberg and Witten, through
these equations, were closely related to the Donaldson invariants \cite{donaldson}, but rose to prominence when
it was observed that they were much simpler to work with, and at times could lead to stronger results
than could be obtained through Donaldson theory.

The Seiberg-Witten equations are a system of first order equations, and have a naturally
associated energy functional, the Seiberg-Witten functional. Given a $spin^c$ manifold $M$, the 
Seiberg-Witten functional is the functional
\begin{equation*}
\mathcal{SW}(\phi, A) = \int_{M} (|F_A|^2 + |\nabla_A\phi|^2 + \frac{S}{4}|\phi|^2 + \frac{1}{8}|
\phi^4)d\mu  +  \pi^2c_1(\mathcal{L}^2)
\end{equation*}
where $\phi$ is a positive spinor, $A$ a unitary connection on the determinant line bundle $\mathcal{L}^2$ 
(associated to the $spin^c$ structure on $M$), and $\nabla_A$ the $spin^c$ connection induced by $A$.
The importance of this functional comes from the fact that
solutions to the Seiberg-Witten equations are absolute minima of the Seiberg-Witten functional. This leads
to a variational approach to study the equations.

The variational aspects of the Seiberg-Witten equations were first studied by Jost, Peng, and Wang  in
\cite{jost_2}. In that paper, they considered the Seiberg-Witten functional, and
proved regularity for weak solutions to the Euler-Lagrange equations associated to the functional. Furthermore, they
proved that the Seiberg-Witten functional satisfies the Palais-Smale condition.

In \cite{hong} Hong and Schabrun introduced the Seiberg-Witten flow, which is the gradient flow associated to the
Seiberg-Witten functional. They were able to demonstrate long time existence of the flow and showed that, upto
gauge transformations, the solution converged to a unique limit, which was then a solution of the 
Euler-Lagrange equations associated to the Seiberg-Witten functional. This behaviour is analogous to the 
behaviour of the Yang-Mills flow in dimensions 2 and 3, and the results can 
be seen as similar to those obtained by R{\aa}de \cite{rade} for the Yang-Mills flow. Schabrun then generalised 
these results to the higher dimensional case in \cite{schabrun}. 

In this paper, we study higher order variants of the Seiberg-Witten functional. Given a $spin^c$ 
Riemannian manifold $M$ of dimension $n$,
and a positive integer $k$, we consider the functionals
\begin{equation*}
\mathcal{SW}^k(A, \phi) = \int_M \big{(}\frac{1}{2}|\nabla_M^{(k)}F_A|^2 + |\nabla_A^{(k)}\nabla_A\phi|^2
                                                 + \frac{S}{4}|\phi|^2 + \frac{1}{8}|\phi|^4\big{)}d\mu  
                                                 +  \pi^2c_1(\mathcal{L}^2)
\end{equation*}    
defined on pairs $(\phi, A)$, where $\phi$ is a positive spinor, and $A$ is a unitary connection on the
determinant line bundle $\mathcal{L}^2$, associated to the $spin^c$ structure on $M$. We are using
$\nabla_M$ to denote the Levi-Civita connection on $M$, $\nabla_M^{(k)}$ and $\nabla_A^{(k)}$ mean
$k$ iterations of these covariant derivatives. 

Critical points of the above functionals satisfy Euler-Lagrange equations, that are higher order 
generalisations of those coming from the Seiberg-Witten functional. In view of the work of Hong and
Schabrun \cite{hong}, we consider the negative gradient flow associated to the functionals, which takes the form
\begin{align*}
\frac{\partial\phi}{\partial t} &= -\nabla_A^{*(k+1)}\nabla_A^{(k+1)}\phi - \frac{1}{4}(S + |\phi|^2)\phi  
\\
\frac{\partial A}{\partial t} &= (-1)^{k+1}d^*\Delta_M^{(k)}F_A - \sum_{v=0}^{2k-1}P_1^{(v)}[F_A] - 2iIm\big{(}\sum_{i=0}^{k}C_i\nabla_M^{*(i)}\langle \nabla_{A}^{(k)}\nabla_{A}\phi, 
\nabla_{A}^{(k-i)}\phi\rangle\big{)}
\end{align*}
where $\Delta_M$ denotes the Bochner Laplacian associated to $\nabla_M$, 
$P_1^{(v)}[F_A]$ denotes a lower order curvature term (see \ref{notation} for an explicit definition).

As the above gradient flow has order $2(k+1)$, the technique
of using maximum principles and Harnack inequalities 
to understand the behaviour of solutions is no longer available. It is in this regard that these higher 
order flows become significantly more difficult to analyse than their second order counterparts. 
The usual approach one takes, is to obtain localised $L^2$-derivative estimates, and energy
estimates for solutions along the flow. Together with the Sobolev embedding theorem, these estimates
are often robust enough to conclude long time existence for sub-critical dimensions.

Higher order functionals have been studied by a few authors, in different settings. In \cite{giorgi_1}, 
\cite{giorgi_2} 
E. De Giorgi 
studies compact 
n-dimensional hypersurfaces in $\R^{n+1}$, evolving via the gradient flow of a functional involving higher
order derivatives of the curvature. He conjectures that the flow does not develop singularities in finite time. 
This was part of his program to study singular flows by approximating them by sequences of smooth ones, which
involved higher order derivatives. 
Mantegazza in \cite{mantegazza} studies higher order generalisations of the mean curvature 
flow, by introducing a family of higher order functionals. He is able to show (theorem 7.8 \cite{mantegazza}),
that provided the derivatives
in his functionals are large enough, singularities in finite time do not occur.
Inspired by this, Kelleher in \cite{kelleher} considers a higher order variant of the Yang-Mills flow. She proves
long time existence in sub-critical dimensions (theorem A in \cite{kelleher}), 
and is able to prove a curvature concentration 
phenomenon in the critical dimension (theorem B \cite{kelleher}), 
analogous to the result obtained by Struwe (theorem 2.3 \cite{struwe})  for the 
Yang-Mills flow in dimension four. 

A recurrent feature in the study of higher order functionals is that
the critical dimension increases with respect to the order
of derivatives. Thus, provided the order of the derivatives are sufficiently high (depending on the dimension
of the manifold), the associated gradient flows will not develop singularities in finite time. 

Our main results are that in dimension $n < 2(k+2)$ (sub-critical dimension), finite time singularities do not
occur, and solutions to the flow exist for all time,  see theorem \ref{main_theorem_1}. Furthermore, 
when $n = 2(k+2)$ we cannot rule out finite time singularities, but we show that if present, they are due
to an $L^{k+2}$ curvature concentration phenomenon, see proposition \ref{curvature_concentration} and 
theorem \ref{main_theorem_2}. This is analogous to what Kelleher observes for the higher order Yang-Mills flow
(theorem B \cite{kelleher}). However,
this is in contrast with the work of Hong and Schabrun 
(theorem 1 in \cite{hong}) and 
Schabrun (theorem 1 in \cite{schabrun}), on the Seiberg-Witten flow, who are able to show that an 
$L^2$ curvature concentration phenomenon can obstruct
long time existence, but are able to rule out such concentration by a careful rescaling argument together
with an $L^2$ energy estimate.
In our case, we observe that
curvature is concentrating in $L^{k+2}$, and while we are able to obtain $L^2$ energy estimates, 
these are not sufficient enough to prove long time existence via the methods of Hong and Schabrun.

The paper is organised as follows. In section \ref{prelims}, we outline the notation we will be using and
explain some basic theory on the action of the gauge group. In section \ref{HSW_analysis}, we derive some
variational formulae for time dependent connections and spinor fields, and then compute the Euler-Lagrange 
equations. The section ends by introducing the higher order Seiberg-Witten gradient flow. In section
\ref{short_time}, we prove short time existence using a gauge fixing technique. Sections \ref{energy_estimates} and
\ref{local_estimates} consider energy and local $L^2$-derivative estimates for solutions of the flow, which
are then used to prove estimates of Bernstein-Bando-Shi type, and show that the only obstruction to long time
existence is curvature blow up. In section \ref{blowup_analysis}, we construct a blow up solution for finite
time solutions admitting a singularity in finite time. In
Section \ref{long_time}, we prove long time existence in the sub-critical dimension, and then show
that in the critical dimension, long time existence is obstructed due to the $L^{k+2}$-norm of the curvature
form
concentrating in smaller and smaller balls. Finally, in section \ref{conclusion} we end with some
concluding remarks.


\section{Preliminaries}\label{prelims}

\subsection{Background and notation}\label{notation}

In this short section, we outline the setup and notation we will be using throughout the paper.

We will let $(M, g)$ denote a smooth, closed Riemannian manifold of dimension $n$. Its canonical Levi-Civita 
connection will be denoted by $\nabla_M$, and the Riemannian volume form will be denote by $d\mu$.
The metric $g$ will be extended to define a metric on all tensor powers
$\bigotimes_{r} T^*M \otimes \bigotimes_sTM$. 
We remind the reader that the Levi-Civita connection can also be extended
to all tensor powers $\bigotimes_{r} T^*M \otimes \bigotimes_sTM$, and we will denote any such extension by
$\nabla_M$ as well. 

As we will be dealing with complex bundles, we will normally be working with the complexification $TM_{\C}$, and
$T^*M_{\C}$. The metric $g$ can be canonically extended to these complexified spaces. We will also extend
the connection $\nabla_M$, to be $\C$-linear, on these complexfied spaces.

Throughout the paper, we will
assume $M$ is a $spin^c$ manifold, with a fixed $spin^c$-structure $\textit{\textbf{s}}$. We will denote the
spinor bundle by $\mathcal{S} = W \otimes \mathcal{L}$, and by $\mathcal{S}^{\pm} = W^{\pm}\otimes \mathcal{L}$ 
the half spinor bundles, with $\mathcal{L}^2$ denoting the corresponding determinant line bundle. 
The spinor
bundles, and the half spinor bundles, will all be assumed to have fixed Hermitian metrics. As we will primarily 
deal with $\mathcal{S}^{+}$, we will call sections of this bundle spinor fields (we should really be calling them
positive spinor fields, but as we will never be considering the negative spinor fields, it seems unnecessary to
need to distinguish them by using the adjective ``positive"). Denote smooth
sections of this bundle by $\Gamma(\mathcal{S}^+)$ (see \cite{moore} and \cite{jost_2}, for more on the
background of these constructions).

A unitary connection on $\mathcal{L}^2$ will be denoted by $A$, recall that $A \in i\Lambda^1(M)$. We denote
the curvature 2-form associated to $A$ by $F_A = dA \in i\Lambda^2(M)$. The space of smooth unitary connections on
$\mathcal{L}^2$ will be denoted by $\mathfrak{A}$.
The $spin^c$ connection defined on $\mathcal{S}$, $\mathcal{S}^{\pm}$, and coming from the $spin^c$ structure 
$\textit{\textbf{s}}$ and the unitary
connection $A$, will be denoted by $\nabla_A$. Locally, we can express $\nabla_A$ by
\begin{equation*}
\nabla_A = d + (\omega + A)
\end{equation*}
where $\omega$ is induced by the Levi-Civita connection and Clifford multiplication 
(see \cite{moore}).
The curvature of
$\nabla_A$ will be denoted by $\Omega_A$.  Furthermore, with respect to the hermitian metrics on 
$\mathcal{S}$ and $\mathcal{S}^{\pm}$, $\nabla_A$ is metric compatible.

Once we have connections on our bundles, we can define their $L^2$-adjoints. We will denote the $L^2$-adjoints
by $\nabla_M^*$, for the Levi-Civita connection, and $\nabla_A^*$, for the adjoint of the $spin^c$ connection.
Thus for example, we have that locally we can write $\nabla_M^* = -g^{ij}(\nabla_M)_i(\nabla_M)_j$.

Using the connection $\nabla_M$, we can extend $\nabla_A$ to any tensor power 
$\bigotimes_{r} T^*M \otimes \mathcal{S}^{\pm}$. We will denote this extended connection again simply by 
$\nabla_A$, as is the usual practice in the literature. Once one has extended the connection to all such
tensor powers, it is then possible to define composed operators of the form: 
$\nabla_A \circ\cdots\circ\nabla_A$, we will often denote such a composition by $\nabla_A^{(j)}$, where 
$j$ is supposed to indicate that we compose $j$ times.

We also point out that the complexified Riemannian metric, together with the Hermitian metric on 
$\mathcal{S}^+$, allow us to naturally define an inner product on any tensor power 
$\bigotimes_{r} T^*M \otimes \mathcal{S}^{\pm}$, which in the course of proofs we will simply denote
by $\langle \hspace{0.1cm},\hspace{0.1cm}\rangle$.

Given a spinor $\phi$, and
$p, q \in \mathbb{N}$, with $p \geq q$, we will often use the notation
$\langle\nabla^{(p)}\phi, \nabla^{(q)}\nabla\rangle$, which will represent a $p-q$ tensor. To see this, 
write $p = q + r$, then $\langle\nabla^{(p)}\phi, \nabla^{(q)}\nabla\rangle
=\langle\nabla^{(r)}\nabla^{(q)}\phi, \nabla^{(q)}\phi\rangle$. We can then define a multilinear map
\begin{equation*}
\langle\nabla^{(r)}\nabla^{(q)}\phi, \nabla^{(q)}\phi\rangle :
T^*M \otimes \cdots \otimes T^*M \rightarrow \C
\end{equation*}
by $\langle\nabla^{(r)}\nabla^{(q)}\phi, \nabla^{(q)}\phi\rangle(X_1,\ldots,X_r) = 
\langle\nabla_{X_1}\cdots\nabla_{X_r}\nabla^{(q)}\phi, \nabla^{(q)}\phi\rangle$.

The connections also give rise to Laplacian operators. We will denote the Bochner (or rough) Laplacians
associated to $\nabla_M$ and $\nabla_A$ by, $\Delta_M = \nabla_M^*\nabla_M$, 
and $\Delta_A = \nabla_A^*\nabla_A$ respectively. Furthermore, we will need the Hodge Laplacian on $M$, which we
denote by $\Delta_M = dd^* + d^*d$, where $d$ denotes the exterior derivative and $d^*$ its adjoint.

Given tensors $S$, and $T$ on $M$, we let $S * T$ denote any multilinear form obtained
from $S \otimes T$ in a universal bilinear way. Therefore, $S * T$ is obtained by starting with
$S \otimes T$, taking any linear combination of this tensor, raising and lowering indices, taking any
number of metric contractions (i.e. traces), and switching any number of factors in the product.
We then have that 
\begin{equation*}
\vert S * T\vert \leq C\vert S\vert\vert T\vert 
\end{equation*}
where $C > 0$ will not depend on $S$ or $T$. Furthermore, we have 
$\nabla(S * T) = \nabla(S) * \nabla(T)$, and in general we can write
$\nabla^{(k)}(S*T) = \sum_{i=0}^{k}C_i(\nabla^{(i)}S * \nabla^{(k-i)}T)$, for some constants $C_i > 0$.
For example, the tensor $\langle\nabla^{(p)}\phi, \nabla^{(q)}\nabla\rangle$, defined above, can be written
as $\nabla^{(p)}\phi * \nabla^{(q)}\phi$.

We will also make use of the $P$ notation, as introduced in \cite{kuwert} p. 314. 
Given a tensor $\omega$, we denote by 
\begin{equation*}
P_n^{(k)}[\omega] = \nabla_M^{(i_1)}\omega*\nabla_M^{(i_2)}\omega*\cdots*\nabla_M^{(i_n)}\omega * T.
\end{equation*}
where $i_1 + \dots + i_n = k$, and $T$ is any background tensor depending on only on the metric $g$.
In our case, most of the time $T$ will be the curvature tensor $Rm$ associated to $\nabla_M$
(or some contraction of it).

Finally, during the course of many estimates, constants will change from line to line. We will often use
the practise of denoting these new constants by the same letter. We will also have many constants depending
on the metric $g$. We will often denote such a constant by $C(g)$, and will also use this notation to denote
constants that depend on any derivatives of the metric. For example, if we obtain a constant $C$ that depends
on the Riemann curvature tensor, we will simply denote this constant by $C(g)$. As the metrics are
not changing with respect to time, this notation should not cause any confusion.

\subsection{Action of the gauge group}

The gauge group on $\mathcal{L}^2$ is given by $Aut(\mathcal{L}^2)$. As $\mathcal{L}^2$ is a line bundle, we
can identify the gauge group with $\mathcal{G} = \{g : M \rightarrow U(1)\}$.

Given the connection $\nabla = d + A$ on $\mathcal{L}^2$ (remember $A \in i\Lambda^1(M)$), we define
the action of the gauge group $\mathcal{G}$ on $\nabla$ as follows. Let $\zeta \in \mathcal{G}$, then
we define a new connection $\zeta^*\nabla$ by
\[ \zeta^*\nabla = \zeta^{-1}\circ \nabla \circ \zeta. \]
Locally, we can express the connection $\zeta^*\nabla$ as
\[ \zeta^*\nabla = d + \zeta^{-1}d\zeta + A. \]

We claim that the curvature $F_{\zeta*\nabla}$ associated to $\zeta^*\nabla$ is actually equal to
$F_{\nabla}$, that is the curvature is invariant under the gauge group. To see this, recall that 
given $\nabla = d + A$, we have that locally $F_{\nabla} = dA$. Therefore, using the formula above, we find that
\[F_{\zeta^*\nabla} = d(\zeta^{-1}d\zeta + A) = d(\zeta^{-1}d\zeta) + dA = 0 + A \]
where we have used the fact that $d(\zeta^{-1}d\zeta) = d(\zeta^{-1})\wedge d\zeta = 0$, as $d\zeta$ is a 1-form.
In particular, for any $k > 0$ we have that $\nabla_M^{(k)}F_{\zeta*\nabla} = \nabla_M^{(k)}F_{\nabla}$.

On the $spin^c$ connection $\nabla_A$, the gauge group $\mathcal{G}$ acts by 
\[ \zeta^*\nabla_A = \zeta^{-1} \circ \nabla_A \circ \zeta. \]
Writing $\nabla_A = d + (\omega + A)$, we find that 
\[ \zeta^{-1} \circ \nabla_A \circ \zeta = d + (\omega + AI) + \zeta^{-1}d\zeta 
= \nabla_A + \zeta^{-1}d\zeta.\]
We point out that the gauge group acts in a similar way on the adjoint: 
$\zeta^*\nabla_A^* = \zeta^{-1} \circ \nabla_A^* \circ \zeta$. 

The action of the gauge group on a spinor field $\phi$ is defined by $\zeta^*\phi = \zeta^{-1}\phi$.

The higher order Seiberg-Witten functional is invariant under the
action of the gauge group $\mathcal{G}$. In fact, it is precisely due to this symmetry that the associated 
higher order Seiberg-Witten gradient flow is not parabolic. As we will see, in order to prove short time existence
of the flow one has to resort to a gauge fixing procedure.

\section{The higher order Seiberg-Witten gradient flow}\label{HSW_analysis}

In this section, we start our analysis of a family of higher order functionals generalising the 
Seiberg-Witten functional. Given a pair $(\phi, A)$, in the configuration space 
$\Gamma(\mathcal{S}^+) \times \mathfrak{A}$, we define the higher order Seiberg-Witten functionals by
\begin{equation}
\mathcal{SW}^k(A, \phi) = \int_M \big{(}\frac{1}{2}|\nabla_M^{(k)}F_A|^2 + |\nabla_A^{(k)}\nabla_A\phi|^2
                                                 + \frac{S}{4}|\phi|^2 + \frac{1}{8}|\phi|^4\big{)}d\mu  
                                                 +  \pi^2c_1(\mathcal{L}^2). \label{HSW_functional}
\end{equation}
When considering the gradient flow associated to these functionals, we note that
the term $\pi^2c_1(\mathcal{L}^2)$ does not change along the flow. Therefore, we will simply
leave it out.

The main difference between these functionals, and the usual Seiberg-Witten functional, is the
higher order derivatives $\nabla_M^{(k)}F_A$, and $\nabla_A^{(k)}\nabla_A\phi$ present in the functional.
The presence of such higher order derivatives makes the associated gradient flow a higher order system, and
this in turn makes their analysis much more involved. In this section, we will start by deriving variational
formulas for the above functional, and then move on to working out their associated Euler-Lagrange equations.
This will then allow us to define their associated gradient flow, which will be the main topic of this paper.

\subsection{Formulas for variations}	
We start by deriving formulas for variations in the configuration space 
$\Gamma(\mathcal{S}^+) \times \mathfrak{A}$. These formulas will prove useful when computing 
the Euler-Lagrange equations.

\begin{lem}\label{variation}
$\frac{\partial}{\partial t}\nabla_A^{(k)}\phi = \nabla_A^{(k)}\frac{\partial \phi}{\partial t} + 
\sum_{i = 0}^{k-1} C_i\nabla_M^{(i)}\frac{\partial A}{\partial t} \otimes  \nabla_A^{(k-1-i)}\phi$, for some constants $C_i > 0$.  
\end{lem}

\begin{proof}
We prove this by induction. For the case $k = 1$, observe that locally we can write 
$\nabla_A = d + (\omega + AI)$. Differentiating this equation with respect to time we obtain, 
$\frac{\partial}{\partial t}(\nabla_A) = \frac{\partial A}{\partial t}$. It then follows that
\[  \frac{\partial}{\partial t}(\nabla_A\phi) = \frac{\partial A}{\partial t} \otimes \phi + \nabla_A\frac{\partial \phi}{\partial t}.\]
This proves the formula for $k = 1$. For the general case, assume the formula is true
for $k-1$. We then have 
\[ \frac{\partial}{\partial t}(\nabla_A^{(k)}\phi) = 
\frac{\partial\nabla_A}{\partial t} \otimes \nabla_A^{(k-1)}\phi
+ \nabla\bigg(\frac{\partial}{\partial t}(\nabla^{k-1}\phi)\bigg).\]
Applying the $k=1$ case and the induction hypothesis, we have that the right hand side of the above equation can
be written as
\[\frac{\partial A}{\partial t} \otimes \nabla_A^{(k-1)}\phi + \nabla_A\bigg(\nabla_A^{(k-1)}\frac{\partial \phi}{\partial t} + 
\sum_{i=0}^{k-2} C_i \nabla_M^{(i)}\frac{\partial A}{\partial t} \otimes \nabla_A^{(k-2-i)}\phi\bigg)\]
which then simplifies to
\[\frac{\partial A}{\partial t} \otimes \nabla_A^{(k-1)}\phi + \nabla_A\nabla_A^{(k-1)}\frac{\partial \phi}{\partial t} 
   + \sum_{i=0}^{k-2}C_i(\nabla_M\nabla_M^{(i)}\frac{\partial A}{\partial t} \otimes \nabla_A^{(k-2-i)}\phi) 
+ C_i(\nabla_M^{(i)}\frac{\partial A}{\partial t} \otimes \nabla_A\nabla_A^{(k-2-i)}\phi).\]
Collecting terms we then arrive at the required formula
\[\nabla_A^{(k)}\frac{\partial \phi}{\partial t} + \sum_{i=0}^{k-1}\tilde{C_i}\nabla_M^{(i)}\frac{\partial A}{\partial t} \otimes \nabla_A^{(k-1-i)}\phi.\]

\end{proof}

Note that as $F_A = dA$, we have that $\frac{\partial F_A}{\partial t} = d\frac{\partial A}{\partial t}$.

\subsection{The Euler-Lagrange equations and the associated gradient flow}

In this subsection we compute the Euler-Lagrange equations associated to the higher order functionals
\eqref{HSW_functional}. Towards the end of this subsection, we will define their associated gradient flow.

\begin{prop}\label{euler_lagrange}
The Euler-Lagrange equations associated to the functional
\begin{equation*}
\mathcal{SW}^k(A, \phi) = \int_M \big{(}\frac{1}{2}|\nabla_M^{(k)}F_A|^2 + |\nabla_A^{(k)}\nabla_A\phi|^2
                                                 + \frac{S}{4}|\phi|^2 + \frac{1}{8}|\phi|^4\big{)}d\mu  
                                                 +  \pi^2c_1(\mathcal{L}^2)
\end{equation*}
are given by
\begin{align*}
\nabla_A^{*(k+1)}\nabla_A^{(k+1)}\phi + \frac{1}{4}(S + |\phi|^2)\phi &= 0 \\
(-1)^kd^*\Delta_M^{(k)}F_A + \sum_{v=0}^{2k-1}P_1^{(v)}[F_A] + 2iIm\big{(}\sum_{i=0}^{k}C_i\nabla_M^{*(i)}\langle \nabla_{A}^{(k)}\nabla_{A}\phi, 
\nabla_{A}^{(k-i)}\phi\rangle\big{)}\big{\rangle} &= 0.
\end{align*}
\end{prop}
\begin{proof}
We start with the term $\int |\nabla_A^{(k)}\nabla_A\phi|^2$. We have to obtain formulas for variations
in the unitary connection $A$, and variations in the spinor field $\phi$. 

Let $A_t$ be a path of unitary connections, with $A(0) = A$, on $\mathcal{L}^2$. We then compute:

\begin{align*}
\frac{\partial}{\partial t}\bigg{\vert}_{t = 0} \int \big{\langle} \nabla_{A_t}^{(k)}\nabla_{A_t}\phi, 
\nabla_{A_t}^{(k)}\nabla_{A_t}\phi\big{\rangle} 
&= \int \frac{\partial}{\partial t}\big{\langle} \nabla_{A_t}^{(k)}\nabla_{A_t}\phi, 
\nabla_{A_t}^{(k)}\nabla_{A_t}\phi\big{\rangle}\big{\vert}_{t = 0} \\
&= \int \big{\langle} \frac{\partial}{\partial t} \nabla_{A_t}^{(k)}\nabla_{A_t}\phi, 
\nabla_{A_t}^{(k)}\nabla_{A_t}\phi 
\big{\rangle}\big{\vert}_{t = 0}  +
\big{\langle} \nabla_{A_t}^{(k)}\nabla_{A_t}\phi, \frac{\partial}{\partial t} \
\nabla_{A_t}^{(k)}\nabla_{A_t}\phi\big{\rangle}\big{\vert}_{t = 0}.  
\end{align*}

Using lemma \ref{variation} we can then write this latter integral (forgetting about the evaluation at 
$t = 0$ for a moment) as
\[\int \big{\langle} \sum_{i=0}^{k} C_i \nabla_M^{(i)}\frac{\partial A}{\partial t} \otimes \nabla_{A_t}^{(k-i)}\phi, 
\nabla_{A_t}^{(k)}\nabla_{A_t}\phi \big{\rangle} + \big{\langle} \nabla_{A_t}^{(k)}\nabla_{A_t}\phi, 
\sum_{i=0}^{k} C_i \nabla_M^{(i)}\frac{\partial A}{\partial t} \otimes \nabla_{A_t}^{(k-i)}\phi \big{\rangle} \]
which we can then express as
\[\int \sum_{i=0}^{k}\big{\langle} \nabla_M^{(i)}\frac{\partial A}{\partial t}, C_i\langle \nabla_{A_t}^{(k)}\nabla_{A_t}\phi, 
\nabla_{A_t}^{(k-i)}\phi\rangle \big{\rangle} + 
\sum_{i=0}^{k} \big{\langle}C_i \langle \nabla_{A_t}^{(k-i)}\phi, \nabla_{A_t}^{(k)}\nabla_{A_t} \rangle, 
\nabla_M^{(i)}\frac{\partial A}{\partial t} \big{\rangle}.\]
Taking adjoints, and simplifying, we have that the above integral can be written as
\begin{align*}
&{}\int \sum_{i=0}^{k}\big{\langle} \frac{\partial A}{\partial t}, C_i\nabla_M^{*(i)}\langle \nabla_{A_t}^{(k)}\nabla_{A_t}\phi, 
\nabla_{A_t}^{(k-i)}\phi\rangle \big{\rangle} +
\big{\langle} C_i\nabla_M^{*(i)}\langle \nabla_{A_t}^{(k)}\nabla_{A_t}\phi, \nabla_{A_t}^{(k-i)}\phi\rangle, 
\frac{\partial A}{\partial t} \big{\rangle} \\
=& \int \sum_{i=0}^{k}\big{\langle} \frac{\partial A}{\partial t}, C_i\nabla_M^{*(i)}\langle \nabla_{A_t}^{(k)}\nabla_{A_t}\phi, 
\nabla_{A_t}^{(k-i)}\phi\rangle \big{\rangle} +
\overline{\big{\langle} \frac{\partial A}{\partial t}, C_i\nabla_M^{*(i)}\langle \nabla_{A_t}^{(k)}\nabla_{A_t}\phi, 
\nabla_{A_t}^{(k-i)}\phi\rangle \big{\rangle}} \\
=& \int \sum_{i=0}^{k}\big{\langle} \frac{\partial A}{\partial t}, C_i\nabla_M^{*(i)}\langle \nabla_{A_t}^{(k)}\nabla_{A_t}\phi, 
\nabla_{A_t}^{(k-i)}\phi\rangle \big{\rangle} +
\big{\langle} \overline{\frac{\partial A}{\partial t}}, C_i\overline{\nabla_M^{*(i)}\langle \nabla_{A_t}^{(k)}\nabla_{A_t}\phi, 
\nabla_{A_t}^{(k-i)}\phi\rangle} \big{\rangle} \\
=& \int \sum_{i=0}^{k}\big{\langle} \frac{\partial A}{\partial t}, C_i\nabla_M^{*(i)}\langle \nabla_{A_t}^{(k)}\nabla_{A_t}\phi, 
\nabla_{A_t}^{(k-i)}\phi\rangle \big{\rangle} -
\big{\langle} \frac{\partial A}{\partial t}, C_i\overline{\nabla_M^{*(i)}\langle \nabla_{A_t}^{(k)}\nabla_{A_t}\phi, 
\nabla_{A_t}^{(k-i)}\phi\rangle} \big{\rangle}  
\end{align*}
where, in order to get the last line, we have used the fact that $A_t$ are unitary connections, hence we can write 
$A_t = ia_t$, with $a_t$ a real valued one form on $M$. 

We can then further simplify the above integral as follows.
\begin{align*}
&\int \sum_{i=0}^{k}\big{\langle} \frac{\partial A}{\partial t}, C_i\nabla_M^{*(i)}\langle \nabla_{A_t}^{(k)}\nabla_{A_t}\phi, 
\nabla_{A_t}^{(k-i)}\phi\rangle \big{\rangle} -
\big{\langle} \frac{\partial A}{\partial t}, C_i\overline{\nabla_M^{*(i)}\langle \nabla_{A_t}^{(k)}\nabla_{A_t}\phi, 
\nabla_{A_t}^{(k-i)}\phi\rangle} \big{\rangle}  \\
=& \int \sum_{i=0}^{k}\big{\langle} \frac{\partial A}{\partial t}, C_i\nabla_M^{*(i)}\langle \nabla_{A_t}^{(k)}\nabla_{A_t}\phi, 
\nabla_{A_t}^{(k-i)}\phi\rangle - C_i\overline{\nabla_M^{*(i)}\langle \nabla_{A_t}^{(k)}\nabla_{A_t}\phi, 
\nabla_{A_t}^{(k-i)}\phi\rangle} \big{\rangle} \\
=& \int \big{\langle} \frac{\partial A}{\partial t}, \sum_{i=0}^{k}C_i\nabla_M^{*(i)}\langle \nabla_{A_t}^{(k)}\nabla_{A_t}\phi, 
\nabla_{A_t}^{(k-i)}\phi\rangle - C_i\overline{\nabla_M^{*(i)}\langle \nabla_{A_t}^{(k)}\nabla_{A_t}\phi, 
\nabla_{A_t}^{(k-i)}\phi\rangle} \big{\rangle} \\
=& \int \big{\langle} \frac{\partial A}{\partial t},2iIm\big{(}\sum_{i=0}^{k}C_i\nabla_M^{*(i)}\langle \nabla_{A_t}^{(k)}\nabla_{A_t}\phi, 
\nabla_{A_t}^{(k-i)}\phi\rangle\big{)}\big{\rangle}. \\
\end{align*}

Putting this together we finally obtain the following formula, for variations with respect to $A_t$.
\[\frac{\partial}{\partial t}\bigg{\vert}_{t = 0} \int \big{\langle} \nabla_{A_t}^{(k)}\nabla_{A_t}\phi, 
\nabla_{A_t}^{(k)}\nabla_{A_t}\phi\big{\rangle} = \int \big{\langle} \frac{\partial A}{\partial t},2iIm\big{(}\sum_{i=0}^{k}C_i\nabla_M^{*(i)}\langle \nabla_{A_t}^{(k)}\nabla_{A_t}\phi, 
\nabla_{A_t}^{(k-i)}\phi\rangle\big{)}\big{\rangle}\big{\vert}_{t=0}.\]  
The next step is to compute variations with respect to the spinor field. Let $\phi_t$ be a path of spinors, we
then need to compute
\begin{align*}
&\hspace{0.5cm}
\frac{\partial}{\partial t}\bigg{\vert}_{t = 0} \int \big{\langle} \nabla_{A}^{(k)}\nabla_{A}\phi_t, 
\nabla_{A}^{(k)}\nabla_{A}\phi_t\big{\rangle} \\
&= \int \frac{\partial}{\partial t}\big{\langle} \nabla_{A}^{(k)}\nabla_{A}\phi_t, 
\nabla_{A}^{(k)}\nabla_{A}\phi_t\big{\rangle}\big{\vert}_{t = 0} \\
&= \int \big{\langle}\frac{\partial}{\partial t} \nabla_{A}^{(k)}\nabla_{A}\phi_t, 
\nabla_{A}^{(k)}\nabla_{A}\phi_t\big{\rangle}\big{\vert}_{t=0} + 
\big{\langle}\nabla_{A}^{(k)}\nabla_{A}\phi_t, 
\frac{\partial}{\partial t}\nabla_{A}^{(k)}\nabla_{A}\phi_t\big{\rangle}\big{\vert}_{t=0} \\
&= \int \big{\langle} \nabla_{A}^{(k)}\nabla_{A}\frac{\partial\phi_t}{\partial t}, 
\nabla_{A}^{(k)}\nabla_{A}\phi_t\big{\rangle}\big{\vert}_{t=0} +
\big{\langle}\nabla_{A}^{(k)}\nabla_{A}\phi_t, 
\nabla_{A}^{(k)}\nabla_{A}\frac{\partial\phi_t}{\partial t}\big{\rangle}\big{\vert}_{t=0} \\
&= \int \big{\langle} \frac{\partial\phi_t}{\partial t}, 
\nabla_{A}^{*}\nabla_{A}^{*(k)}\nabla_{A}^{(k)}\nabla_{A}\phi_t\big{\rangle}\big{\vert}_{t=0} +
\big{\langle}\nabla_{A}^{*}\nabla_{A}^{*(k)}\nabla_{A}^{(k)}\nabla_{A}\phi_t, 
\frac{\partial\phi_t}{\partial t}\big{\rangle}\big{\vert}_{t=0}.
\end{align*}
We now move on to deal with the curvature term in the higher order Seiberg-Witten functional. Recall, this
term is given by the integral $\int \frac{1}{2}|\nabla_M^{(k)}F_A|^2$, we therefore need to compute a formula
for the variation with respect to a path of unitary connections $A_t$. We remind the reader that
$F_{A_t} = dA_t$, and the unitary condition on $A_t$ means that we can write $A_t = ia_t$, where
$a_t$ is a real valued one form.

\begin{align*}
\frac{\partial}{\partial t} \int \frac{1}{2} \langle \nabla_M^{(k)}F_{A_t}, \nabla_M^{(k)}F_{A_t} \rangle &= 
\int \frac{1}{2} \langle \nabla_M^{(k)}d\frac{\partial A}{\partial t}, \nabla_M^{(k)}F_{A_t}\rangle + 
\frac{1}{2} \langle \nabla_M^{(k)}F_{A_t}, \nabla_M^{(k)}d\frac{\partial A}{\partial t} \rangle \\
&= \int \langle \nabla_M^{(k)}d\frac{\partial A}{\partial t}, \nabla_M^{(k)}F_{A_t}\rangle \\
&= \int \langle \frac{\partial A}{\partial t}, d^*\nabla_M^{*(k)}\nabla_M^{(k)}F_{A_t}\rangle.     
\end{align*}
We can further simplify the integral in the last line above by appealing to 
corollary \ref{connection_3}. Using this we obtain
\begin{align*}
\int \langle \frac{\partial A}{\partial t}, d^*\nabla_M^{*(k)}\nabla_M^{(k)}F_{A_t}\rangle &= 
(-1)^k \int \langle \frac{\partial A}{\partial t}, d^*\Delta_M^{(k)}F_{A_t} \rangle + 
\int \langle \frac{\partial A}{\partial t}, \sum_{v=0}^{2k-1}P_1^{(v)}[F_{A_t}]\rangle \\
&= \int \langle \frac{\partial A}{\partial t}, (-1)^kd^*\Delta_M^{(k)}F_{A_t} + \sum_{v=0}^{2k-1}P_1^{(v)}[F_{A_t}]\rangle.   
\end{align*}  
Finally, for the term $\int \frac{S}{4}|\phi|^2 + \frac{1}{8}|\phi|^4$, variations with respect to $\phi$ give
\[ \frac{\partial}{\partial t} \int \frac{S}{4}|\phi|^2 + \frac{1}{8}|\phi|^4 = 
\int \langle \frac{\partial \phi}{\partial t}, \frac{1}{4}(S + |\phi|^2)\phi \rangle. \] 

It follows that the Euler-Lagrange equations are given by
\begin{align*}
\nabla_A^{*(k+1)}\nabla_A^{(k+1)}\phi + \frac{1}{4}(S + |\phi|^2)\phi &= 0 \\
(-1)^kd^*\Delta_M^{(k)}F_A + \sum_{v=0}^{2k-1}P_1^{(v)}[F_A] + 2iIm\big{(}\sum_{i=0}^{k}C_i\nabla_M^{*(i)}\langle \nabla_{A}^{(k)}\nabla_{A}\phi, 
\nabla_{A}^{(k-i)}\phi\rangle\big{)} &= 0
\end{align*}
which proves the proposition.

\end{proof}

In view of these equations we will be studying the associated gradient flow. Given 
$(\phi(t), A(t)) \in \Gamma(\mathcal{S}^+) \times \mathfrak{A}$, we define the higher order Seiberg-Witten
gradient flow to be the following system
\begin{align}
\frac{\partial\phi}{\partial t} &= -\nabla_A^{*(k+1)}\nabla_A^{(k+1)}\phi - \frac{1}{4}(S + |\phi|^2)\phi  
\label{SW_1}\\
\frac{\partial A}{\partial t} &= (-1)^{k+1}d^*\Delta_M^{(k)}F_A - \sum_{v=0}^{2k-1}P_1^{(v)}[F_A] - 2iIm\big{(}\sum_{i=0}^{k}C_i\nabla_M^{*(i)}\langle \nabla_{A}^{(k)}\nabla_{A}\phi, 
\nabla_{A}^{(k-i)}\phi\rangle\big{)}. \label{SW_2}
\end{align}
On setting $k = 0$, we see that the above system becomes the Seiberg-Witten flow (see \cite{hong}).

We also note that during the integration by parts, carried out in the proof of proposition 
\ref{euler_lagrange}, we used the fact that 
\begin{align}
d^*\nabla_M^{*(k)}\nabla_M^{(k)}F_{A_t} = (-1)^{k}d^*\Delta_M^{(k)}F_A + \sum_{v=0}^{2k-1}P_1^{(v)}[F_A].
\label{SW_3}
\end{align}
During the course of the paper, there will be times when it is more convenient to use the term
$d^*\nabla_M^{*(k)}\nabla_M^{(k)}F_{A_t}$, and we shall do so without hesitation.

\section{Short time Existence}\label{short_time}

In this section we begin the study of short time existence of the higher order Seiberg-Witten flow. We start
by explaining why the system is not parabolic, and then move on to showing that, via a gauge fixing technique, 
solutions exist and are unique on some time interval.

The gradient flow system corresponding to the higher order Seiberg-Witten functional is not parabolic due to
the term $d^*\Delta^{(k)}_MF_A = d^*\Delta^{(k)}_MdA$.

\begin{prop}
The operator $d^*\Delta^{(k)}_Md$ is not elliptic.
\end{prop}
\begin{proof}
We recall that the Weitzenb\"{o}ck identity, proposition \ref{connection_1},  
tells us that $\Delta_M = \Delta_H + E$, where $E$ is a lower order 
derivative term, depending on the curvature. As we will be interested in computing the principal symbol of the
operator, we don't actually need to know $E$ explicitly. 

From this identity, we obtain the following
\begin{align*}
d^*\Delta_M &= d^*(d^*d + dd^*) + d^*E \\
&= (d + d^*dd^*) + d^*E \\
&= (dd^* + d^*d)d^* + d^*E \\
&= \Delta_Hd^* + d^*E \\
&= \Delta_M d^* + F
\end{align*}
where $F$ has order $2$. 

Iterating this construction, we find that
\[ d^*\Delta_M^{(k)}d = \Delta_M^{(k)}d^*d + G \]
where $G$ is a term of order $2k + 1$.

It follows that the principal symbol of $d^*\Delta_M^{(k)}d$ is equal to the principal symbol of the operator
$\Delta_M^{(k)}d^*d$. However, it is clear that $d^*d$ is not an elliptic operator, from which it immediately 
follows that $\Delta_M^{(k)}d^*d$ is not elliptic.

\end{proof}

Since the gradient system is not parabolic in order to prove the existence of a solution, with a given
initial condition, we need to follow the method of gauge fixing.

We start by adding the term $(-1)^{k}(\Delta_M^{(k)}d^*A)\phi$ to the first equation \eqref{SW_1},
and the term
$(-1)^{k+1}d\Delta_M^{(k)}d^*A$ to \eqref{SW_2}. 
We then get the new system
\begin{align}
\frac{\partial\phi}{\partial t} &= -\nabla_A^{*(k+1)}\nabla_A^{(k+1)}\phi - \frac{1}{4}(S + |\phi|^2)\phi 
+ (-1)^{k}(\Delta_M^{(k)}d^*A)\phi \label{SW_gauge_1}\\
\frac{\partial A}{\partial t} &= (-1)^{k+1}d^*\Delta_M^{(k)}F_A - \sum_{v=0}^{2k-1}P_1^{(v)}[F_A]  - 2iIm\big{(}\sum_{i=0}^{k}C_i\nabla_M^{*(i)}\langle \nabla_{A}^{(k)}\nabla_{A}\phi, 
\nabla_{A}^{(k-i)}\phi\rangle\big{)} \label{SW_gauge_2}\\
&\hspace{0.5cm}  + (-1)^{k+1}d\Delta_M^{(k)}d^*A. \nonumber
\end{align}

\begin{prop}
The above system \eqref{SW_gauge_1}-\eqref{SW_gauge_2} is parabolic
\end{prop}
\begin{proof}
\textbf{Existence:}
Observe that we can write the term $(-1)^{k+1}d^*\Delta_M^{(k)}F_A + (-1)^{k+1}d\Delta_M^{(k)}d^*A$ as
$(-1)^{k+1}\Delta_M^{(k)}d^*dA + (-1)^{k+1}\Delta_M^{(k)}dd^*A + G$, where $G$ has order $2k + 1$. 
Therefore, when computing the principal symbol we can forget about this lower order term. 
We then note that  
$(-1)^{k+1}\Delta_M^{(k)}d^*dA + (-1)^{k+1}\Delta_M^{(k)}dd^*A = (-1)^{k+1}\Delta_M^{(k)}\Delta_HA$. The operator
$(-1)^{k+1}\Delta_M^{(k)}\Delta_H$ is the highest order part in the second equation of the above system.
Using the Weitzenb\"{o}ck identity, proposition \ref{connection_1}, we see that we can write this as 
$(-1)^{k+1}\Delta^{(k+1)}_M + J$, where $J$ is a lower order term. It is clear that
$(-1)^{k+1}\Delta^{(k+1)}_M$ is elliptic, and hence ellipticity of the highest order term in the above second
equation follows.

For the first equation, we observe that the highest order term is given by 
$\nabla_A^{*(k+1)}\nabla_A^{(k+1)}$, which we can express as $\Delta_A^{(k+1)} + T$, where $T$ is a term of order
$2k + 1$. In computing the principal symbol we can forget about $T$. Furthermore, given 
$A_0$ we can write $\Delta_A^{(k+1)} = \Delta_{A_0}^{(k+1)} + T'$, where $T'$ is again a lower order term. The 
ellipticity of the operator $\nabla_A^{*(k+1)}\nabla_A^{(k+1)}$ is then an immediate consequence of these 
observations.

We thus see that the above system is a quasilinear parabolic system of order $2k + 2$. 

\end{proof}

Existence and uniqueness of 
higher order quasilinear parabolic systems (see \cite{mant}) 
then implies that, given an initial condition
$(\phi_0, A_0)$ there exists a unique solution $(\phi(t), A(t))$ to the system, on some time interval $[0, T)$, 
where $0 < T \leq \infty$.
   
We are going to use this solution, to the above parabolic system, to build a solution to the higher order 
Seiberg-Witten flow, via a gauge fixing procedure. We fix an initial condition $(\phi_0, A_0)$, and from here
on in $(\phi(t), A(t))$ will denote the unique solution to the above parabolic system with initial condition
$(\phi_0, A_0)$.

\begin{thm}
Given an initial condition $(\phi_0, A_0) \in \Gamma(\mathcal{S}^+) \times \mathfrak{A}$, there exists 
a unique solution to the higher order Seiberg-Witten flow \eqref{SW_1}-\eqref{SW_2}, on some
time interval $0 < T \leq \infty$.
\end{thm}
\begin{proof}
We start by defining the gauge we are going to be working in. Define a gauge $g(t)$ as the solution to the
following ODE
\begin{align*}
\frac{\partial}{\partial t}g(t) &= g(t)(-1)^k\Delta_M^{(k)}d^*A(t) \\
g(0) &= I. 
\end{align*} 
The term $(-1)^k\Delta_M^{(k)}d^*A(t)$ is a function on $M \times [0, T)$. Therefore, solving the above ODE gives
\[ g(t) = e^{\int_0^t(-1)^k\Delta_M^{(k)}d^*A(s)ds}. \]
We know that $A(t) \in i\Lambda^1(M)$, because $A(t)$ is unitary, which implies 
$g(t) = e^{if(t)}$, with $f(t) = \int_0^t(-1)^k\Delta^{(k)}_Md^*a(s)ds$, where we are writing 
$A(s) = ia(s)$, with $a(s)$ a real valued 1-form. This implies that the solution $g(t)$ is indeed a $U(1)$-gauge.

We then consider $(g^*\phi, g^*A)$. We are going to prove that this is a solution of the higher order 
Seiberg-Witten flow. In order to do this we are going to make use of the following formula
\[ \frac{\partial g^{-1}}{\partial t} = -g^{-2}\frac{\partial g}{\partial t} 
= -g^{-2}g(-1)^k\Delta_M^{(k)}d^*A = (-1)^{k+1}g^{-1}\Delta_M^{(k)}d^*A. \]

We will start by computing $\frac{\partial g^*\phi}{\partial t}$:
\begin{align*}
\frac{\partial g^*\phi}{\partial t} &= \frac{\partial g^{-1}\phi}{\partial t} = 
\bigg{(}\frac{\partial g^{-1}}{\partial t}\bigg{)}\phi + g^{-1}\bigg{(}\frac{\partial \phi}{\partial t}\bigg{)} \\
&= (-1)^{k+1}g^{-1}(\Delta_M^{(k)}d^*A)\phi + g^{-1}\bigg{(} -\nabla_A^{*(k+1)}\nabla_A^{(k+1)}\phi
- \frac{1}{4}\big{(}S + \vert\phi\vert^2 \big{)}\phi + (-1)^k(\Delta_M^{(k)}d^*A)\phi \bigg{)} \\
&= -g^{-1}\nabla_A^{*(k+1)}\nabla_A^{(k+1)}\phi - g^{-1}\frac{1}{4}\big{(}S + \vert\phi\vert^2 \big{)}\phi \\
&= -\nabla_{g^{*}A}^{*(k+1)}\nabla_{g^{*}A}^{(k+1)}\phi - \frac{1}{4}\big{(}S + \vert g^*\phi\vert^2 \big{)}g^*\phi.
\end{align*}

We move on to computing $\frac{\partial g^*A}{\partial t}$. 
\begin{align*}
\frac{\partial g^*A}{\partial t} &= \frac{\partial }{\partial t}(A + g^{-1}dg) = \frac{\partial A}{\partial t}
+ \frac{\partial g^{-1}}{\partial t}dg + g^{-1}d\bigg{(}\frac{\partial g}{\partial t}\bigg{)} \\
&= -d^*\nabla_M^*{(k)}\nabla_M^{(k)}F_A - 2iIm\big{(}\sum_{i=0}^{k}C_i\nabla_M^{*(i)}\langle 
\nabla_{A}^{(k)}\nabla_{A}\phi, \nabla_{A}^{(k-i)}\phi\rangle\big{)} + (-1)^{k+1}d\Delta_M^{(k)}d^*A \\
&+ (-1)^{k+1}g^{-1}(\Delta^{(k)}d^*A)dg + g^{-1}\bigg{(} (-1)^k(dg)(\Delta_M^{(k)}d^*A) + 
(-1)^kgd\Delta_M^{(k)}d^*A \bigg{)} \\
&= -d^*\nabla_M^{*(k)}\nabla_M^{(k)}F_{g^*A} - 2iIm\big{(}\sum_{i=0}^{k}C_i\nabla_M^{*(i)}\langle 
\nabla_{g^*A}^{(k)}\nabla_{g^*A}\phi, \nabla_{g^*A}^{(k-i)}\phi\rangle\big{)}.
\end{align*}

It follows that $(g^*\phi, g^*A)$ is a solution to the higher order Seiberg-Witten flow with initial
condition $(g^*(0)\phi(0), g^*(0)A(0)) = (\phi_0, A_0)$, using the fact that $g(0) = I$. This proves
existence.

\textbf{Uniqueness:}
To see that solutions are unique,
observe that given a solution $(\phi, A)$ of the higher order Seiberg-Witten flow, with initial
condition $(\phi_0, A_0)$, we can then construct a gauge $g(t)$ as we did above. If we then
consider $((g^{-1})^*\phi, (g^{-1})^*A)$, then a simple computation shows that this solves the parabolic system 
\eqref{SW_gauge_1}-\eqref{SW_gauge_2}, with initial condition $(\phi_0, A_0)$.

This means that if we had two solutions to the higher order Seiberg-Witten flow, $(\phi_1, A_1)$ and
$(\phi_2, A_2)$, such that $(\phi_1(0), A_1(0)) = (\phi_2(0), A_2(0)) = (\phi_0, A_0)$. Then
we find that $((g^{-1})^*\phi_1, (g^{-1})^*A_1)$ and $((g^{-1})^*\phi_2, (g^{-1})^*A_2)$ both solve the
parabolic system \eqref{SW_gauge_1}-\eqref{SW_gauge_2}, with the same initial condition
$(\phi_0, A_0)$. Uniqueness of this system then gives    
$((g^{-1})^*\phi_1, (g^{-1})^*A_1) = ((g^{-1})^*\phi_2, (g^{-1})^*A_2)$. Applying $g^*$ to this equation, and using
the fact that $(g^*) \circ (g^{-1})^* = I$, it follows that 
$(\phi_1, A_1) = (\phi_2, A_2)$, and uniqueness is established.

\end{proof}

\section{Energy estimates}\label{energy_estimates}

In this section we derive energy estimates for solutions of the higher order Seiberg-Witten flow. These
estimates will then be used in the study of long time existence in section \ref{long_time}.

We start by showing that the spinor field does not blow up along the flow as you approach the maximal time.

\begin{prop}\label{spinor_bounded_along_flow}
Given a solution $(\phi_t, A_t)$ to the higher order Seiberg-Witten flow on some time interval $[0, T)$, where
$T \leq \infty$. We have that $\sup_{t \in [0, T)}\vert\phi_t\vert < \infty$.
\end{prop}

\begin{proof}
We compute
\begin{align*}
\frac{\partial}{\partial t}\langle \phi, \phi\rangle &= \langle \frac{\partial\phi}{\partial t}, \phi\rangle +
\langle \phi, \frac{\partial\phi}{\partial t}\rangle \\
&= \langle -\nabla_A^{*(k+1)}\nabla_A^{(k+1)}\phi - \frac{1}{4}(S + \vert\phi\vert^2)\phi, \phi \rangle + 
\langle \phi, -\nabla_A^{*(k+1)}\nabla_A^{(k+1)}\phi - \frac{1}{4}(S + \vert\phi\vert^2)\phi \rangle \\
&= -2\langle \nabla_A^{(k+1)}\phi, \nabla_A^{(k+1)}\phi\rangle - 
\frac{1}{2}(S + \vert\phi\vert^2)\langle \phi, \phi\rangle \\
&= -2\vert \nabla_A^{(k+1)}\phi\vert^2 - \frac{1}{2}(S + \vert\phi\vert^2)\vert \phi\vert^2.
\end{align*}
Let $S_0 = min\{S(x) : x \in M\}$, and
choose $0 < \epsilon << 1$. Suppose there exists $(x, t)$ such that 
$\vert\phi(x, t)\vert \geq \sqrt{\vert S_0\vert} + \epsilon$. Let $t_0$ be the first time when this happens, so that
there exists $(x_0, t_0)$ such that $\vert\phi(x_0, t_0)\vert \geq \sqrt{\vert S_0\vert} + \epsilon$. Without loss
of generality we assume $t_0 > 0$, for if $t_0 = 0$, then replace $\epsilon$ with $2\epsilon$ and consider
$\sqrt{\vert S_0\vert} + 2\epsilon$ instead.

Therefore assuming $t_0 > 0$, we get, by the continuity of $\phi$, that 
$\vert\phi(x_0, t_0)\vert = \sqrt{\vert S_0\vert} + \epsilon$. By continuity, we also know that there exists 
an interval $(t_1, t_2)$ such that $t_0 \in (t_1, t_2)$ and $\vert\phi(x_0, t)\vert > \sqrt{\vert S_0\vert}$, for 
all $t \in (t_1, t_2)$.

Then for any such $t \in (t_1, t_2)$, we have  
$\vert\phi(x_0, t)\vert^2 + S \geq \vert\phi(x_0, t)\vert^2 + S_0 \geq 0$. This in turn implies that 
\[(\vert\phi(x_0, t)\vert^2 + S)\vert\phi(x_0, t)\vert^2 \geq 0, \forall t \in (t_1, t_2).\]
Substituting this into the formula obtained for $\frac{\partial}{\partial t}\langle \phi, \phi\rangle$ at the
start of this proof, we find that
\[ \frac{\partial}{\partial t}\vert \phi(x_0, t)\vert^2 \leq 0, \forall t \in (t_1, t_2).\]
This implies that $\vert \phi(x_0, t)\vert^2$ is a non-increasing function for $t \in (t_1, t_2)$.
In particular, this implies that
\[ \vert \phi(x_0, t)\vert \geq \vert\phi(x_0, t_0)\vert = \sqrt{\vert S_0\vert} + \epsilon,  
 \forall t \in (t_1, t_0).\]
However, this contradicts the fact that $t_0$ was the first time such that 
$\vert \phi(x, t)\vert \geq \sqrt{\vert S_0\vert} + \epsilon$. It follows that no such time $t_0$ exists, and that
in fact we have that 
\[\vert \phi(x, t)\vert \leq  \sqrt{\vert S_0\vert} + \epsilon, \forall t \]
which in turn implies that $\sup_{t \in [0, T)}\vert\phi_t\vert < \infty$.

\end{proof}

\begin{lem}\label{energy-k}
\begin{equation*}
\frac{\partial}{\partial t}\mathcal{SW}^k(\phi(t), A(t)) = -2 \big{(}\vert\vert\frac{\partial \phi}{\partial t}(t)\vert\vert_{L^2}^2 +
\vert\vert\frac{\partial A}{\partial t}(t)\vert\vert_{L^2}^2\big{)} \leq 0. 
\end{equation*}
In particular, the higher order Seiberg-Witten energy remains bounded along the flow.
\end{lem}
\begin{proof}
For $\tau > 0$, we can compute the above derivative as follows
\[\frac{\partial}{\partial t}\mathcal{SW}^k(\phi(t), A(t)) 
= \frac{\partial}{\partial \tau}\bigg{\vert}_{\tau = 0}\mathcal{SW}^k(\phi(t) + \tau\frac{\partial \phi}{\partial t}, A(t)) +
\frac{\partial}{\partial \tau}\bigg{\vert}_{\tau = 0}\mathcal{SW}^k(\phi(t), A(t) + \tau\frac{\partial A}{\partial t}).\]
We start by computing $\frac{\partial}{\partial \tau}\big{\vert}_{\tau = 0}\mathcal{SW}^k(\phi(t) + 
\tau\frac{\partial \phi}{\partial t}, A(t))$. We can write this derivative as
\begin{align*}
 \frac{\partial}{\partial \tau}\bigg{\vert}_{\tau = 0} \int 
\langle \nabla_A^{(k+1)}(\phi(t) + \tau\frac{\partial \phi}{\partial t}), \nabla_A^{(k+1)}(\phi(t) + \tau\frac{\partial \phi}{\partial t})\rangle &+ 
\vert \nabla_M^{(k)}F_A\vert^2 + \frac{S}{4}\langle \phi(t) + \tau\frac{\partial \phi}{\partial t}, \phi(t) + \tau\frac{\partial \phi}{\partial t}\rangle \\
&+ \frac{1}{8}\langle \phi(t) + \tau\frac{\partial \phi}{\partial t}, \phi(t) + \tau\frac{\partial \phi}{\partial t}\rangle^2.  
\end{align*}
Getting rid of the terms that don't involve $\tau$, we can express the above as
\begin{align*}
\frac{\partial}{\partial \tau}\bigg{\vert}_{\tau = 0} \int 
\tau \langle \nabla_A^{(k+1)}\frac{\partial \phi}{\partial t}, \nabla_A^{(k+1)}\phi(t)\rangle &+ 
\tau \langle \nabla_A^{(k+1)}\phi(t), \nabla_A^{(k+1)}\frac{\partial \phi}{\partial t}\rangle + 
\frac{S}{4}\tau\langle\frac{\partial \phi}{\partial t}, \phi(t)\rangle + 
\frac{S}{4}\tau\langle\phi(t), \frac{\partial \phi}{\partial t}\rangle \\
&+ \frac{1}{8}\big{(}\vert \phi(t)\vert^2 + \tau\langle\phi, \frac{\partial \phi}{\partial t}\rangle + 
\tau\langle\frac{\partial \phi}{\partial t}, \phi(t)\rangle + \tau^2\langle\frac{\partial \phi}{\partial t}, \frac{\partial \phi}{\partial t}\rangle\big{)}^2.
\end{align*}
Computing the above derivative we obtain
\begin{align*}
&\hspace{0.5cm}
\int \langle \frac{\partial \phi}{\partial t}(t), \nabla_A^{*(k+1)}\nabla_A^{(k+1)}\phi + 
\big{(}\frac{S}{4} + \frac{\vert\phi\vert^2}{4}\big{)}\phi\rangle + 
\langle \nabla_A^{*(k+1)}\nabla_A^{(k+1)}\phi + 
\big{(}\frac{S}{4} + \frac{\vert\phi\vert^2}{4}\big{)}\phi, \frac{\partial \phi}{\partial t}(t)\rangle \\
&= \int \langle \frac{\partial \phi}{\partial t}(t), -\frac{\partial \phi}{\partial t}(t)\rangle + \langle -\frac{\partial \phi}{\partial t}(t), \frac{\partial \phi}{\partial t}(t)\rangle \\
&= -2\vert\vert \frac{\partial \phi}{\partial t}(t)\vert\vert^2_{L^2}.
\end{align*}
A similar computation proves that 
\[\frac{\partial}{\partial \tau}\bigg{\vert}_{\tau = 0}\mathcal{SW}^k(\phi(t), A(t) + \tau\frac{\partial A}{\partial t}) = 
-2\vert\vert \frac{\partial A}{\partial t}(t)\vert\vert^2_{L^2}, \]
which gives the statement of the lemma.

\end{proof}

Recall that the Seiberg-Witten functional is defined as 
\begin{equation*}
\mathcal{SW}(\phi, A) = \int (|\nabla_A\phi|^2 + |F_A|^2 + \frac{S}{4}|\phi|^2 + \frac{1}{8}|
\phi^4)d\mu  +  \pi^2c_1(\mathcal{L}^2). 
\end{equation*}

In the previous lemma, we saw how the higher order Seiberg-Witten energy decreased along the flow, and therefore
we could conclude that it remains bounded in time. The following lemma proves that given a solution to the
higher order Seiberg-Witten flow for finite time $T < \infty$, its Seiberg-Witten energy is also bounded along the 
flow.  

\begin{lem}\label{SW_finite_energy}
Let $(\phi(t), A(t))$ be a solution to the higher order Seiberg-Witten flow, on $[0, T)$ for $T < \infty$. Then
the Seiberg-Witten energy
\begin{equation*}
\mathcal{SW}(\phi, A) = \int (|\nabla_A\phi|^2 + |F_A|^2 + \frac{S}{4}|\phi|^2 + \frac{1}{8}|
\phi^4)d\mu  +  \pi^2c_1(\mathcal{L}^2)
\end{equation*}
is bounded along the flow. That is $\sup_{t\in [0, T)} \mathcal{SW}(\phi_t, A_t) < \infty$.
\end{lem}
\begin{proof}
We start by computing
\begin{align*}
 \frac{\partial}{\partial t}\mathcal{SW}(\phi_t, A_t) 
= \int \langle\frac{\partial\phi}{\partial t}, \nabla_A^{*}\nabla_A\phi\rangle &+ 
\langle \nabla_A^{*}\nabla_A\phi, \frac{\partial\phi}{\partial t}\rangle +
\big{(}\frac{S}{4} + \frac{|\phi|^2}{4}\big{)}\big{(} \langle\frac{\partial\phi}{\partial t}, \phi\rangle
+ \langle\phi, \frac{\partial\phi}{\partial t}\rangle\big{)} \\ 
&+ 2\langle d\frac{\partial A}{\partial t}, F_A\rangle 
+ \langle\frac{\partial A}{\partial t} \otimes \phi, 
\nabla_A\phi\rangle + \langle \nabla_A\phi, \frac{\partial A}{\partial t} \otimes \phi\rangle.   
\end{align*}

We now explain how we can bound the quantity on the right. In doing so we will need to define the following
constant $C := max\{1, sup_{M \times [0,T)}\{S/4 + |\phi|^2/4\}\}$. We then have
\begin{align*}
&\hspace{0.5cm}
\int \langle\frac{\partial\phi}{\partial t}, \nabla_A^{*}\nabla_A\phi\rangle + 
\langle \nabla_A^{*}\nabla_A\phi, \frac{\partial\phi}{\partial t}\rangle +
\big{(}\frac{S}{4} + \frac{|\phi|^2}{4}\big{)}\big{(} \langle\frac{\partial\phi}{\partial t}, \phi\rangle
 + \langle\phi, \frac{\partial\phi}{\partial t}\rangle\big{)} 
+ 2\langle d\frac{\partial A}{\partial t}, F_A\rangle \\
&\hspace{1cm}+ \langle\frac{\partial A}{\partial t} \otimes \phi, 
\nabla_A\phi\rangle + \langle \nabla_A\phi, \frac{\partial A}{\partial t} \otimes \phi\rangle \\
&\leq  \int C\langle \frac{\partial\phi}{\partial t}, \nabla_A^{*}\nabla_A\phi + \phi\rangle +
C\langle\nabla_A^{*}\nabla_A\phi + \phi, \frac{\partial\phi}{\partial t}\rangle + 
2\langle \frac{\partial A}{\partial t}, d^*F_A\rangle +  \langle\frac{\partial A}{\partial t} \otimes \phi, 
\nabla_A\phi\rangle \\
&\hspace{1cm}+ \langle \nabla_A\phi, \frac{\partial A}{\partial t} \otimes \phi\rangle \\
&\leq C \int 2|\langle \frac{\partial\phi}{\partial t}, \nabla_A^{*}\nabla_A\phi + \phi\rangle| + 
2|\langle \frac{\partial A}{\partial t}, d^*F_A\rangle| + 2|\langle\frac{\partial A}{\partial t} \otimes \phi, 
\nabla_A\phi\rangle| \\
&\leq 2C \int |\langle \frac{\partial\phi}{\partial t}, \nabla_A^{*}\nabla_A\phi\rangle| + 
|\langle\frac{\partial\phi}{\partial t},  \phi\rangle| +
|\langle \frac{\partial A}{\partial t}, d^*F_A\rangle| + |\langle\frac{\partial A}{\partial t} \otimes \phi, 
\nabla_A\phi\rangle|.
 \end{align*}
On appealing to Young's inequality, we can further bound the right hand side of this last inequality as follows.
\begin{align*}
&\hspace{0.5cm}
2C \int |\langle \frac{\partial\phi}{\partial t}, \nabla_A^{*}\nabla_A\phi\rangle| + 
|\langle\frac{\partial\phi}{\partial t},  \phi\rangle| +
|\langle \frac{\partial A}{\partial t}, d^*F_A\rangle| + |\langle\frac{\partial A}{\partial t} \otimes \phi, 
\nabla_A\phi\rangle| \\
&\leq 2C \big{(}2 \vert\vert\frac{\partial\phi}{\partial t}\vert\vert^2_{L^2} 
+ C_1(g)\vert\vert \nabla_A\nabla_A\phi\vert\vert^2_{L^2} + \vert\vert\phi\vert\vert^2_{L^2} + 
 \vert\vert\frac{\partial A}{\partial t}\vert\vert^2_{L^2} + C_2(g)\vert\vert\nabla_MF_A\vert\vert^2_{L^2} \\
&\hspace{1cm} 
 + \vert\vert\phi\vert\vert_{\infty}\vert\vert\frac{\partial A}{\partial t}\vert\vert^2_{L^2} 
  + \vert\vert\nabla_A\phi\vert\vert^2_{L^2}\big{)} \\
& \leq C(g, \phi)\bigg{(} \vert\vert\frac{\partial\phi}{\partial t}\vert\vert^2_{L^2} 
 + \vert\vert\frac{\partial A}{\partial t}\vert\vert^2_{L^2} + \vert\vert \nabla_A\nabla_A\phi\vert\vert^2_{L^2} +
\vert\vert\nabla_MF_A\vert\vert^2_{L^2} + \vert\vert\phi\vert\vert^2_{L^2} +
\vert\vert\nabla_A\phi\vert\vert^2_{L^2} \bigg{)}
\end{align*}
where the constant $C(g, \phi)$ depends on $\phi$ through $\vert\vert\phi\vert\vert_{\infty}$, which we know is
bounded along the flow by proposition \ref{spinor_bounded_along_flow}.

Applying the energy estimate, lemma \ref{energy-k}, we can write this last quantity as
\[ C(g, \phi) \bigg{(} -\frac{\partial}{\partial t}\mathcal{SW}^{k}(\phi_t, A_t) + 
\vert\vert \nabla_A\nabla_A\phi\vert\vert^2_{L^2} +
\vert\vert\nabla_MF_A\vert\vert^2_{L^2} + \vert\vert\phi\vert\vert^2_{L^2} +
\vert\vert\nabla_A\phi\vert\vert^2_{L^2} \bigg{)}. \]
In order to estimate this quantity we are going to apply the interpolation inequality, lemma \ref{interp_3}. 
Let $\epsilon_1 > \epsilon_2 > 0$, we then have
\begin{align*}
&\hspace{0.5cm} C(g, \phi) \bigg{(} -\frac{\partial}{\partial t}\mathcal{SW}^{k}(\phi_t, A_t) + 
\vert\vert \nabla_A\nabla_A\phi\vert\vert^2_{L^2} +
\vert\vert\nabla_MF_A\vert\vert^2_{L^2} + \vert\vert\phi\vert\vert^2_{L^2} +
\vert\vert\nabla_A\phi\vert\vert^2_{L^2} \bigg{)} \\
& \leq C(g, \phi) \bigg{(} -\frac{\partial}{\partial t}\mathcal{SW}^{k}(\phi_t, A_t) + 
C(\epsilon_1)\vert\vert \nabla_A^{(k+1)}\phi\vert\vert^2_{L^2} + 
\epsilon_1\vert\vert\nabla_A\phi\vert\vert^2_{L^2} + C(\epsilon_2)\vert\vert\nabla_M^{(k)}F_A\vert\vert^2_{L^2} +
\epsilon_2\vert\vert F_A\vert\vert^2_{L^2} \\
&\hspace{2cm}+ \vert\vert\nabla_A\phi\vert\vert^2_{L^2} + 
\vert\vert\phi\vert\vert^2_{L^2}\bigg{)}.
\end{align*}
Therefore for any $t < T$, we have that
\begin{align*}
\mathcal{SW}(\phi_t, A_t) - \mathcal{SW}(\phi_0, A_0) 
&\leq C_1(g, \phi)(\mathcal{SW}^{k}(\phi_0, A_0) - \mathcal{SW}^{k}(\phi_t, A_t)) \\
&+  C_2(g, \phi, \epsilon_1, \epsilon_2) \int_{0}^t \bigg{(} \vert\vert \nabla_A^{(k+1)}\phi\vert\vert^2_{L^2}  + 
\vert\vert\nabla_M^{(k)}F_A\vert\vert^2_{L^2} +  \int\frac{S}{4}\vert\phi\vert^2 
+ \frac{1}{8}\vert\vert\phi\vert\vert^4_{L^2}\bigg{)} \\
&+ C(g, \phi)\epsilon_1 \int_0^t \bigg{(} \vert\vert F_A\vert\vert^2_{L^2} + 
\vert\vert\nabla_A\phi\vert\vert^2_{L^2} + \int\frac{S}{4}\vert\phi\vert^2 +
\frac{1}{8}\vert\vert\phi\vert\vert^4_{L^2}\bigg{)} + C(\phi)
\end{align*}
where the constant $C(\phi)$ comes from the fact that we added in the terms $\int\frac{S}{4}\vert\phi\vert^2$
and $\frac{1}{8}\vert\vert\phi\vert\vert^4_{L^2}$, remembering that these quantities are bounded along the flow.

We can rewrite this latter quantity as
\begin{align*}
 C_1(g, \phi)(\mathcal{SW}^{k}(\phi_0, A_0) - \mathcal{SW}^{k}(\phi_t, A_t)) &+ 
   C_2(g, \phi, \epsilon_1, \epsilon_2)\int_0^t\mathcal{SW}^k(\phi_s, A_s)ds \\
   &+    C(g, \phi)\epsilon_1\int_0^t\mathcal{SW}(\phi_s, A_s)ds + C(\phi).
\end{align*}
Using the fact that the higher order Seiberg-Witten energy decreases along the flow, we can estimate this quantity
as follows.
\begin{align*}
&\hspace{0.5cm} 
C_1(g, \phi)(\mathcal{SW}^{k}(\phi_0, A_0) - \mathcal{SW}^{k}(\phi_t, A_t)) + 
   C_2(g, \phi, \epsilon_1, \epsilon_2)\int_0^t\mathcal{SW}^k(\phi_s, A_s)ds \\
&\hspace{1cm}+    C(g, \phi)\epsilon_1\int_0^t\mathcal{SW}(\phi_s, A_s)ds + C(\phi) \\
& \leq (C_1(g, \phi) +  tC_2(g, \phi, \epsilon_1, \epsilon_2))\mathcal{SW}^{k}(\phi_0, A_0) -  
C_1(g, \phi)\mathcal{SW}^{k}(\phi_t, A_t) + tC(g, \phi)\epsilon_1\sup_{s \in [0,t]}\mathcal{SW}(\phi_s, A_s) 
+  C(\phi) \\
& \leq C_3(g, \phi, T)\mathcal{SW}^{k}(\phi_0, A_0) + tC(g, \phi)\epsilon_1\sup_{s \in [0,t]}\mathcal{SW}
(\phi_s,A_s)  
\end{align*}
where the constant $C_3(g, \phi, T)$ comes from using $t < T$, and absorbing $C(\phi)$ into 
$(C_1(g, \phi) +  tC_2(g, \phi, \epsilon_1, \epsilon_2))$.

In particular, by taking $\epsilon_1 = \epsilon/tC(g, \phi)$, we get the following inequality
\[\mathcal{SW}(\phi_t, A_t) - \mathcal{SW}(\phi_0, A_0) \leq C_3(g, \phi, T)\mathcal{SW}^{k}(\phi_0, A_0) + 
\epsilon\sup_{s \in [0,t]}\mathcal{SW}(\phi_s,A_s).  \]
This implies
\begin{equation}\label{eqn:energy inequality}
\mathcal{SW}(\phi_t, A_t) - \epsilon\sup_{s \in [0,t]}\mathcal{SW}(\phi_s,A_s) - \mathcal{SW}(\phi_0, A_0) \leq   
 C_3(g, \phi, T)\mathcal{SW}^{k}(\phi_0, A_0).
 \end{equation}

Suppose there exists $t_m \rightarrow T$ such that 
$\lim_{m \rightarrow \infty} \mathcal{SW}(\phi_{t_m}, A_{t_m}) \rightarrow \infty$. By throwing out some of the
$t_m$ we can assume that $\mathcal{SW}(\phi_{t_m}, A_{t_m}) \geq \mathcal{SW}(\phi_{t_n}, A_{t_n})$ for
$m \geq n$, and that $t_m \geq t_n$, when $m \geq n$. 

Partition $[0, T)$ in the following way, $[0, T) = [t_0, t_1]\cup [t_1, t_2] \cup \ldots [t_k, t_{k+1}]\ldots$, 
where $t_0 = 0$. 
Then define $s_i \in [t_i, t_{i+1}]$ by 
$\sup_{t\in [t_i, t_{i+1}]}\mathcal{SW}(\phi_t, A_t) = \mathcal{SW}(\phi_{s_i}, A_{s_i})$. It is easy to see that
$s_i \rightarrow T$, and $\mathcal{SW}(\phi_{s_i}, A_{s_i}) \rightarrow \infty$ as $i \rightarrow \infty$.
Furthermore, we also have that $\mathcal{SW}(\phi_{s_j}, A_{s_j}) \leq \mathcal{SW}(\phi_{s_i}, A_{s_i})$ when
$j \leq i$.

We now substitute $s_i$ for $t$ in the above inequality (\ref{eqn:energy inequality}) to obtain
\begin{equation*}
\mathcal{SW}(\phi_{s_i}, A_{s_i}) - \epsilon\mathcal{SW}(\phi_{s_i},A_{s_i}) - \mathcal{SW}(\phi_0, A_0) \leq   
 C_3(g, \phi, T)\mathcal{SW}^{k}(\phi_0, A_0)
 \end{equation*}
from which we obtain
\begin{equation*}
\mathcal{SW}(\phi_{s_i}, A_{s_i}) \leq   
 \frac{1}{1-\epsilon}\big{(}C_3(g, \phi, T)\mathcal{SW}^{k}(\phi_0, A_0) +  \mathcal{SW}(\phi_0, A_0)\big{)}.
 \end{equation*}
The right hand side of the above equation is finite, and independent of $i$. Therefore, 
taking $i \rightarrow \infty$ 
on the left, we contradict the fact that the left hand side should approach $\infty$. It follows that no such
$\{t_m\}$ exists, and that $\sup_{t\in [0, T)} \mathcal{SW}(\phi_t, A_t) < \infty$. 

\end{proof}

\section{Local $L^2$-derivative estimates}\label{local_estimates}

In this section we prove local $L^2$-derivative estimates for solutions of the higher order Seiberg-Witten flow.
As the system \eqref{SW_1}-\eqref{SW_2} is a higher order system, one cannot appeal to the use of maximum 
principles and Harnack inequalities to study such systems. It is in this regard that the obtaining of 
local derivative estimates become a vital tool for the study of such higher order systems. We shall then
put these derivative estimates to use when we consider questions of long time existence.

\subsection{Bump functions}

In the course of obtaining local $L^2$-derivative estimates, we will need to make use of bump functions. In this
brief subsection, we outline the notation we use and prove a simple lemma that will be used in our estimates in the
subsections to come.

\begin{defn}
Given $\gamma \in C_c^{\infty}(M)$, we say $\gamma$ is a bump function if $0 \leq \gamma \leq 1$.
\end{defn}
In this paper, we will always use the notation $\gamma$ to denote such a bump function.

\begin{lem}\label{bump}
Let $\gamma$ be a bump function. Fix $i \in \mathbb{N}$, and let $s$ be a positive real number such that
$s \geq i$. We then have
\begin{equation*}
\nabla^{(i)}\gamma^{s} = \sum_{\substack{n_1+\cdots +n_i = i\\0\leq n_1\leq \cdots\leq n_i\leq i}}
C_{(n_1,\ldots,n_i)}(\gamma, s)\gamma^{s-i}\nabla^{n_1}\gamma*\cdots*\nabla^{n_i}\gamma. 
\end{equation*}
\end{lem}
\begin{proof}
One simply has to compute derivatives. First observe that $\nabla(\gamma^s) = s\gamma^{s-1}\nabla\gamma$, and
\begin{align*}
\nabla^{(2)}(\gamma^s) = \nabla(s\gamma^{s-1}\nabla\gamma) &= \nabla(s\gamma^{s-1})\otimes\nabla\gamma + 
s\gamma^{s-1}\nabla^{(2)}\gamma \\
&= s(s-1)\gamma^{s-2}\nabla\gamma\otimes\nabla\gamma + s\gamma^{s-1}\nabla^{(2)}\gamma \\
&= s(s-1)\gamma^{s-2}\nabla\gamma\otimes\nabla\gamma + s\gamma\gamma^{s-2}\nabla^{(2)}\gamma.
\end{align*}
Continuing to take derivatives, we see that we can write
\begin{align*}
\nabla^{(i)}(\gamma^s) =
\sum_{n_1+\cdots +n_i = i}
\tilde{C}_{(n_1,\ldots,n_i)}(\gamma, s)\gamma^{s-i}\nabla^{n_1}\gamma\otimes\cdots\otimes\nabla^{n_i}\gamma. 
\end{align*}
By swapping some products, and collecting like terms, it is then easy to see that we can write
\begin{align*}
\nabla^{(i)}\gamma^{s} = \sum_{\substack{n_1+\cdots +n_i = i\\0\leq n_1\leq \cdots\leq n_i\leq i}}
C_{(n_1,\ldots,n_i)}(\gamma, s)\gamma^{s-i}\nabla^{n_1}\gamma*\cdots*\nabla^{n_i}\gamma.
\end{align*}
\end{proof}

In the subsections to come, we will obtain local $L^2$-derivative estimates for the spinor field
and curvature form. During these estimates we will obtain constants that will depend on a fixed bump
function $\gamma$, and its derivatives. We will denote such a constant by $C(\gamma)$, with the understanding
that $C(\gamma)$ may be depending on derivatives of $\gamma$ as well.

\subsection{Evolution equations}

We start by computing the evolution equations satisfied by the spinor field and the 
curvature form under the flow.

\begin{lem}
Let $(\phi(t), A(t))$ be a solution to the higher order Seiberg-Witten flow. Then
\[ \frac{\partial F_{A(t)}}{\partial t} = (-1)^{k+1}\Delta_M^{(k+1)}F_{A(t)} + \sum_{v = 0}^{2k}P_1[F_{A(t)}] 
-2iIm\bigg{(} \sum_{i=1}^kC_id\nabla_M^{*(i)}\langle\nabla_A^{(k)}\nabla_A\phi, \nabla_A^{(k-i)}\phi
\rangle \bigg{)}. \]
\end{lem}
\begin{proof}
We have that $\frac{\partial F_{A(t)}}{\partial t} = \frac{\partial d{A(t)}}{\partial t}
= d\frac{\partial A(t)}{\partial t}$. As $A$ satisfies the higher order Seiberg-Witten flow, we obtain
\begin{align*}
\frac{\partial F_{A(t)}}{\partial t} &= 
d\bigg{(} (-1)^{k+1}d^*\Delta^{(k)}_MF_A - \sum_{v=0}^{2k-1}P_1^{(v)}[F_A] - 
2iIm\big{(}\sum_{i=0}^{k}C_i\nabla_M^{*(i)}\langle \nabla_{A}^{(k)}\nabla_{A}\phi, 
\nabla_{A}^{(k-i)}\phi\rangle\big{)} \bigg{)} \\
&= (-1)^{k+1}dd^*\Delta^{(k)}_MF_A - \sum_{v=0}^{2k}P_1^{(v)}[F_A] - 
2iIm\big{(}\sum_{i=0}^{k}C_id\nabla_M^{*(i)}\langle \nabla_{A}^{(k)}\nabla_{A}\phi, 
\nabla_{A}^{(k-i)}\phi\rangle\big{)} \\
&= (-1)^{k+1}\Delta^{(k+1)}_MF_A + \sum_{v=0}^{2k}P_1^{(v)}[F_A] - 
2iIm\big{(}\sum_{i=0}^{k}C_id\nabla_M^{*(i)}\langle \nabla_{A}^{(k)}\nabla_{A}\phi, 
\nabla_{A}^{(k-i)}\phi\rangle\big{)}
\end{align*}
where to obtain the last equality we have used proposition \ref{connection_1}, 
and have absorbed the extra lower order derivative terms, arising from this formula, into the quantity 
$\sum_{v=0}^{2k}P_1^{(v)}[F_A]$.
\end{proof}

\begin{cor}\label{derivative_curvature}
Let $(\phi(t), A(t))$ be a solution to the higher order Seiberg-Witten flow. Then
\[ \frac{\partial}{\partial t}\nabla_M^{(l)}F_{A(t)} = (-1)^{k+1}\Delta_M^{k+1}\nabla_M^{(l)}F_{A(t)}
+ \sum_{v = 0}^{2k+l}P_1[F_{A(t)}] - 
2iIm\bigg{(} \sum_{i=1}^kC_i\nabla_M^{(l)}d\nabla_M^{*(i)}\langle\nabla_A^{(k)}\nabla_A\phi, 
\nabla_A^{(k-i)}\phi
\rangle \bigg{)}.\]
\end{cor}
\begin{proof}
This follows by using the above lemma
\begin{align*}
\frac{\partial}{\partial t}\nabla_M^{(l)}F_{A(t)} &= \nabla_M^{(l)}\frac{\partial F_A}{\partial t} \\
&= \nabla_M^{(l)}\bigg{(} (-1)^{k+1}\Delta_M^{(k+1)}F_{A(t)} + \sum_{v = 0}^{2k}P_1[F_{A(t)}] 
-2iIm\big{(} \sum_{i=1}^kC_id\nabla_M^{*(i)}\langle\nabla_A^{(k)}\nabla_A\phi, \nabla_A^{(k-i)}\phi
\rangle  \bigg{)} \\
&= (-1)^{k+1}\Delta_M^{k+1}\nabla_M^{(l)}F_{A(t)}
+ \sum_{v = 0}^{2k+l}P_1[F_{A(t)}] - 
2iIm\bigg{(} \sum_{i=1}^kC_i\nabla_M^{(l)}d\nabla_M^{*(i)}\langle\nabla_A^{(k)}\nabla_A\phi, 
\nabla_A^{(k-i)}\phi
\rangle \bigg{)}
\end{align*}
where to obtain the last equality we have used proposition \ref{connection_1}, 
and have absorbed the extra lower order derivative terms, arising from this formula, into the quantity 
$\sum_{v=0}^{2k+l}P_1^{(v)}[F_A]$.

\end{proof}

\begin{lem}\label{derivative_spinor}
Let $(\phi(t), A(t))$ be a solution to the higher order Seiberg-Witten flow. Then
\begin{align*}
\frac{\partial}{\partial t}\nabla_A^{(l)}\phi = -\Delta_A^{(k+1)}\nabla_A^{(l)}\phi &+ 
\sum_{j=0}^{2k-2+l}\nabla^{(j)}_MRm * \nabla_A^{(2k-2+l-j)}\phi  +
\sum_{j=0}^{2k+l}\nabla^{(j)}_MRm * \nabla_A^{(2k+l-j)}\phi \\
&+ \sum_{j=0}^{2k-2+l}\nabla^{(j)}_MF_A * \nabla_A^{(2k-2+l-j)}\phi +
\sum_{j=0}^{2k+l}\nabla^{(j)}_MF_A * \nabla_A^{(2k+l-j)}\phi \\
&+ -\frac{1}{4}\nabla_A^{(l)}\big{(}(S + \vert\phi\vert^2)\phi\big{)} + 
\sum_{i=0}^{l-1}(-1)^{k+1}C_i\nabla_M^{(i)}d^*\Delta_M^{(k)}F_A \otimes \nabla_A^{(l-1-i)}\phi \\
&+  \sum_{i=0}^{l-1}\sum_{v=0}^{2k-1+i}P_1^{(v)}[F_A] \otimes \nabla_A^{(l-1-i)}\phi \\
&- 2iIm\bigg{(} \sum_{j=0}^{l-1}\sum_{i=1}^kC_i\nabla_M^{(j)}\nabla_M^{*(i)}\langle\nabla_A^{(k)}\nabla_A\phi, 
\nabla_A^{(k-i)}\phi
\rangle \bigg{)} \otimes \nabla_A^{(l-1-j)}\phi.
\end{align*}
\end{lem}
\begin{proof}
From lemma \ref{variation}, we have that
$\frac{\partial}{\partial t}\nabla_A^{(l)}\phi = \nabla_A^{(l)}\frac{\partial \phi}{\partial t} + 
\sum_{i=0}^{l-1} C_i\nabla^{(i)} \otimes \nabla_A^{(l-1-i)}\phi$. Since 
$(\phi(t), A(t))$ is a solution to the higher order Seiberg-Witten flow we find
\begin{align*}
\frac{\partial}{\partial t}\nabla_A^{(l)}\phi &= \nabla_A^{(l)}\frac{\partial \phi}{\partial t} + 
\sum_{i=0}^{l-1} C_i\nabla^{(i)}\dot{A} \otimes \nabla_A^{(l-1-i)}\phi \\
&= \nabla_A^{(l)}\bigg{(} -\Delta_A^{(k+1)}\phi + \sum_{j=0}^{2k}\nabla_M^{(j)}Rm * \nabla_A^{(2k-j)}\phi
+ \sum_{j=0}^{2k}\nabla_M^{(j)}F_A * \nabla_A^{(2k-j)}\phi 
- \frac{1}{4}(S + \vert\phi\vert^2)\phi\bigg{)} \\
&\hspace{1cm} + \sum_{i=0}^{l-1}C_i\nabla_M^{(i)}\bigg{(} (-1)^{k+1}d^*\Delta_M^{(k)}F_A - \sum_{v=0}^{2k-1}P^{(v)}_1[F_A] \\
&\hspace{4cm} - 2iIm\big{(} \sum_{i=0}^kC_i\nabla_M^{*(i)}\langle\nabla_A^{(k)}\nabla_A\phi, \nabla_A^{(k-i)}\phi
\rangle \big{)}\bigg{)} \otimes \nabla_A^{(l-1-i)}\phi \\
&= -\nabla_A^{(l)}\Delta_A^{(k+1)}\phi + \sum_{j=0}^{2k+l}\nabla_M^{(j)}Rm * \nabla_A^{(2k+l-j)}\phi +
 \sum_{j=0}^{2k+l}\nabla_M^{(j)}F_A * \nabla_A^{(2k+l-j)}\phi \\
&\hspace{1cm} - \frac{1}{4}\nabla_A^{(l)}\big{(}(S + \vert\phi\vert^2)\phi\big{)} \\
&\hspace{1cm} + \bigg{(}\sum_{i=0}^{l-1}(-1)^{k+1}C_i\nabla_M^{(i)}d^*\Delta_M^{(k)}F_A +  
\sum_{v=0}^{2k-1 + l-1}P^{(v)}_1[F_A] \\
&\hspace{3cm} - 2iIm\big{(}\sum_{j=0}^{l-1} \sum_{i=0}^kC_i\nabla_M^{(j)}\nabla_M^{*(i)}\langle\nabla_A^{(k)}\nabla_A\phi, \nabla_A^{(k-i)}\phi
\rangle \big{)}\bigg{)} \otimes \nabla_A^{(l-1-i)}\phi.
\end{align*}
Applying the commutation formula, lemma \ref{connection_4}, then gives the result.

\end{proof}

\subsection{Estimates for derivatives of the spinor field}

We will prove local $L^2$-derivative estimates for the spinor field, and take up the case of the
curvature in the next subsection.

We start with the following lemma, which will prove to be very useful in the course of obtaining several
local estimates.

\begin{lem}\label{product_estimate}	
Let $\phi \in \Gamma(\mathcal{S}^+)$ and $p, q \in \mathbb{N}$, such that $p > q$. Given $k \in \mathbb{N}$, 
we have
\begin{equation*}
\big{\vert}\nabla_M^{(k)}\big{\langle}\nabla_A^{(p)}\phi, \nabla_A^{(q)}\phi\big{\rangle}\big{\vert}
\leq \sum_{j=0}^{k}C(g)\big{\vert}
\big{\langle}\nabla_A^{(j)}\nabla_A^{(p)}\phi, \nabla_A^{(k-j)}\nabla_A^{(q)}\phi
\big{\rangle}\big{\vert}.
\end{equation*}
\end{lem}
\begin{proof}
For this proof we will denote $\nabla_A$ by $\nabla$.

Start with the case $p = 1$ and $q = 0$, and suppose $k = 1$. We want to start by working out a formula for
$\nabla_M\langle\nabla\phi, \phi\rangle$. As everything is tensorial, we can work in coordinates. We fix a 
point $x \in M$, and work in normal coordinates centred at $x$. 
In these coordinates we write $\langle \nabla\phi, \phi\rangle$ as 
$\langle \nabla_i\phi, \phi\rangle dx^i$. Applying $\nabla_M$ to this we get (at the point $x$) 
\begin{align*}
\nabla_M\big{(}\langle \nabla_i\phi, \phi\rangle dx^i\big{)} &=
d\langle \nabla_i\phi, \phi\rangle \otimes dx^i + \langle \nabla_i\phi, \phi\rangle\nabla_M(dx^i) \\
&= \big{(}\langle\nabla\nabla_i\phi, \phi\rangle + \langle\nabla_i\phi, \nabla\phi\rangle\big{)} \otimes dx^i
+ \langle \nabla_i\phi, \phi\rangle\nabla_M(dx^i) \\
&= \big{(}\langle\nabla\nabla_i\phi, \phi\rangle + \langle\nabla_i\phi, \nabla\phi\rangle\big{)} \otimes dx^i
\end{align*}
where to get the second equality we have used the fact that $\nabla$ is metric compatible, and to get 
the third equality we are using the fact that at $x$ the Christoffel symbols vanish.

Since we are working with tensors, we thus have the formula
\begin{equation*}
\nabla_M\langle\nabla\phi, \phi\rangle = \langle\nabla\nabla\phi, \phi\rangle + 
\overline{\langle\nabla\phi, \nabla\phi\rangle}
\end{equation*}
where we are abusing notation and writing $\langle\nabla\phi, \nabla\phi\rangle$ to denote the 2-tensor, which
in coordinates is given by $\langle\nabla_i\phi, \nabla_j\phi\rangle$. Note that
$\vert\overline{\langle \nabla\phi, \nabla\phi\rangle}\vert = 
\vert\langle \nabla\phi, \nabla\phi\rangle\vert$.
The result for this case then follows.

Still assuming $k=1$, and taking
general $p$ and $q$, we write $p = q + r$, and then write 
$\nabla^{(p)}\phi = \nabla^{(r)}\nabla^{(q)}\phi$. Then in coordinates we can write
$\langle \nabla^{(r)}\nabla^{(q)}\phi, \nabla^{(q)}\phi\rangle$ as
\begin{equation*}
\langle \nabla_{i_1}\nabla_{i_2}\cdots\nabla_{i_r}\nabla^{(q)}\phi, \nabla^{(q)}\phi\rangle
dx^{i_1}\otimes\cdots\otimes dx^{i_r}.
\end{equation*}
Applying what we did above, we can then see that
\begin{equation*}
\nabla_M\langle\nabla^{(r)}\nabla^{(q)}\phi, \nabla^{(q)}\phi\rangle = 
\langle\nabla\nabla^{(r)}\nabla^{(q)}\phi, \nabla^{(q)}\phi\rangle + 
\overline{\langle\nabla\nabla^{(q)}\phi, \nabla^{(r)}\nabla^{(q)}\phi\rangle}
\end{equation*}
and the bound for this case follows as well.

Now, suppose we apply $\nabla_M$ to the above formula, we get
\begin{align*}
\nabla_M\nabla_M\langle\nabla^{(r)}\nabla^{(q)}\phi, \nabla^{(q)}\phi\rangle = 
\nabla_M\langle\nabla\nabla^{(r)}\nabla^{(q)}\phi, \nabla^{(q)}\phi\rangle + 
\nabla_M\overline{\langle\nabla\nabla^{(q)}\phi, \nabla^{(r)}\nabla^{(q)}\phi\rangle}.
\end{align*}
We can then apply what we did above, for the case of just one $\nabla_M$, to take the $\nabla_M$ 
to the inside on the right hand side, and then the bound follows. Iterating this, we get the full bound
for all $k$.

\end{proof}

Observe that
\begin{equation*}
\frac{\partial}{\partial t}\vert\vert \gamma^{s/2}\nabla_A^{(l)}\phi\vert\vert^2_{L^2} = 
\frac{\partial}{\partial t}\int \langle \gamma^{s/2}\nabla_A^{(l)}\phi, \gamma^{s/2}\nabla_A^{(l)}\phi\rangle = 
\int \langle \frac{\partial}{\partial t}\nabla_A^{(l)}\phi, \gamma^{s}\nabla_A^{(l)}\phi\rangle 
+ \langle \gamma^{s}\nabla_A^{(l)}\phi, \frac{\partial}{\partial t}\nabla_A^{(l)}\phi\rangle. 
\end{equation*}

From lemma \ref{derivative_spinor}, we then get the following proposition.
\begin{prop}\label{derivative_L^2_spinor}
Let $(\phi(t), A(t))$ be a solution to the generalised Seiberg-Witten flow. Then
\begin{align*}
&\hspace{0.5cm}
\frac{\partial}{\partial t}\vert\vert \gamma^{s/2}\nabla_A^{(l)}\phi\vert\vert^2_{L^2} \\
&= \int -2Re\big{(}\langle \Delta_A^{(k+1)}\nabla_A^{(l)}\phi, \gamma^{s}\nabla_A^{(l)}\phi \rangle \big{)} +
2Re\bigg{(}\langle \sum_{j=0}^{2k-2+l}\nabla^{(j)}_MRm * \nabla_A^{(2k-2+l-j)}\phi, 
\gamma^{s}\nabla_A^{(l)}\phi\rangle 
\bigg{)} \\
&\hspace{0.5cm} + 2Re\bigg{(}\langle\sum_{j=0}^{2k+l}\nabla^{(j)}_MRm * \nabla_A^{(2k+l-j)}\phi, 
\gamma^{s}\nabla_A^{(l)}\phi \rangle \bigg{)} \\
&\hspace{0.5cm} + 2Re\bigg{(}\langle \sum_{j=0}^{2k-2+l}\nabla^{(j)}_MF_A * \nabla_A^{(2k-2+l-j)}\phi, 
\gamma^{s}\nabla_A^{(l)}\phi\rangle\bigg{)} \\
&\hspace{0.5cm} + 
2Re\bigg{(}\langle \sum_{j=0}^{2k+l}\nabla^{(j)}_MF_A * \nabla_A^{(2k+l-j)}\phi, 
\gamma^{s}\nabla_A^{(l)}\phi\rangle\bigg{)} \\
&\hspace{0.5cm} + 2Re\bigg{(}\langle \sum_{i=0}^{l-1}(-1)^{k+1}C_i\nabla_M^{(i)}d^*\Delta_M^{(k)}F_A \otimes 
\nabla_A^{(l-1-i)}\phi, \gamma^{s}\nabla_A^{(l)}\phi\rangle\bigg{)} \\
&\hspace{0.5cm} + 2Re\bigg{(}\langle  \sum_{i=0}^{l-1}\sum_{v=0}^{2k-1+i}P_1^{(v)}[F_A] \otimes 
\nabla_A^{(l-1-i)}\phi, \gamma^{s}\nabla_A^{(l)}\phi\rangle\bigg{)} \\
&\hspace{0.5cm} + 2Re\bigg{(}\langle - 2iIm\bigg{(} 
\sum_{j=0}^{l-1}\sum_{i=1}^kC_i\nabla_M^{(j)}\nabla_M^{*(i)}\langle\nabla_A^{(k)}\nabla_A\phi, 
\nabla_A^{(k-i)}\phi\rangle\bigg{)}\otimes \nabla_A^{(l-1-j)}\phi, \gamma^{s}\nabla_A^{(l)}\phi\rangle\bigg{)} \\
&\hspace{0.5cm} - \frac{1}{2}Re\bigg{(}\langle  \nabla_A^{(l)}\big{(}(S + \vert\phi\vert^2)\phi\big{)}, 
\gamma^{s}\nabla_A^{(l)}\phi\rangle\bigg{)}.
\end{align*}
\end{prop} 

We are now going to estimate each term on the right hand side of the above proposition.


\begin{lem}\label{smoothing_spinor_1}
Assume $\sup_{t \in [0,T)}\vert\vert F_A\vert\vert_{\infty} < \infty$, and let 
$K(\vert\vert\phi\vert\vert_{\infty}) 
= max\{1, \sup_{t\in [0,T)}\vert\vert\phi\vert\vert_{\infty}\}$. Suppose $\gamma$ is a bump function, and
$s \geq 2(k+l)$.
Then for 
$\epsilon_1, \epsilon_2, \epsilon_3, \tilde{\epsilon}_{3}, \epsilon_4, \tilde{\epsilon}_4 > 0$ sufficiently small,   
we have the following estimate
\begin{align*}
&\hspace{0.5cm}
\int -2Re\big{(}\langle \Delta_A^{(k+1)}\nabla_A^{(l)}\phi, \gamma^{s}\nabla_A^{(l)}\phi \rangle \big{)} \\
&\leq
\bigg{(} -2 + C(g, \gamma)(\epsilon_1 + \epsilon_2) + C(g)\epsilon_3 + 
C(g, \gamma)K(\vert\vert\phi\vert\vert_{\infty})\tilde{\epsilon}_3 \\
&\hspace{1cm} 
C(g, \gamma)\big{(}\sup_{t \in [0,T)}\vert\vert F_A\vert\vert_{\infty}\big{)}(\epsilon_4 + 
 K(\vert\vert\phi\vert\vert_{\infty})\tilde{\epsilon}_4)\bigg{)}
 \vert\vert\gamma^{s/2}\nabla_A^{k+1}\nabla_A^{(l)}\phi\vert\vert^2_{L^2} \\
&\hspace{0.5cm} +
\bigg{(}\frac{C(\epsilon_2, g, \gamma)}{\epsilon_1^2} + (C(\epsilon_3, g) + 
C(\tilde{\epsilon}_3, g, \gamma))K(\vert\vert\phi\vert\vert_{\infty}) \\
&\hspace{0.5cm} +
 C(g, \gamma)\big{(}\sup_{t \in [0,T)}\vert\vert F_A\vert\vert_{\infty}\big{)}(C(\epsilon_4, g, \gamma) + 
  C(\tilde{\epsilon}_4, g, \gamma)K(\vert\vert\phi\vert\vert_{\infty}))\bigg{)}
  \vert\vert\phi\vert\vert^2_{L^2, \gamma>0}
\end{align*}
where $C(g), C(g, \gamma), C(\epsilon_2, g, \gamma), C(\epsilon_3, g, \gamma), C(\tilde{\epsilon}_{3}, g, \gamma), 
C(\epsilon_4, g, \gamma), C(\tilde{\epsilon}_4, g, \gamma)$ are constants that do not depend on $t \in [0, T)$.
\end{lem}
\begin{proof}
We will start by performing an estimate on the quantity 
$\int -\langle \Delta_A^{(k+1)}\nabla_A^{(l)}\phi, \gamma^{s}\nabla_A^{(l)}\phi \rangle$. Observe that
using lemma \ref{connection_6}, we have
\begin{align}
\int -\langle \Delta_A^{(k+1)}\nabla_A^{(l)}\phi, \gamma^{s}\nabla_A^{(l)}\phi \rangle &=
\int -\langle \nabla_A^{(k+1)}\nabla_A^{(l)}\phi, \nabla_A^{(k+1)}\big{(}\gamma^{s}\nabla_A^{(l)}\phi\big{)} 
 \rangle  \label{smoothing_est_1}     \\
&\hspace{0.5cm} + \int\langle \sum_{w=0}^{2k-2}\nabla^{(w)}_MRm * \nabla_A^{(2k-2-w)}\nabla_A^{(l)}\phi, 
\gamma^{s}\nabla_A^{(l)}\phi\rangle \label{smoothing_est_2} \\
&\hspace{0.5cm} + \int\langle \sum_{w=0}^{2k-2}\nabla^{(w)}_MF_A * \nabla_A^{(2k-2-w)}\nabla_A^{(l)}\phi, 
 \gamma^{s}\nabla_A^{(l)}\phi\rangle. \label{smoothing_est_3}
\end{align}

Note that we then have
\begin{align*}
\int -2Re\big{(}\langle \Delta_A^{(k+1)}\nabla_A^{(l)}\phi, \gamma^{s}\nabla_A^{(l)}\phi \rangle \big{)} &=
\int -2Re\big{(}\langle \nabla_A^{(k+1)}\nabla_A^{(l)}\phi, \nabla_A^{(k+1)}\big{(}\gamma^{s}\nabla_A^{(l)}\phi\big{)} 
 \rangle\big{)}      \\
&\hspace{0.5cm} + \int 2Re\big{(}\langle \sum_{w=0}^{2k-2}\nabla^{(w)}_MRm * \nabla_A^{(2k-2-w)}\nabla_A^{(l)}\phi, 
\gamma^{s}\nabla_A^{(l)}\phi\rangle\big{)}  \\
&\hspace{0.5cm} + \int 2Re\big{(} \langle \sum_{w=0}^{2k-2}\nabla^{(w)}_MF_A * \nabla_A^{(2k-2-w)}\nabla_A^{(l)}\phi, 
 \gamma^{s}\nabla_A^{(l)}\phi\rangle\big{)}. 
\end{align*}

We first estimate the quantity on the right hand side of \eqref{smoothing_est_1}.
\begin{align*}
\int -\langle \nabla_A^{(k+1)}\nabla_A^{(l)}\phi, \nabla_A^{(k+1)}\big{(}\gamma^{s}\nabla_A^{(l)}\phi\big{)} 
 \rangle &= \int -\langle \nabla_A^{(k+1)}\nabla_A^{(l)}\phi,
 \sum_{j=0}^{k+1}C_j\nabla^{(j)}\gamma^{s} \otimes\nabla_A^{(k+1-j)}\nabla_A^{(l)}\phi\big{)} \rangle
\end{align*}
where $C_j$ is a constant and $C_0 = 1$. We can then split this into two terms, giving
\begin{align*}
\int -\langle \nabla_A^{(k+1)}\nabla_A^{(l)}\phi, \nabla_A^{(k+1)}\big{(}\gamma^{s}\nabla_A^{(l)}\phi\big{)} 
 \rangle &=
\int -\langle \nabla_A^{(k+1)}\nabla_A^{(l)}\phi, \gamma^{s}\nabla_A^{(k+1)}\nabla_A^{(l)}\phi\rangle \\
&\hspace{0.5cm} + 
 \int -\langle \nabla_A^{(k+1)}\nabla_A^{(l)}\phi,
 \sum_{j=1}^{k+1}C_j\nabla^{(j)}\gamma^{s} \otimes\nabla_A^{(k+1-j)}\nabla_A^{(l)}\phi\big{)} \rangle \\
&= -\vert\vert\gamma^{s/2}\nabla_A^{(k+1)}\nabla_A^{(l)}\phi\vert\vert^2_{L^2} \\
&\hspace{1cm} + 
\int \sum_{j=1}^{k+1}\nabla^{(j)}(\gamma^{s}) *  \langle \nabla_A^{(k+1)}\nabla_A^{(l)}\phi, 
\nabla_A^{(k+1 - j)}\nabla_A^{(l)}\phi\rangle.
\end{align*} 
In order to estimate  $\int \sum_{j=1}^{k+1}\nabla^{(j)}(\gamma^{s}) *  \langle \nabla_A^{(k+1)}\nabla_A^{(l)}\phi, 
\nabla_A^{(k+1 - j)}\nabla_A^{(l)}\phi\rangle$, we proceed as follows.
\begin{align*}
&\hspace{0.5cm}
\bigg{\vert} \int \sum_{j=1}^{k+1}\nabla^{(j)}(\gamma^{s}) *  \langle \nabla_A^{(k+1)}\nabla_A^{(l)}\phi, 
\nabla_A^{(k+1 - j)}\nabla_A^{(l)}\phi\rangle \bigg{\vert} \\
&\leq
\int \sum_{j=1}^{k+1}C(g)\vert\nabla^{(j)}(\gamma^{s})\vert   \vert \nabla_A^{(k+1)}\nabla_A^{(l)}\phi\vert \vert 
\nabla_A^{(k+1 - j)}\nabla_A^{(l)}\phi\vert \\
&\leq  \int \sum_{j=1}^{k+1}C(g, \gamma)\gamma^{s-j}   \vert \nabla_A^{(k+1)}\nabla_A^{(l)}\phi\vert 
\vert \nabla_A^{(k+1 - j)}\nabla_A^{(l)}\phi\vert \\
&= \int \sum_{j=1}^{k+1}C(g, \gamma) \vert\gamma^{s/2}  \nabla_A^{(k+1)}\nabla_A^{(l)}\phi\vert 
\vert\gamma^{\frac{s-2j}{2}}  \nabla_A^{(k+1 - j)}\nabla_A^{(l)}\phi\vert. \\
\end{align*}

We then apply a weighted Young's inequality to obtain
\begin{align*}
&\hspace{0.5cm}
\int \sum_{j=1}^{k+1}C(g, \gamma) \vert\gamma^{s/2}  \nabla_A^{(k+1)}\nabla_A^{(l)}\phi\vert 
\vert\gamma^{\frac{s-2j}{2}}  \nabla_A^{(k+1 - j)}\nabla_A^{(l)}\phi\vert \\
&\leq 
(k+1)C(g, \gamma)\epsilon_1\vert\vert\gamma^{s/2}\nabla_A^{(k+1)}\nabla_A^{(l)}\phi\vert\vert^2_{L^2} 
 + \sum_{j=1}^{k+1}\frac{C(g,\gamma)}{\epsilon_1}
\vert\vert\gamma^{\frac{s-2j}{2}}\nabla_A^{(k+1-j)}\nabla_A^{(l)}\phi\vert\vert^2_{L^2}.
\end{align*}
Choose $\epsilon_1$ sufficiently small so that $\frac{C(g,\gamma)}{\epsilon_1} \geq 1$. By applying lemma 
\ref{interp_3}, we can bound the above by
\begin{align*}
C(g, \gamma)(\epsilon_1 + \epsilon_2)\vert\vert\gamma^{s/2}\nabla_A^{(k+1)}\nabla_A^{(l)}\phi\vert\vert^2_{L^2} 
+ \frac{C(\epsilon_2, g, \gamma)}{\epsilon_1^2}\vert\vert\phi\vert\vert^2_{L^2, \gamma > 0}
\end{align*}
where we have absorbed the $(k+1)$ into the constant $C(g, \gamma)$. 

Putting this together with the previous estimate, 
we get the following estimate

\begin{align}
\int -2Re\bigg{(}\langle \nabla_A^{(k+1)}\nabla_A^{(l)}\phi, 
\nabla_A^{(k+1)}\big{(}\gamma^{s}\nabla_A^{(l)}\phi\big{)} 
 \rangle\bigg{)} &\leq -2\vert\vert\gamma^{s/2}\nabla_A^{(k+1)}\nabla_A^{(l)}\phi\vert\vert^2_{L^2}
 \label{smoothing_est_1a} \\
&\hspace{0.5cm} +
C(g, \gamma)(\epsilon_1 + \epsilon_2)\vert\vert\gamma^{s/2}\nabla_A^{(k+1)}\nabla_A^{(l)}\phi\vert\vert^2_{L^2} 
\nonumber \\
&\hspace{0.5cm} + \frac{C(\epsilon_2, g, \gamma)}{\epsilon_1^2}\vert\vert\phi\vert\vert^2_{L^2, \gamma > 0}. \nonumber
\end{align}

The next step is to estimate the absolute value of \ref{smoothing_est_2}. 
\begin{align*}
&\hspace{0.5cm}
\bigg{\vert}\int \langle \sum_{w=0}^{2k-2}\nabla^{(w)}_MRm * \nabla_A^{(2k-2-w)}\nabla_A^{(l)}\phi, 
\gamma^{s}\nabla_A^{(l)}\phi\rangle \bigg{\vert} \\
&\leq 
\int \big{\vert}\langle \sum_{w=0}^{2k-2}\nabla^{(w)}_MRm * \nabla_A^{(2k-2-w)}\nabla_A^{(l)}\phi, 
\gamma^{s}\nabla_A^{(l)}\phi\rangle\big{\vert} \\
&\leq \int \sum_{w=0}^{2k-2}\gamma^{s}\vert\nabla^{(w)}_MRm\vert \vert\nabla_A^{(2k-2-w)}\nabla_A^{(l)}\phi\vert 
\vert\nabla_A^{(l)}\phi\vert.
\end{align*}
In order to estimate the quantity
\begin{equation*}
\int \sum_{w=0}^{2k-2}\gamma^{s}\vert\nabla^{(w)}_MRm\vert \vert\nabla_A^{(2k-2-w)}\nabla_A^{(l)}\phi\vert 
\vert\nabla_A^{(l)}\phi\vert
\end{equation*}
we will split the integrand into two parts, those for which $w$ is even, and those for which $w$ is odd. We
will then show how to estimate each piece.
\begin{itemize}
\item[1.] Fix $w = 2\alpha$ for $0\leq \alpha \leq k-1$. We then have
\begin{equation*}
\int\gamma^{s}\vert\nabla^{(w)}_MRm\vert \vert\nabla_A^{(2k-2-w)}\nabla_A^{(l)}\phi\vert 
\vert\nabla_A^{(l)}\phi\vert = \int\gamma^{s}\vert\nabla^{(w)}_MRm\vert 
\vert\nabla_A^{(2k-2-2\alpha)}\nabla_A^{(l)}\phi\vert 
\vert\nabla_A^{(l)}\phi\vert.
\end{equation*}
The term $\vert\nabla^{(w)}_MRm\vert$ does not depend on time, and since $M$ is compact, is bounded on $M$. We can
therefore view it as a constant $C(g)$. Applying corollary \ref{interp_2}, and then 
lemma \ref{interp_3}, we have
\begin{align*}
\int \gamma^{s} \vert\nabla_A^{(2k-2-2\alpha)}\nabla_A^{(l)}\phi\vert \vert\nabla_A^{(l)}\phi\vert &\leq
C(g)\vert\vert \gamma^{s/2}\nabla_A^{(k-1+l-\alpha)}\phi\vert\vert^2_{L^2} + C(g)\vert\vert\phi\vert\vert^2_{L^2, 
\gamma>0} \\
&\leq C(g)\epsilon_3\vert\vert\gamma^{s/2} \nabla_A^{(k+1)}\nabla_A^{(l)}\phi\vert\vert^2_{L^2}  
+ C(\epsilon_3,g)\vert\vert\phi\vert\vert^2_{L^2, \gamma>0}.
\end{align*} 

\item[2.] Fix $w = 2\alpha + 1$ for $0\leq \alpha \leq k-2$. We then have 
\begin{align*}
\int \gamma^{s}\vert\nabla^{(w)}_MRm\vert \vert\nabla_A^{(2k-2-w)}\nabla_A^{(l)}\phi\vert 
\vert\nabla_A^{(l)}\phi\vert &= 
\int\gamma^{s} \vert\nabla^{(2\alpha +1)}_MRm\vert \vert\nabla_A^{(2k-2-2\alpha -1)}\nabla_A^{(l)}\phi\vert 
\vert\nabla_A^{(l)}\phi\vert \\
&= \int\gamma^{s}\vert\nabla_MT\vert \vert\nabla_A^{(2k-2-2\alpha -1)}\nabla_A^{(l)}\phi\vert 
\vert\nabla_A^{(l)}\phi\vert
\end{align*}
where we are letting $T = \nabla_M^{2\alpha}Rm$. We remind the reader that $T$ does not depend on time $t$, and
by compactness of $M$, is uniformly bounded above by some constant.

Applying Holder's inequality, we can bound the quantity 
\begin{equation*}
\int\gamma^{s}\vert\nabla_MT\vert \vert\nabla_A^{(2k-2-2\alpha -1)}\nabla_A^{(l)}\phi\vert 
\vert\nabla_A^{(l)}\phi\vert
\end{equation*}
by the quantity
\begin{align*}
\bigg{(}\int\gamma^{s} \vert\nabla_MT\vert^{\frac{2(k-1-\alpha +l)}{1}}\bigg{)}^{\frac{1}{2(k-1-\alpha +l)}}
\bigg{(}\int\gamma^{s} \vert\nabla_A^{(2k-2-2\alpha -1 + l)}\phi\vert^{\frac{2(k-1-\alpha +l)}{2k-2-2\alpha -1+l}}
\bigg{)}^{\frac{2k-2-2\alpha -1+l}{2(k-1-\alpha +l)}} \\
\bigg{(}\int\gamma^{s} \vert\nabla_A^{(l)}\phi\vert^{\frac{2(k-1-\alpha +l)}{l}}\bigg{)}^{\frac{l}{2(k-1-\alpha +l)}}.
\end{align*}
As mentioned before, since $T$ does not depend on time, and using the compactness of $M$, we can simply express
the term in the first bracket as $C(g, \gamma)$. We therefore need to estimate the quantity
\begin{align*}
C(g, \gamma)
\bigg{(}\int\gamma^{s} \vert\nabla_A^{(2k-2-2\alpha -1 + l)}\phi\vert^{\frac{2(k-1-\alpha +l)}{2k-2-2\alpha -1+l}}
\bigg{)}^{\frac{2k-2-2\alpha -1+l}{2(k-1-\alpha +l)}} 
\bigg{(}\int\gamma^{s} \vert\nabla_A^{(l)}\phi\vert^{\frac{2(k-1-\alpha +l)}{l}}\bigg{)}^{\frac{l}{2(k-1-\alpha +l)}}.
\end{align*}
Appealing to theorem \ref{interp_1}, we can bound it above by
\begin{align*}
C(g, \gamma)\bigg{[}\vert\vert\phi\vert\vert_{\infty}^{1-\frac{2k-2-2\alpha -1+l}{k-1-\alpha +l}}
\bigg{(}\vert\vert\gamma^{s/2}\nabla_A^{(k-1-\alpha +l)}\phi\vert\vert_{L^2} 
+ \vert\vert\phi\vert\vert_{L^2, \gamma>0}\bigg{)}^{\frac{2k-2-2\alpha -1+l}{k-1-\alpha +l}}\bigg{]} \\
\bigg{[}\vert\vert\phi\vert\vert_{\infty}^{1-\frac{l}{k-1-\alpha +l}}
\bigg{(}\vert\vert\gamma^{s/2}\nabla_A^{(k-1-\alpha +l)}\phi\vert\vert_{L^2}
+ \vert\vert\phi\vert\vert_{L^2, \gamma>0}\bigg{)}^{\frac{l}{k-1-\alpha +l}}\bigg{]} 
\end{align*}
which simplifies to
\begin{align*}
C(g, \gamma)\vert\vert\phi\vert\vert_{\infty}^
{\frac{1}{k-1-\alpha +l}}\bigg{(}\vert\vert\gamma^{s/2}\nabla_A^{(k-1-\alpha +l)}\phi\vert\vert_{L^2} 
+ \vert\vert\phi\vert\vert_{L^2, \gamma>0}\bigg{)}^{\frac{2(k-1-\alpha +l) -1}{k-1-\alpha +l}}.
\end{align*}
Recall we defined $K(\vert\vert\phi\vert\vert_{\infty}) 
= max\{1, \sup_{t\in [0,T)}\vert\vert\phi\vert\vert_{\infty}\}$. We can then bound the above by 
\begin{align*}
C(g, \gamma)K(\vert\vert\phi\vert\vert_{\infty})^
{\frac{1}{k-1-\alpha +l}}\bigg{(}\vert\vert\gamma^{s/2}\nabla_A^{(k-1-\alpha +l)}\phi\vert\vert_{L^2} 
+ \vert\vert\phi\vert\vert_{L^2, \gamma>0}\bigg{)}^{\frac{2(k-1-\alpha +l) -1}{k-1-\alpha +l}}.
\end{align*}
We then have that
\begin{align*}
&\hspace{0.5cm} C(g, \gamma)K(\vert\vert\phi\vert\vert_{\infty})^
{\frac{1}{k-1-\alpha +l}}\bigg{(}\vert\vert\gamma^{s/2}\nabla_A^{(k-1-\alpha +l)}\phi\vert\vert_{L^2} 
+ \vert\vert\phi\vert\vert_{L^2, \gamma>0}\bigg{)}^{\frac{2(k-1-\alpha +l) -1}{k-1-\alpha +l}} \\
&\leq
C(g, \gamma)K(\vert\vert\phi\vert\vert_{\infty})^
{\frac{1}{k-1-\alpha +l}}\bigg{(}\vert\vert\gamma^{s/2}\nabla_A^{(k-1-\alpha +l)}\phi\vert\vert_{L^2} 
+ \vert\vert\phi\vert\vert_{L^2, \gamma>0} + 1\bigg{)}^{\frac{2(k-1-\alpha +l) -1}{k-1-\alpha +l}} \\
&\leq 
C(g, \gamma)K(\vert\vert\phi\vert\vert_{\infty})^
{\frac{1}{k-1-\alpha +l}}\bigg{(}\vert\vert\gamma^{s/2}\nabla_A^{(k-1-\alpha +l)}\phi\vert\vert_{L^2} 
+ \vert\vert\phi\vert\vert_{L^2, \gamma>0} + 1\bigg{)}^{2}  \\
&\leq
C(g, \gamma)K(\vert\vert\phi\vert\vert_{\infty})
\bigg{(}\vert\vert\gamma^{s/2}\nabla_A^{(k-1-\alpha +l)}\phi\vert\vert_{L^2} 
+ \vert\vert\phi\vert\vert_{L^2, \gamma>0} + 1\bigg{)}^{2} \\
&\leq 
C(g, \gamma)K(\vert\vert\phi\vert\vert_{\infty})
\bigg{(}\vert\vert\gamma^{s/2}\nabla_A^{(k-1-\alpha +l)}\phi\vert\vert^2_{L^2} 
+ \vert\vert\phi\vert\vert^2_{L^2, \gamma>0} + 1\bigg{)}
\end{align*}
where we have used the general fact that, given any three non-negative integers $a, b, c$ we have
$(a + b + c)^2 \leq 2(a^2 + b^2 + c^2)$. In our situation we have absorbed the $2$ into the constant
$C(g,\gamma)$.

We then apply lemma \ref{interp_3}, to the first term in the bracket, 
obtaining
\begin{align*}
&\hspace{0.5cm} C(g, \gamma)K(\vert\vert\phi\vert\vert_{\infty})
\bigg{(}\vert\vert\gamma^{s/2}\nabla_A^{(k-1-\alpha +l)}\phi\vert\vert^2_{L^2} 
+ \vert\vert\phi\vert\vert^2_{L^2, \gamma>0} + 1\bigg{)} \\
&\leq C(g, \gamma)K(\vert\vert\phi\vert\vert_{\infty})\tilde{\epsilon}_3\vert\vert
\gamma^{s/2}\nabla_A^{(k+1)}\nabla_A^{(l)}\phi\vert\vert^2_{L^2}  
+ C(\tilde{\epsilon}_3, g, \gamma)K(\vert\vert\phi\vert\vert_{\infty})\vert\vert\phi\vert\vert^2_{L^2, \gamma>0} 
\end{align*}
where we have absorbed the extra $C(g, \gamma)K(\vert\vert\phi\vert\vert_{\infty})$, coming from taking this
into the bracket and multiplying by $1$, into the coefficient of 
$\vert\vert\phi\vert\vert^2_{L^2, \gamma>0}$.
\end{itemize}
Putting the two estimates together gives the following
\begin{align*}
&\hspace{0.5cm}
\int \big{\vert}\langle \sum_{w=0}^{2k-2}\nabla^{(w)}_MRm * \nabla_A^{(2k-2-w)}\nabla_A^{(l)}\phi, 
\gamma^{s}\nabla_A^{(l)}\phi\rangle\big{\vert} \\
&\leq 
C(g)\epsilon_3\vert\vert\gamma^{s/2} \nabla_A^{(k+1)}\nabla_A^{(l)}\phi\vert\vert^2_{L^2}  
+ C(\epsilon_3,g)\vert\vert\phi\vert\vert^2_{L^2, \gamma>0} 
 + C(g, \gamma)K(\vert\vert\phi\vert\vert_{\infty})\tilde{\epsilon}_3\vert\vert
\gamma^{s/2}\nabla_A^{(k+1)}\nabla_A^{(l)}\phi\vert\vert^2_{L^2} \\ 
&\hspace{0.5cm} + 
C(\tilde{\epsilon}_3, g, \gamma)K(\vert\vert\phi\vert\vert_{\infty})\vert\vert\phi\vert\vert^2_{L^2, \gamma>0}. 
\end{align*}
We have thus obtained the following estimate 
\begin{align}
& \hspace{0.5cm} 
\int 2Re\big{(}\langle \sum_{w=0}^{2k-2}\nabla^{(w)}_MRm * \nabla_A^{(2k-2-w)}\nabla_A^{(l)}\phi, 
\gamma^{s}\nabla_A^{(l)}\phi\rangle\big{)} \label{smoothing_est_2a}\\ 
&\leq 
\big{(}C(g)\epsilon_3 + C(g, \gamma)K(\vert\vert\phi\vert\vert_{\infty})\tilde{\epsilon}_3\big{)}
\vert\vert\gamma^{s/2} \nabla_A^{(k+1)}\nabla_A^{(l)}\phi\vert\vert^2_{L^2} \nonumber \\
&\hspace{0.5cm} + \big{(}C(\epsilon_3,g) + C(\tilde{\epsilon}_3, g, \gamma)K(\vert\vert\phi\vert\vert_{\infty})  
\big{)}\vert\vert\phi\vert\vert^2_{L^2, \gamma>0}. \nonumber
\end{align} 

The final step is to estimate the absolute value of \ref{smoothing_est_3}. We first observe that we can write the
term $\nabla^{(w)}_MF_A * \nabla_A^{(2k-2-w)}\nabla_A^{(l)}\phi$ as 
$\sum_{j=0}^{w}C_j\nabla^{j}(F_A * \nabla_A^{(2k-2-j)}\nabla_A^{(l)}\phi)$, 
which follows from the fact that, for two tensors $S$ and $T$ we have
$\nabla^{(k)}(S*T) = \sum_{i=0}^kC_i\nabla^{(i)}S*\nabla^{(k-i)}T$. We then obtain
\begin{align*}
\int\langle \sum_{w=0}^{2k-2}\nabla^{(w)}_MF_A * \nabla_A^{(2k-2-w)}\nabla_A^{(l)}\phi, 
 \gamma^{s}\nabla_A^{(l)}\phi\rangle &= 
 \int \sum_{w=0}^{2k-2}\sum_{j=0}^wC_j\langle\nabla^{(j)}(F_A * \nabla_A^{(2k-2-j)}\nabla_A^{(l)}\phi), 
 \gamma^{s}\nabla_A^{(l)}\phi\rangle \\
&=
 \int \sum_{w=0}^{2k-2}\sum_{j=0}^wC_j\langle F_A * \nabla_A^{(2k-2-j)}\nabla_A^{(l)}\phi), 
 P_1^{(j)}\big{(}\gamma^{s}\nabla_A^{(l)}\phi\big{)}\rangle
\end{align*}
where to get the last equality we have applied integration by parts, and absorbed the constant $(-1)^j$ into
the $C_j$. We point out that $C_j$ in general won't be positive, some of them will be negative.

We then have
\begin{align*}
\bigg{\vert}\int\langle \sum_{w=0}^{2k-2}\nabla^{(w)}_MF_A * \nabla_A^{(2k-2-w)}\nabla_A^{(l)}\phi, 
 \gamma^{s}\nabla_A^{(l)}\phi\rangle\bigg{\vert} &\leq 
\int \sum_{w=0}^{2k-2}\sum_{j=0}^wC_j\vert\langle F_A * \nabla_A^{(2k-2-j)}\nabla_A^{(l)}\phi), 
 P_1^{(j)}\big{(}\gamma^{s}\nabla_A^{(l)}\phi\big{)}\rangle\vert.
\end{align*} 
When we take the absolute value inside to the
integrand, in the above inequality, the constants $C_j$ become $\vert C_j\vert$, and we have simply called this
$C_j$ again. Thus, on the right hand side of the above inequality, the $C_j$ are now all positive. We can then
estimate the right hand side of the above inequality by
\begin{align*}
&\hspace{1cm} \int \sum_{w=0}^{2k-2}\sum_{j=0}^wC_j\vert\langle F_A * \nabla_A^{(2k-2-j)}\nabla_A^{(l)}\phi), 
 P_1^{(j)}\big{(}\gamma^{s}\nabla_A^{(l)}\phi\big{)}\rangle\vert \\
 &\leq
\int \sum_{w=0}^{2k-2}\sum_{j=0}^wC_j\bigg{(}\sup_{t \in [0,T)}\vert\vert F_A\vert\vert_{\infty} \bigg{)} 
\vert\nabla_A^{(2k-2-j)}\nabla_A^{(l)}\phi\vert 
 \vert P_1^{(j)}\big{(}\gamma^{s}\nabla_A^{(l)}\phi\big{)}\vert. 
 \end{align*}

The term $P_1^{(j)}\big{(}\gamma^{s}\nabla_A^{(l)}\phi\big{)} = 
\nabla_A^{(j)}\big{(}\gamma^{s}\nabla_A^{(l)}\phi\big{)} * S$, where $S$ is some tensor depending on the metric
$g$, and in particular does not depend on $t$. Therefore, we have the bound 
\begin{equation*}
\vert P_1^{(j)}\big{(}\gamma^{s}\nabla_A^{(l)}\phi\big{)}\vert 
\leq C(g)\vert\nabla_A^{(j)}\big{(}\gamma^{s}\nabla_A^{(l)}\phi\big{)}\vert \leq 
\sum_{i=0}^{j}C(g)\vert\nabla^{(i)}(\gamma^{s})\vert\vert\nabla_A^{(j-i)}\nabla_A^{(l)}\phi\vert
\end{equation*}
where we have used the fact that we can write 
$\nabla_A^{(j)}(\gamma^{s}\nabla_A^{(l)}\phi) = 
\sum_{i=0}^{j}C_i\nabla^{(i)}(\gamma^{s})\otimes\nabla_A^{(j-i)}\nabla_A^{(l)}\phi$, for some positive constants
$C_i$. 

Putting this together, we obtain the bound
\begin{align}
&\hspace{0.5cm} \int \sum_{w=0}^{2k-2}\sum_{j=0}^wC_j\bigg{(}\sup_{t \in [0,T)}\vert\vert F_A\vert\vert_{\infty} \bigg{)} 
\vert\nabla_A^{(2k-2-j)}\nabla_A^{(l)}\phi\vert 
 \vert P_1^{(j)}\big{(}\gamma^{s}\nabla_A^{(l)}\phi\big{)}\vert \label{smoothing_est_3a}\\
 &\leq
\int \sum_{w=0}^{2k-2}\sum_{j=0}^w\sum_{i=0}^jC(g)\bigg{(}\sup_{t \in [0,T)}\vert\vert F_A\vert\vert_{\infty} 
\bigg{)} 
\vert\nabla_A^{(2k-2-j)}\nabla_A^{(l)}\phi\vert 
 \vert\nabla^{(i)}(\gamma^{s})\vert\vert\nabla_A^{(j-i)}\nabla_A^{(l)}\phi\vert. \nonumber
\end{align}
In order to estimate the right hand side of the above inequality, we will split the integrand into two
cases based on the parity of $i$. 

\begin{itemize}
\item[1.] Suppose that $i$ is even. Write $i = 2\alpha$ for $\alpha \geq 0$. Then
\begin{equation*}
\int \vert \nabla_A^{(2k-2-j+l)}\phi\vert\vert\nabla^{(2\alpha)}(\gamma^{s})\vert
\vert\nabla_A^{(j-2\alpha)}\nabla_A^{(l)}\phi\vert \leq
\int C(\gamma, g)\gamma^{s-2\alpha} \vert \nabla_A^{(2k-2-j+l)}\phi\vert
\vert\nabla_A^{(j-2\alpha)}\nabla_A^{(l)}\phi\vert
\end{equation*}
where we are using lemma \ref{bump}. 

Applying theorem \ref{interp_1}, we obtain
\begin{align*}
&\hspace{0.5cm}
\int C(\gamma, g)\gamma^{s-2\alpha} \vert \nabla_A^{(2k-2-j+l)}\phi\vert
\vert\nabla_A^{(j-2\alpha)}\nabla_A^{(l)}\phi\vert \\
&\leq
C(\gamma, g)\big{(}\vert\vert \gamma^{s-2\alpha}\nabla_A^{(k-1-\alpha+l)}\phi\vert\vert_{L^2} +
\vert\vert\phi\vert\vert_{L^2, \gamma>0} \big{)}^2 \\
&\leq
C(\gamma, g)\big{(}\vert\vert \gamma^{s-2\alpha}\nabla_A^{(k-1-\alpha+l)}\phi\vert\vert_{L^2}^2 +
\vert\vert\phi\vert\vert_{L^2, \gamma>0}^2 \big{)}
\end{align*}
Applying interpolation, lemma \ref{interp_3}, we then get
\begin{align*}
&\hspace{0.5cm}
C(\gamma, g)\big{(}\vert\vert \gamma^{s-2\alpha}\nabla_A^{(k-1-\alpha+l)}\phi\vert\vert_{L^2}^2 +
\vert\vert\phi\vert\vert_{L^2, \gamma>0}^2 \big{)} \\
&\leq 
C(\gamma, g)\epsilon_4\vert\vert \gamma^{s/2}\nabla_A^{(k+1)}\nabla_A^{(l)}\phi\vert\vert_{L^2}^2 +
C(\epsilon_4, g, \gamma)\vert\vert\phi\vert\vert_{L^2, \gamma>0}^2.
\end{align*}

\item[2.] We now consider the case that $i$ is odd. Write $i = 2\alpha + 1$ for $\alpha \geq 0$. 
From lemma \ref{bump}, we can write the derivative
$\nabla^{(i)}\gamma^{s}$ as
\begin{equation*}
\nabla^{(i)}\gamma^{s} = \sum_{\substack{n_1+\cdots +n_i = i\\0\leq n_1\leq \cdots\leq n_i\leq i}}
C_{(n_1,\ldots,n_i)}(\gamma)\gamma^{s-i}\nabla_A^{n_1}*\cdots*\nabla_A^{n_i} 
\end{equation*}
and obtain the bound
\begin{align*}
\vert \nabla^{(i)}\gamma^{s} \vert &\leq 
\sum_{\substack{n_1+\cdots +n_i = i\\0\leq n_1\leq \cdots\leq n_i\leq i}}C(\gamma)\gamma^{s-i}
\vert\nabla^{n_1}\gamma\vert\cdots\vert\nabla^{n_{i-1}}\gamma\vert\vert\nabla^{n_i}\gamma\vert \\
&\leq 
\sum_{\substack{n_1+\cdots +n_i = i\\0\leq n_1\leq \cdots\leq n_i\leq i}}
C(\gamma)\gamma^{s-i}\vert\vert\nabla^{n_1}\gamma\vert\vert_{\infty}
\cdots\vert\vert\nabla^{n_{i-1}}\gamma\vert\vert_{\infty}
\vert\nabla\nabla^{n_i-1}\gamma\vert \\
&\leq 
\sum_{n_i=1}^{i}C(\gamma)\gamma^{s-i}\vert\nabla\nabla^{n_i-1}\gamma\vert
\end{align*}
where to get the last inequality we have absorbed the norms $\vert\vert\nabla^{n_q}\gamma\vert\vert_{\infty}$, for 
$1\leq q \leq i-1$, into the constant $C(\gamma)$.

This gives the integral bound
\begin{align*}
&\hspace{0.5cm}
\int \vert \nabla_A^{(2k-2-j+l)}\phi\vert\vert\nabla^{(2\alpha + 1)}(\gamma^{s})\vert
\vert\nabla_A^{(j-2\alpha -1)}\nabla_A^{(l)}\phi\vert \\
&\leq 
\int \sum_{n_i=1}^{i}C(\gamma)\gamma^{s-i}\vert\nabla\nabla^{n_i-1}\gamma\vert 
\vert \nabla_A^{(2k-2-j+l)}\phi\vert 
\vert\nabla_A^{(j-2\alpha -1)}\nabla_A^{(l)}\phi\vert.
\end{align*}
We then bound this latter integral by using theorem \ref{interp_1}. We note
that in applying theorem \ref{interp_1}, we get a term involving $\gamma$, which
we will absorb into the constant $C(g, \gamma)$.
\begin{align*}
&\hspace{0.5cm}
\int \sum_{n_i=1}^{i}C(\gamma)\gamma^{s-i}\vert\nabla\nabla^{n_i-1}\gamma\vert 
\vert \nabla_A^{(2k-2-j+l)}\phi\vert 
\vert\nabla_A^{(j-2\alpha -1)}\nabla_A^{(l)}\phi\vert \\
&\leq
C(\gamma, g)\vert\vert\phi\vert\vert_{\infty}^{\frac{1}{k-1-\alpha + l}}\big{(}
\vert\vert\gamma^{(s-2\alpha -1)/2}\nabla_A^{(k-1-\alpha + l)}\phi\vert\vert_{L^2} + 
\vert\vert\phi\vert\vert_{L^2, \gamma>0}\big{)}^{\frac{2k+2l-2-2\alpha -1}{k-1-\alpha+l}} \\
&\leq
C(\gamma, g)K(\vert\vert\phi\vert\vert_{\infty})^{\frac{1}{k-1-\alpha + l}}\big{(}
\vert\vert\gamma^{(s-2\alpha -1)/2}\nabla_A^{(k-1-\alpha + l)}\phi\vert\vert_{L^2} + 
\vert\vert\phi\vert\vert_{L^2, \gamma>0} + 1\big{)}^{\frac{2k+2l-2-2\alpha -1}{k-1-\alpha+l}} \\
&\leq
C(\gamma, g)K(\vert\vert\phi\vert\vert_{\infty})\big{(}
\vert\vert\gamma^{(s-2\alpha -1)/2}\nabla_A^{(k-1-\alpha + l)}\phi\vert\vert_{L^2} + 
\vert\vert\phi\vert\vert_{L^2, \gamma>0} + 1\big{)}^{2} \\
&\leq 
C(\gamma, g)K(\vert\vert\phi\vert\vert_{\infty})\big{(}
\vert\vert\gamma^{(s-2\alpha -1)/2}\nabla_A^{(k-1-\alpha + l)}\phi\vert\vert^2_{L^2} + 
\vert\vert\phi\vert\vert^2_{L^2, \gamma>0} + 1\big{)}\\
&\leq C(\gamma, g)K(\vert\vert\phi\vert\vert_{\infty})
\tilde{\epsilon}_4\vert\vert\gamma^{s/2}\nabla_A^{k+1}\nabla_A^{(l)}\phi\vert\vert^2_{L^2} + 
C(\tilde{\epsilon}_4, g, \gamma)K(\vert\vert\phi\vert\vert_{\infty})\vert\vert\phi\vert\vert^2_{L^2, \gamma>0}.
\end{align*}
where to get the last inequality we have applied lemma \ref{interp_3}.
\end{itemize}

Substituting the two estimates carried out above back into \eqref{smoothing_est_3a}, we obtain
\begin{align*}
&\hspace{0.5cm}
\int \sum_{w=0}^{2k-2}\sum_{j=0}^w\sum_{i=0}^jC(g)\bigg{(}\sup_{t \in [0,T)}\vert\vert F_A\vert\vert_{\infty} 
\bigg{)} 
\vert\nabla_A^{(2k-2-j)}\nabla_A^{(l)}\phi\vert 
 \vert\nabla^{(i)}(\gamma^{s})\vert\vert\nabla_A^{(j-i)}\nabla_A^{(l)}\phi\vert \\
 &\leq
C(g, \gamma)\big{(}\sup_{t \in [0,T)}\vert\vert F_A\vert\vert_{\infty}\big{)}(\epsilon_4 + 
 K(\vert\vert\phi\vert\vert_{\infty})\tilde{\epsilon}_4)
 \vert\vert\gamma^{s/2}\nabla_A^{k+1}\nabla_A^{(l)}\phi\vert\vert^2_{L^2} \\
 &\hspace{1cm} +
C(g, \gamma)\big{(}\sup_{t \in [0,T)}\vert\vert F_A\vert\vert_{\infty}\big{)}(C(\epsilon_4, g, \gamma) + 
  C(\tilde{\epsilon}_4, g, \gamma)K(\vert\vert\phi\vert\vert_{\infty}))
  \vert\vert\phi\vert\vert^2_{L^2, \gamma>0}.
 \end{align*} 

Finally, we obtain the estimate
\begin{align}
&\hspace{0.5cm}
\int 2Re\big{(} \langle \sum_{w=0}^{2k-2}\nabla^{(w)}_MF_A * \nabla_A^{(2k-2-w)}\nabla_A^{(l)}\phi, 
 \gamma^{s}\nabla_A^{(l)}\phi\rangle\big{)} \label{smoothing_est_3b}\\
& \leq
C(g, \gamma)\big{(}\sup_{t \in [0,T)}\vert\vert F_A\vert\vert_{\infty}\big{)}(\epsilon_4 + 
 K(\vert\vert\phi\vert\vert_{\infty})\tilde{\epsilon}_4)
 \vert\vert\gamma^{s/2}\nabla_A^{k+1}\nabla_A^{(l)}\phi\vert\vert^2_{L^2} \nonumber\\
 &\hspace{1cm} +
C(g, \gamma)\big{(}\sup_{t \in [0,T)}\vert\vert F_A\vert\vert_{\infty}\big{)}(C(\epsilon_4, g, \gamma) + 
  C(\tilde{\epsilon}_4, g, \gamma)K(\vert\vert\phi\vert\vert_{\infty}))
  \vert\vert\phi\vert\vert^2_{L^2, \gamma>0}. \nonumber
 \end{align}

Combining the estimates, \eqref{smoothing_est_1a}, \eqref{smoothing_est_2a}, and \eqref{smoothing_est_3b}, we 
obtain
\begin{align*}
&\hspace{0.5cm}
\int -2Re\big{(}\langle \Delta_A^{(k+1)}\nabla_A^{(l)}\phi, \gamma^{s}\nabla_A^{(l)}\phi \rangle \big{)} \\
&\leq
\bigg{(} -2 + C(g, \gamma)(\epsilon_1 + \epsilon_2) + C(g)\epsilon_3 + 
C(g, \gamma)K(\vert\vert\phi\vert\vert_{\infty})\tilde{\epsilon}_3 \\
&\hspace{1cm} 
C(g, \gamma)\big{(}\sup_{t \in [0,T)}\vert\vert F_A\vert\vert_{\infty}\big{)}(\epsilon_4 + 
 K(\vert\vert\phi\vert\vert_{\infty})\tilde{\epsilon}_4)\bigg{)}
 \vert\vert\gamma^{s/2}\nabla_A^{k+1}\nabla_A^{(l)}\phi\vert\vert^2_{L^2} \\
&\hspace{0.5cm} +
\bigg{(}\frac{C(\epsilon_2, g, \gamma)}{\epsilon_1^2} + (C(\epsilon_3, g) + 
C(\tilde{\epsilon}_3, g, \gamma))K(\vert\vert\phi\vert\vert_{\infty}) \\
&\hspace{1cm} +
 C(g, \gamma)\big{(}\sup_{t \in [0,T)}\vert\vert F_A\vert\vert_{\infty}\big{)}(C(\epsilon_4, g, \gamma) + 
  C(\tilde{\epsilon}_4, g, \gamma)K(\vert\vert\phi\vert\vert_{\infty}))\bigg{)}
  \vert\vert\phi\vert\vert^2_{L^2, \gamma>0}
\end{align*}
which proves the lemma.

\end{proof}


The next lemma gives estimates for the next four terms in proposition \ref{derivative_L^2_spinor}

\begin{lem}\label{smoothing_spinor_2}
Assume $\sup_{t \in [0,T)}\vert\vert F_A\vert\vert_{\infty} < \infty$, and let 
$K(\vert\vert\phi\vert\vert_{\infty}) 
= max\{1, \sup_{t\in [0,T)}\vert\vert\phi\vert\vert_{\infty}\}$.
Suppose $\gamma$ is a bump function, and $s \geq 2(k+l)$.
Then for 
$\epsilon_5, \tilde{\epsilon}_5, \epsilon_6, \tilde{\epsilon}_6 > 0$ sufficiently small,   
we have the following estimate
\begin{align*}
&\hspace{0.5cm}
2Re\bigg{(}\langle \sum_{j=0}^{2k-2+l}\nabla^{(j)}_MRm * \nabla_A^{(2k-2+l-j)}\phi, 
\gamma^{s}\nabla_A^{(l)}\phi\rangle 
\bigg{)} 
+ 2Re\bigg{(}\langle\sum_{j=0}^{2k+l}\nabla^{(j)}_MRm * \nabla_A^{(2k+l-j)}\phi, 
\gamma^{s}\nabla_A^{(l)}\phi \rangle \bigg{)} \\
&\hspace{0.5cm}
 + 2Re\bigg{(}\langle \sum_{j=0}^{2k-2+l}\nabla^{(j)}_MF_A * \nabla_A^{(2k-2+l-j)}\phi, 
\gamma^{s}\nabla_A^{(l)}\phi\rangle\bigg{)} 
 + 
2Re\bigg{(}\langle \sum_{j=0}^{2k+l}\nabla^{(j)}_MF_A * \nabla_A^{(2k+l-j)}\phi, 
\gamma^{s}\nabla_A^{(l)}\phi\rangle\bigg{)} \\
&\leq
\big{(}C(g)\epsilon_5 + C(g, \gamma)K(\vert\vert\phi\vert\vert_{\infty})\tilde{\epsilon}_5\big{)}
\vert\vert\gamma^{s/2} \nabla_A^{(k+1)}\nabla_A^{(l)}\phi\vert\vert^2_{L^2} 
+ \big{(}C(\epsilon_5,g) + C(\tilde{\epsilon}_5, g, \gamma)K(\vert\vert\phi\vert\vert_{\infty})  
\big{)}\vert\vert\phi\vert\vert^2_{L^2, \gamma>0} \\
&\hspace{0.5cm} +
C(g, \gamma)\big{(}\sup_{t \in [0,T)}\vert\vert F_A\vert\vert_{\infty}\big{)}(\epsilon_6 + 
 K(\vert\vert\phi\vert\vert_{\infty})\tilde{\epsilon}_6)
 \vert\vert\gamma^{s/2}\nabla_A^{k+1}\nabla_A^{(l)}\phi\vert\vert^2_{L^2} \\
 &\hspace{1cm} +
C(g, \gamma)\big{(}\sup_{t \in [0,T)}\vert\vert F_A\vert\vert_{\infty}\big{)}(C(\epsilon_6, g, \gamma) + 
  C(\tilde{\epsilon}_6, g, \gamma)K(\vert\vert\phi\vert\vert_{\infty}))
  \vert\vert\phi\vert\vert^2_{L^2, \gamma>0}.
\end{align*}
where $C(g), C(g, \gamma), C(\epsilon_5, g, \gamma), C(\tilde{\epsilon}_{5}, g, \gamma), 
C(\epsilon_6, g, \gamma), C(\tilde{\epsilon}_6, g, \gamma)$ are constants that do not depend on $t \in [0, T)$.
\end{lem}

We won't give the proof of this lemma, as the four terms on the left of the inequality are exactly analogous
to the terms that turned up in the course of the proof of lemma \ref{smoothing_spinor_1}. Therefore, 
one needs only to apply exactly the same argument we did to obtain \eqref{smoothing_est_2a} and 
\eqref{smoothing_est_3b}.


\begin{lem}\label{smoothing_spinor_3}
Assume $\sup_{t \in [0,T)}\vert\vert F_A\vert\vert_{\infty} < \infty$, and let 
$K(\vert\vert\phi\vert\vert_{\infty}) 
= max\{1, \sup_{t\in [0,T)}\vert\vert\phi\vert\vert_{\infty}\}$.
Suppose $\gamma$ is a bump function, and $s \geq 2(k+l)$.
Then for 
$\epsilon_{7}, \tilde{\epsilon}_{7} > 0$ sufficiently small,   
we have the following estimate 
\begin{align*}
&\hspace{0.5cm}
\int 2Re\bigg{(}\langle \sum_{i=0}^{l-1}(-1)^{k+1}C_i\nabla_M^{(i)}d^*\Delta_M^{(k)}F_A \otimes 
\nabla_A^{(l-1-i)}\phi, \gamma^{s}\nabla_A^{(l)}\phi\rangle\bigg{)} \\
&\leq 
C(\gamma, g)\bigg{(}\sup_{t \in [0,T)}\vert\vert F_A\vert\vert_{\infty}\bigg{)}
\big{(}\epsilon_{7} + K(\vert\vert\phi\vert\vert_{\infty}))\tilde{\epsilon}_{7}\big{)}
 \vert\vert\gamma^{s/2}\nabla_A^{k+1}\nabla_A^{(l)}\phi\vert\vert^2_{L^2} \\
&\hspace{1cm} + 
\bigg{(}\sup_{t \in [0,T)}\vert\vert F_A\vert\vert_{\infty}\bigg{)}
\big{(}C(\epsilon_{7}, \gamma, g) + C(\tilde{\epsilon}_{7}, \gamma, g)K(\vert\vert\phi\vert\vert_{\infty})\big{)} 
\vert\vert\phi\vert\vert^2_{L^2, \gamma>0}.
\end{align*}
\end{lem}
\begin{proof}
We start by observing that, we can write
\begin{align*}
&\hspace{0.5cm}
\int \langle \sum_{i=0}^{l-1}(-1)^{k+1}C_i\nabla_M^{(i)}d^*\Delta_M^{(k)}F_A \otimes 
\nabla_A^{(l-1-i)}\phi, \gamma^{s}\nabla_A^{(l)}\phi\rangle \\
&=
\int \big{\langle} \sum_{i=0}^{l-1}(-1)^{k+1}C_i\nabla_M^{(i)}d^*\Delta_M^{(k)}F_A, 
\langle\gamma^{s}\nabla_A^{(l)}\phi, \nabla_A^{(l-1-i)}\phi\rangle\big{\rangle} \\
&= 
\int\sum_{i=0}^{l-1}(-1)^{k+1}C_i\big{\langle}\nabla_M^{(i)}d^*\Delta_M^{(k)}F_A, 
\langle\gamma^{s}\nabla_A^{(l)}\phi, \nabla_A^{(l-1-i)}\phi\rangle\big{\rangle} \\
&= 
\int\sum_{i=0}^{l-1}(-1)^{k+1}C_i\big{\langle}\Delta_M^{(k)}F_A, 
d\nabla_M^{*(i)}
\langle\gamma^{s}\nabla_A^{(l)}\phi, \nabla_A^{(l-1-i)}\phi\rangle\big{\rangle}.
\end{align*}
We then integrate, this latter integral, by parts to obtain
\begin{align*}
\int\sum_{i=0}^{l-1}(-1)^{k+1}C_i\big{\langle}\Delta_M^{(k)}F_A, 
d\nabla_M^{*(i)}
\langle\gamma^{s}\nabla_A^{(l)}\phi, \nabla_A^{(l-1-i)}\phi\rangle\big{\rangle} \\
=
\int\sum_{i=0}^{l-1}(-1)^{k+1}C_i\big{\langle}F_A, 
P_1^{(2k)}\big{(}d\nabla_M^{*(i)}
\langle\gamma^{s}\nabla_A^{(l)}\phi, \nabla_A^{(l-1-i)}\phi\rangle\big{)}\big{\rangle}.
\end{align*}
We then have the bound
\begin{align*}
&\hspace{0.5cm}
\bigg{\vert} \int\sum_{i=0}^{l-1}(-1)^{k+1}C_i\big{\langle}F_A, 
P_1^{(2k)}\big{(}d\nabla_M^{*(i)}
\langle\gamma^{s}\nabla_A^{(l)}\phi, \nabla_A^{(l-1-i)}\phi\rangle\big{)}\big{\rangle}
\bigg{\vert} \\
&\leq
\int\sum_{i=0}^{l-1}C_i\big{\vert}\big{\langle}F_A, 
P_1^{(2k)}\big{(}d\nabla_M^{*(i)}
\langle\gamma^{s}\nabla_A^{(l)}\phi, \nabla_A^{(l-1-i)}\phi\rangle\big{)}\big{\rangle}\big{\vert} \\
&\leq 
\int\sum_{i=0}^{l-1}C_i\bigg{(}\sup_{t \in [0,T)}\vert\vert F_A\vert\vert_{\infty}\bigg{)}
\big{\vert} 
P_1^{(2k)}\big{(}d\nabla_M^{*(i)}
\langle\gamma^{s}\nabla_A^{(l)}\phi, \nabla_A^{(l-1-i)}\phi\rangle\big{)}\big{\vert}.
\end{align*}
Observe that we can bound
\begin{align*}
\big{\vert} 
P_1^{(2k)}\big{(}d\nabla_M^{*(i)}
\langle\gamma^{s}\nabla_A^{(l)}\phi, \nabla_A^{(l-1-i)}\phi\rangle\big{)}\big{\vert} \leq
C(g)\big{\vert} 
\nabla_M^{(2k+i+1)}\langle\gamma^{s}\nabla_A^{(l)}\phi, \nabla_A^{(l-1-i)}\phi\rangle\big{\vert}.
\end{align*}
Applying lemma \ref{product_estimate}, we then obtain
\begin{align*}
&\hspace{0.5cm}
C(g)\big{\vert} 
\nabla_M^{(2k+i+1)}\langle\gamma^{s}\nabla_A^{(l)}\phi, \nabla_A^{(l-1-i)}\phi\rangle\big{\vert} \\
&\leq
\sum_{j=0}^{2k+i+1}C(g)\big{\vert} 
\langle\nabla_A^{(j)}\big{(}\gamma^{s}\nabla_A^{(l)}\phi\big{)}, 
\nabla_A^{(2k+i+1-j)}\nabla_A^{(l-1-i)}\phi\rangle\big{\vert} \\
&\leq
\sum_{j=0}^{2k+i+1}\sum_{n=0}^{j}C(g)\big{\vert} 
\langle\nabla_A^{(n)}\big{(}\gamma^{s}\big{)}\otimes\nabla_A^{(j-n)}\nabla_A^{(l)}\phi, 
\nabla_A^{(2k+i+1-j)}\nabla_A^{(l-1-i)}\phi\rangle\big{\vert} \\
&\leq
\sum_{j=0}^{2k+i+1}\sum_{n=0}^{j}C(g)\big{\vert} 
\nabla_A^{(n)}\big{(}\gamma^{s}\big{)}\big{\vert}
\big{\vert}\nabla_A^{(j-n)}\nabla_A^{(l)}\phi\big{\vert} 
\big{\vert}\nabla_A^{(2k+i+1-j)}\nabla_A^{(l-1-i)}\phi\big{\vert}.
\end{align*}

These computations show that we can estimate
\begin{align}
&\hspace{0.5cm}
\int 2Re\bigg{(}\langle \sum_{i=0}^{l-1}(-1)^{k+1}C_i\nabla_M^{(i)}d^*\Delta_M^{(k)}F_A \otimes 
\nabla_A^{(l-1-i)}\phi, \gamma^{s}\nabla_A^{(l)}\phi\rangle\bigg{)} \label{smoothing_spinor_3a}\\
&\leq
\int \sum_{i=0}^{l-1}\sum_{j=0}^{2k+i+1}\sum_{n=0}^{j}C(g)
\big{(}\sup_{t \in [0,T)}\vert\vert F_A\vert\vert_{\infty}\big{)}
\big{\vert} 
\nabla_A^{(n)}\big{(}\gamma^{s}\big{)}\big{\vert}
\big{\vert}\nabla_A^{(j-n)}\nabla_A^{(l)}\phi\big{\vert} 
\big{\vert}\nabla_A^{(2k+i+1-j)}\nabla_A^{(l-1-i)}\phi\big{\vert}. \nonumber
\end{align}
It therefore suffices to estimate the following integral
\begin{equation}
\int \sum_{i=0}^{l-1}\sum_{j=0}^{2k+i+1}\sum_{n=0}^{j}C(g)
\big{(}\sup_{t \in [0,T)}\vert\vert F_A\vert\vert_{\infty}\big{)}
\big{\vert} 
\nabla_A^{(n)}\big{(}\gamma^{s}\big{)}\big{\vert}
\big{\vert}\nabla_A^{(j-n)}\nabla_A^{(l)}\phi\big{\vert} 
\big{\vert}\nabla_A^{(2k+i+1-j)}\nabla_A^{(l-1-i)}\phi\big{\vert} \label{smoothing_spinor_3b}.
\end{equation}
In order to estimate this integral, we split the integral into two cases, depending on the parity of $n$.

We start by considering the case that $n$ is even. Write $n = 2\alpha$ for $\alpha \geq 0$. Then
\begin{align}
&\hspace{0.5cm}
\int \big{\vert} 
\nabla_A^{(n)}\big{(}\gamma^{s}\big{)}\big{\vert}
\big{\vert}\nabla_A^{(j-n)}\nabla_A^{(l)}\phi\big{\vert} 
\big{\vert}\nabla_A^{(2k+i+1-j)}\nabla_A^{(l-1-i)}\phi\big{\vert} \label{smoothing_spinor_3c}\\
&=
\int
\big{\vert} 
\nabla_A^{(2\alpha)}\big{(}\gamma^{s}\big{)}\big{\vert}
\big{\vert}\nabla_A^{(j-2\alpha)}\nabla_A^{(l)}\phi\big{\vert} 
\big{\vert}\nabla_A^{(2k+i+1-j)}\nabla_A^{(l-1-i)}\phi\big{\vert} \nonumber\\
&\leq 
\int C(\gamma)\gamma^{s-2\alpha}
\big{\vert}\nabla_A^{(j-n)}\nabla_A^{(l)}\phi\big{\vert} 
\big{\vert}\nabla_A^{(2k+i+1-j)}\nabla_A^{(l-1-i)}\phi\big{\vert} \nonumber\\
&\leq
C(\gamma)\big{(}\vert\vert\gamma^{(s-2\alpha)/2}\nabla_A^{(k-\alpha + l)}\phi\vert\vert_{L^2} + 
\vert\vert\phi\vert\vert_{L^2, \gamma>0}\big{)}^2 \nonumber\\
&\leq
C(\gamma)\big{(}\vert\vert\gamma^{(s-2\alpha)/2}\nabla_A^{(k-\alpha + l)}\phi\vert\vert^2_{L^2} + 
\vert\vert\phi\vert\vert^2_{L^2, \gamma>0}\big{)} \nonumber\\
&\leq
C(\gamma)\epsilon_7\vert\vert\gamma^{(s-2\alpha)/2}\nabla_A^{(k-\alpha + l)}\phi\vert\vert^2_{L^2} + 
C(\epsilon_{7}, \gamma)\vert\vert\phi\vert\vert^2_{L^2, \gamma>0} \nonumber
\end{align}
where to get from the second to the third line we have used theorem \ref{interp_1}, and to obtain the last 
inequality we have applied lemma \ref{interp_3}.

We then consider the case where $n$ is odd. Write $n = 2\alpha + 1$, for $\alpha \geq 0$. 
By lemma \ref{bump}, we can write
$\nabla^{(n)}\gamma^{s}$ as
\begin{equation*}
\nabla^{(n)}\gamma^{s} = \sum_{\substack{p_1+\cdots +p_n = n\\0\leq p_1\leq \cdots\leq p_n\leq n}}
C_{(p_1,\ldots,p_n)}(\gamma)\gamma^{s-i}\nabla_A^{p_1}*\cdots*\nabla_A^{p_n} 
\end{equation*}
and obtain the pointwise bound
\begin{align*}
\vert \nabla^{(n)}\gamma^{s} \vert &\leq 
\sum_{\substack{p_1+\cdots +p_n = n\\0\leq p_1\leq \cdots\leq p_n\leq n}}C(\gamma)\gamma^{s-n}
\vert\nabla^{p_1}\gamma\vert\cdots\vert\nabla^{p_{n-1}}\gamma\vert\vert\nabla^{p_n}\gamma\vert \\
&\leq 
\sum_{\substack{p_1+\cdots +p_n = n\\0\leq p_1\leq \cdots\leq p_n\leq n}}
C(\gamma)\gamma^{s-n}\vert\vert\nabla^{p_1}\gamma\vert\vert_{\infty}
\cdots\vert\vert\nabla^{p_{n-1}}\gamma\vert\vert_{\infty}
\vert\nabla\nabla^{p_n-1}\gamma\vert \\
&\leq 
\sum_{p_n=1}^{n}C(\gamma)\gamma^{s-n}\vert\nabla\nabla^{p_n-1}\gamma\vert
\end{align*}
where to get the last inequality we have absorbed the norms $\vert\vert\nabla^{p_q}\gamma\vert\vert_{\infty}$, for 
$1\leq q \leq n-1$, into the constant $C(\gamma)$.

We then estimate
\begin{align}
&\hspace{0.5cm}
\int
\big{\vert} 
\nabla^{(2\alpha +1)}\big{(}\gamma^{s}\big{)}\big{\vert}
\big{\vert}\nabla_A^{(j-2\alpha -1)}\nabla_A^{(l)}\phi\big{\vert} 
\big{\vert}\nabla_A^{(2k+i+1-j)}\nabla_A^{(l-1-i)}\phi\big{\vert} \label{smoothing_spinor_3d} \\
&\leq 
\int
C(\gamma)\gamma^{s-2\alpha - 1}
\big{\vert} 
\nabla\nabla^{p_n-1}\gamma\big{\vert}
\big{\vert}\nabla_A^{(j-2\alpha -1)}\nabla_A^{(l)}\phi\big{\vert} 
\big{\vert}\nabla_A^{(2k+i+1-j)}\nabla_A^{(l-1-i)}\phi\big{\vert} \nonumber\\
&\leq
C(\gamma)K(\vert\vert\phi\vert\vert_{\infty}))\big{(} 
\vert\vert\gamma^{(s-2\alpha - 1)/2}\nabla_A^{(k-\alpha +l)}\phi\vert\vert_{L^2} + 
\vert\vert\phi\vert\vert_{L^2, \gamma>0}\big{)}^{\frac{2(k-\alpha +l - 1)}{k+l-\alpha}} \nonumber\\
&\leq
C(\gamma)K(\vert\vert\phi\vert\vert_{\infty}))\big{(} 
\vert\vert\gamma^{(s-2\alpha - 1)/2}\nabla_A^{(k-\alpha +l)}\phi\vert\vert_{L^2} + 
\vert\vert\phi\vert\vert_{L^2, \gamma>0} + 1\big{)}^{\frac{2(k-\alpha +l - 1)}{k+l-\alpha}} \nonumber\\
&\leq
C(\gamma)K(\vert\vert\phi\vert\vert_{\infty}))\big{(} 
\vert\vert\gamma^{(s-2\alpha - 1)/2}\nabla_A^{(k-\alpha +l)}\phi\vert\vert_{L^2} + 
\vert\vert\phi\vert\vert_{L^2, \gamma>0} + 1 \big{)}^{2} \nonumber\\
&\leq
C(\gamma)K(\vert\vert\phi\vert\vert_{\infty}))\big{(} 
\vert\vert\gamma^{(s-2\alpha - 1)/2}\nabla_A^{(k-\alpha +l)}\phi\vert\vert^2_{L^2} + 
\vert\vert\phi\vert\vert^2_{L^2, \gamma>0} + 1\big{)} \nonumber\\
&\leq
C(\gamma)K(\vert\vert\phi\vert\vert_{\infty}))\tilde{\epsilon}_7 
\vert\vert\gamma^{s/2}\nabla_A^{(k+1)}\nabla_A^{(l)}\phi\vert\vert^2_{L^2} + 
C(\tilde{\epsilon}_7, \gamma)K(\vert\vert\phi\vert\vert_{\infty}))
\vert\vert\phi\vert\vert^2_{L^2, \gamma>0} \nonumber
\end{align}
where to get from the second line to the third line we apply theorem \ref{interp_1}, 
and to get the last line we apply lemma \ref{interp_3}.

Substituting the estimate we obtained for $n$ even, \eqref{smoothing_spinor_3c}, and the one for $n$ odd, 
\eqref{smoothing_spinor_3d}, into \eqref{smoothing_spinor_3b} gives
\begin{align*}
&\hspace{0.5cm}
\int \sum_{i=0}^{l-1}\sum_{j=0}^{2k+i+1}\sum_{n=0}^{j}C(g)
\big{(}\sup_{t \in [0,T)}\vert\vert F_A\vert\vert_{\infty}\big{)}
\big{\vert} 
\nabla_A^{(n)}\big{(}\gamma^{s}\big{)}\big{\vert}
\big{\vert}\nabla_A^{(j-n)}\nabla_A^{(l)}\phi\big{\vert} 
\big{\vert}\nabla_A^{(2k+i+1-j)}\nabla_A^{(l-1-i)}\phi\big{\vert} \\
&\leq
C(g, \gamma)\bigg{(}\sup_{t \in [0,T)}\vert\vert F_A\vert\vert_{\infty}\bigg{)}
\big{(}
\epsilon_7 + K(\vert\vert\phi\vert\vert_{\infty})\tilde{\epsilon}_7) 
\vert\vert\gamma^{s/2}\nabla_A^{(k+1)}\nabla_A^{(l)}\phi\vert\vert^2_{L^2} \\
&\hspace{0.5cm} +
\bigg{(}\sup_{t \in [0,T)}\vert\vert F_A\vert\vert_{\infty}\bigg{)}
\big{(}C(\epsilon_{7}, \gamma, g) + C(\tilde{\epsilon}_{7}, \gamma, g)K(\vert\vert\phi\vert\vert_{\infty})\big{)} 
\vert\vert\phi\vert\vert^2_{L^2, \gamma>0}.
\end{align*}
Using the estimate \eqref{smoothing_spinor_3a}, we then obtain the statement of the lemma.

\end{proof}


\begin{lem}\label{smoothing_spinor_4}
Assume $\sup_{t \in [0,T)}\vert\vert F_A\vert\vert_{\infty} < \infty$, and let 
$K(\vert\vert\phi\vert\vert_{\infty}) 
= max\{1, \sup_{t\in [0,T)}\vert\vert\phi\vert\vert_{\infty}\}$.
Suppose $\gamma$ is a bump function, and
$s \geq 2(k+l)$.
Then for 
$\epsilon_{8}, \tilde{\epsilon}_{8} > 0$ sufficiently small,   
we have the following estimate 
\begin{align*}
&\hspace{0.5cm}
\int 2Re\bigg{(}\langle  \sum_{i=0}^{l-1}\sum_{v=0}^{2k-1+i}P_1^{(v)}[F_A] \otimes 
\nabla_A^{(l-1-i)}\phi, \gamma^{s}\nabla_A^{(l)}\phi\rangle\bigg{)}\\
&\leq 
C(\gamma, g)\bigg{(}\sup_{t \in [0,T)}\vert\vert F_A\vert\vert_{\infty}\bigg{)}
\big{(}\epsilon_{8} + K(\vert\vert\phi\vert\vert_{\infty}))\tilde{\epsilon}_{8}\big{)}
 \vert\vert\gamma^{s/2}\nabla_A^{k+1}\nabla_A^{(l)}\phi\vert\vert^2_{L^2} \\
&\hspace{1cm} + 
\bigg{(}\sup_{t \in [0,T)}\vert\vert F_A\vert\vert_{\infty}\bigg{)}
\big{(}C(\epsilon_{8}, \gamma, g) + C(\tilde{\epsilon}_{8}, \gamma, g)K(\vert\vert\phi\vert\vert_{\infty})\big{)} 
\vert\vert\phi\vert\vert^2_{L^2, \gamma>0}.
\end{align*}
\end{lem}
\begin{proof}
The proof of this lemma proceeds in the same way to the proof of lemma \ref{smoothing_spinor_3}. We start
by writing
\begin{align*}
\int \langle  \sum_{i=0}^{l-1}\sum_{v=0}^{2k-1+i}P_1^{(v)}[F_A] \otimes 
\nabla_A^{(l-1-i)}\phi, \gamma^{s}\nabla_A^{(l)}\phi\rangle
=
\int\langle  \sum_{i=0}^{l-1}\sum_{v=0}^{2k-1+i}P_1^{(v)}[F_A], \langle 
\gamma^{s}\nabla_A^{(l)}\phi, \nabla_A^{(l-1-i)}\phi \rangle \rangle.
\end{align*}
Performing an integration by parts, we obtain
\begin{align*}
\int\langle  \sum_{i=0}^{l-1}\sum_{v=0}^{2k-1+i}P_1^{(v)}[F_A], \langle 
\gamma^{s}\nabla_A^{(l)}\phi, \nabla_A^{(l-1-i)}\phi \rangle \rangle
=
\int\sum_{i=0}^{l-1}\sum_{v=0}^{2k-1+i}(-1)^v\langle F_A, P_1^{(v)}\big{(}\langle 
\gamma^{s}\nabla_A^{(l)}\phi, \nabla_A^{(l-1-i)}\phi \rangle\big{)} \rangle.
\end{align*}
We then estimate 
\begin{align*}
&\hspace{0.5cm}
\bigg{\vert}
\int\sum_{i=0}^{l-1}\sum_{v=0}^{2k-1+i}(-1)^v\langle F_A, P_1^{(v)}\big{(}\langle 
\gamma^{s}\nabla_A^{(l)}\phi, \nabla_A^{(l-1-i)}\phi \rangle\big{)} \rangle
\bigg{\vert} \\
&\leq
\int\sum_{i=0}^{l-1}\sum_{v=0}^{2k-1+i}\big{\vert}\langle F_A, P_1^{(v)}\big{(}\langle 
\gamma^{s}\nabla_A^{(l)}\phi, \nabla_A^{(l-1-i)}\phi \rangle\big{)} \rangle\big{\vert} \\
&\leq
\int\sum_{i=0}^{l-1}\sum_{v=0}^{2k-1+i}\bigg{(}\sup_{t \in [0,T)}\vert\vert F_A\vert\vert_{\infty}\bigg{)} 
\big{\vert}P_1^{(v)}\big{(}\langle 
\gamma^{s}\nabla_A^{(l)}\phi, \nabla_A^{(l-1-i)}\phi\rangle\big{)}\big{\vert}.
\end{align*}
We have $\big{\vert}P_1^{(v)}\big{(}\langle 
\gamma^{s}\nabla_A^{(l)}\phi, \nabla_A^{(l-1-i)}\phi\rangle\big{)}\big{\vert} \leq
C(g)\big{\vert}\nabla_M^{(v)}\langle 
\gamma^{s}\nabla_A^{(l)}\phi, \nabla_A^{(l-1-i)}\phi\rangle\big{\vert}$. 
Applying lemma \ref{product_estimate}, we get
\begin{align*}
\big{\vert}\nabla^{(v)}\langle 
\gamma^{s}\nabla_A^{(l)}\phi, \nabla_A^{(l-1-i)}\phi\rangle\big{\vert} &\leq 
\sum_{n=0}^{v}C(g)
\big{\vert}\langle
\nabla_A^{(n)}\big{(} 
\gamma^{s}\nabla_A^{(l)}\phi\big{)}, \nabla_A^{(v-n)}\nabla_A^{(l-1-i)}\phi\rangle\big{\vert} \\
&\leq 
\sum_{n=0}^{v}\sum_{m=0}^{n}C(g)
\big{\vert}\langle
\nabla^{(m)}(\gamma^{s}) \otimes \nabla_A^{(n-m)} 
\nabla_A^{(l)}\phi, \nabla_A^{(v-n)}\nabla_A^{(l-1-i)}\phi\rangle\big{\vert} \\
&\leq
\sum_{n=0}^{v}\sum_{m=0}^{n}C(g)
\big{\vert}
\nabla^{(m)}(\gamma^{s})\big{\vert} \big{\vert} \nabla_A^{(n-m)} 
\nabla_A^{(l)}\phi\big{\vert} \big{\vert}\nabla_A^{(v-n)}\nabla_A^{(l-1-i)}\phi\big{\vert}. 
\end{align*}  
Therefore, we get the estimate
\begin{align}
&\hspace{0.5cm}
\int 2Re\bigg{(}\langle  \sum_{i=0}^{l-1}\sum_{v=0}^{2k-1+i}P_1^{(v)}[F_A] \otimes 
\nabla_A^{(l-1-i)}\phi, \gamma^{s}\nabla_A^{(l)}\phi\rangle\bigg{)} \label{smoothing_spinor_4a}\\
&\leq
\int\sum_{i=0}^{l-1}\sum_{v=0}^{2k-1+i} 
\sum_{n=0}^{v}\sum_{m=0}^{n}C(g)
\bigg{(}\sup_{t \in [0,T)}\vert\vert F_A\vert\vert_{\infty}\bigg{)}
\big{\vert}
\nabla^{(m)}(\gamma^{s})\big{\vert} \big{\vert} \nabla_A^{(n-m)} 
\nabla_A^{(l)}\phi\big{\vert} \big{\vert}\nabla_A^{(v-n)}\nabla_A^{(l-1-i)}\phi\big{\vert}. \nonumber
\end{align}
The way to proceed to evaluate an estimate for the above integral is to apply the same technique we used
in proving lemma \ref{smoothing_spinor_3}. That is, 
we need to set up the integral in a form
for which theorem \ref{interp_1} is applicable. In order to 
apply theorem \ref{interp_1}, we first note that the sum of
the exponents of the derivatives of the spinor field $\phi$ is $(n-m+l) + (v-n+l-1-i) = 2l + v-m-i-1$.
Therefore we split the integral into two parts, $v-m-i-1$ is even and $v-m-i-1$ is odd. The proof
in each case exactly follows what we did in the proof of lemma \ref{smoothing_spinor_3}, see the proof of
\eqref{smoothing_spinor_3c}, and the proof of \eqref{smoothing_spinor_3d}. Due to this, we will just state
the final result of applying that technique.

\begin{itemize}
\item[1.] When $v-m-i-1$ is even, we obtain the estimate
\begin{align*}
&\hspace{0.5cm}
\int \nabla^{(m)}(\gamma^{s})\big{\vert} \big{\vert} \nabla_A^{(n-m)} 
\nabla_A^{(l)}\phi\big{\vert} \big{\vert}\nabla_A^{(v-n)}\nabla_A^{(l-1-i)}\phi\big{\vert} \\
&\leq 
C(\gamma)\epsilon_8\vert\vert\gamma^{s/2}\nabla_A^{k+1}\nabla_A^{(l)}\phi\vert\vert^2_{L^2}  
+
C(\epsilon_8, \gamma)\vert\vert\phi\vert\vert^2_{L^2, \gamma>0}.
\end{align*}

\item[2.] When $v-m-i-1$ is odd, we obtain the estimate
\begin{align*}
&\hspace{0.5cm}
\int \nabla^{(m)}(\gamma^{s})\big{\vert} \big{\vert} \nabla_A^{(n-m)} 
\nabla_A^{(l)}\phi\big{\vert} \big{\vert}\nabla_A^{(v-n)}\nabla_A^{(l-1-i)}\phi\big{\vert} \\
&\leq 
C(\gamma)
K(\vert\vert\phi\vert\vert_{\infty}))
\tilde{\epsilon}_8\vert\vert\gamma^{s/2}\nabla_A^{k+1}\nabla_A^{(l)}\phi\vert\vert^2_{L^2}  
+
C(\tilde{\epsilon}_8, \gamma)
K(\vert\vert\phi\vert\vert_{\infty}))
\vert\vert\phi\vert\vert^2_{L^2, \gamma>0}.
\end{align*}
\end{itemize}

Substituting these two estimates into   \eqref{smoothing_spinor_4a} we obtain the statement of the lemma.

\end{proof}


\begin{lem}\label{smoothing_spinor_5}
Let 
$K(\vert\vert\phi\vert\vert_{\infty}) 
= max\{1, \sup_{t\in [0,T)}\vert\vert\phi\vert\vert_{\infty}\}$.
Suppose $\gamma$ is a bump function, and $s \geq 2(k+l)$.
Then for 
$\epsilon_{9} > 0$ sufficiently small,   
we have the following estimate 
\begin{align*}
&\hspace{0.5cm}
\int 2Re\bigg{(}\langle - 2iIm\bigg{(} 
\sum_{j=0}^{l-1}\sum_{i=1}^kC_i\nabla_M^{(j)}\nabla_M^{*(i)}\langle\nabla_A^{(k)}\nabla_A\phi, 
\nabla_A^{(k-i)}\phi\rangle\bigg{)}\otimes \nabla_A^{(l-1-j)}\phi, \gamma^{s}\nabla_A^{(l)}\phi\rangle\bigg{)} \\
&\leq
C(g)K(\vert\vert\phi\vert\vert_{\infty}))\epsilon_9
\vert\vert\gamma^{s/2}\nabla_A^{k+1}\nabla_A^{(l)}\phi\vert\vert^2_{L^2} 
+
C(\epsilon_{9}, g)K(\vert\vert\phi\vert\vert_{\infty}))
\vert\vert\phi\vert\vert^2_{L^2, \gamma>0}.
\end{align*}
\end{lem}
\begin{proof}
We have 
\begin{align*}
&\hspace{0.5cm}
\int 2Re\bigg{(}\langle - 2iIm\bigg{(} 
\sum_{j=0}^{l-1}\sum_{i=1}^kC_i\nabla_M^{(j)}\nabla_M^{*(i)}\langle\nabla_A^{(k)}\nabla_A\phi, 
\nabla_A^{(k-i)}\phi\rangle\bigg{)}\otimes \nabla_A^{(l-1-j)}\phi, \gamma^{s}\nabla_A^{(l)}\phi\rangle\bigg{)} \\
&=
\int 2Re\bigg{(}\langle - 2iIm\bigg{(} 
\sum_{j=0}^{l-1}\sum_{i=1}^kC_i\nabla_M^{(j)}\nabla_M^{*(i)}\langle\nabla_A^{(k)}\nabla_A\phi, 
\nabla_A^{(k-i)}\phi\rangle\bigg{)}, \langle\gamma^{s}\nabla_A^{(l)}\phi, \nabla_A^{(l-1-j)}\phi\rangle
\rangle\bigg{)}.
\end{align*}
We estimate
\begin{align*}
&\hspace{0.5cm}
\int \big{\vert}\langle - 2iIm\bigg{(} 
\sum_{j=0}^{l-1}\sum_{i=1}^kC_i\nabla_M^{(j)}\nabla_M^{*(i)}\langle\nabla_A^{(k)}\nabla_A\phi, 
\nabla_A^{(k-i)}\phi\rangle\bigg{)}, \langle\gamma^{s}\nabla_A^{(l)}\phi, \nabla_A^{(l-1-j)}\phi\rangle
\rangle\big{\vert} \\
&\leq
\int \sum_{j=0}^{l-1}\sum_{i=1}^k2\gamma^{s}\big{\vert}
\langle
\nabla_M^{(j)}\nabla_M^{*(i)}\langle\nabla_A^{(k)}\nabla_A\phi, 
\nabla_A^{(k-i)}\phi\rangle
\big{\vert}
 \big{\vert}
 \langle\nabla_A^{(l)}\phi, \nabla_A^{(l-1-j)}\phi\rangle
 \big{\vert}.
\end{align*}
Applying lemma \ref{product_estimate}, we then have
\begin{align*}
&\hspace{0.5cm}
\int \sum_{j=0}^{l-1}\sum_{i=1}^k2\gamma^{s}\big{\vert}
\langle
\nabla_M^{(j)}\nabla_M^{*(i)}\langle\nabla_A^{(k)}\nabla_A\phi, 
\nabla_A^{(k-i)}\phi\rangle
\big{\vert}
 \big{\vert}
 \langle\nabla_A^{(l)}\phi, \nabla_A^{(l-1-j)}\phi\rangle
 \big{\vert} \\
&\leq
 \int \sum_{j=0}^{l-1}\sum_{i=1}^k\sum_{n=0}^{i+j}C(g)
\big{\vert}\langle 
 \nabla_A^{(n)}\nabla_A^{(k)}\nabla_A\phi, \nabla_A^{(i+j-n)}\nabla_A^{(k-i)}\phi\rangle\big{\vert}
 \big{\vert} \langle\nabla_A^{(l)}\phi, \nabla_A^{(l-1-j)}\phi\rangle
 \big{\vert}.
 \end{align*}
Applying theorem \ref{interp_1}, followed by lemma \ref{interp_3}, we obtain
\begin{align*}
&\hspace{0.5cm}
\int \sum_{j=0}^{l-1}\sum_{i=1}^k\sum_{n=0}^{i+j}C(g)
\big{\vert}\langle 
 \nabla_A^{(n)}\nabla_A^{(k)}\nabla_A\phi, \nabla_A^{(i+j-n)}\nabla_A^{(k-i)}\phi\rangle\big{\vert}
 \big{\vert} \langle\nabla_A^{(l)}\phi, \nabla_A^{(l-1-j)}\phi\rangle
 \big{\vert} \\
&\leq
C(g)K(\vert\vert\phi\vert\vert_{\infty})^2\big{(}
\vert\vert\gamma^{s/2}\nabla_A^{k+1}\nabla_A^{(l)}\phi\vert\vert^2_{L^2}  +
\vert\vert\phi\vert\vert^2_{L^2, \gamma>0}\big{)} \\
&\leq
C(g)K(\vert\vert\phi\vert\vert_{\infty})^2\epsilon_{9}
\vert\vert\gamma^{s/2}\nabla_A^{k+1}\nabla_A^{(l)}\phi\vert\vert^2_{L^2}  +
C(g)K(\vert\vert\phi\vert\vert_{\infty})^2C(\epsilon_{9}, g)
\vert\vert\phi\vert\vert^2_{L^2, \gamma>0}
\end{align*}
and the lemma follows.

\end{proof}


\begin{lem}\label{smoothing_spinor_6}
Assume $\sup_{t \in [0,T)}\vert\vert F_A\vert\vert_{\infty} < \infty$, and let 
$K(\vert\vert\phi\vert\vert_{\infty}) 
= max\{1, \sup_{t\in [0,T)}\vert\vert\phi\vert\vert_{\infty}\}$.
Suppose $\gamma$ is a bump function, and
$s \geq 2(k+l)$.
Then for 
$\epsilon_{10}, \tilde{\epsilon}_{10} > 0$ sufficiently small,   
we have the following estimate 
\begin{align*}
&\hspace{0.5cm}
\int 
- \frac{1}{2}Re\bigg{(}\langle  \nabla_A^{(l)}\big{(}(S + \vert\phi\vert^2)\phi\big{)}, 
\gamma^{s}\nabla_A^{(l)}\phi\rangle\bigg{)} \\
&\leq
C(g, \gamma)K(\vert\vert\phi\vert\vert_{\infty})^3
(\epsilon_{10} + K(\vert\vert\phi\vert\vert_{\infty})\tilde{\epsilon_{10}})
\vert\vert\gamma^{s/2}\nabla_A^{k+1}\nabla_A^{(l)}\phi\vert\vert^2_{L^2}  \\
&\hspace{1cm} +
K(\vert\vert\phi\vert\vert_{\infty})^{3}
\big{(} C(\epsilon_{10}, g, \gamma) + K(\vert\vert\phi\vert\vert_{\infty})C(\tilde{\epsilon}_{10}, g, \gamma)\big{)}
\vert\vert\phi\vert\vert^2_{L^2, \gamma>0}.
\end{align*}
\end{lem}
\begin{proof}
We start by observing that
\begin{align*}
\int 
- \frac{1}{2}\langle  \nabla_A^{(l)}\big{(}(S + \vert\phi\vert^2)\phi\big{)}, 
\gamma^{s}\nabla_A^{(l)}\phi\rangle
=
\int 
- \frac{1}{2}\langle (S + \vert\phi\vert^2)\phi, 
 \nabla_A^{*(l)}\big{(}\gamma^{s}\nabla_A^{(l)}\phi\big{)}\rangle.
\end{align*}
We can then bound
\begin{align*}
\bigg{\vert}\int 
- \frac{1}{2}\langle (S + \vert\phi\vert^2)\phi, 
 \nabla_A^{*(l)}\big{(}\gamma^{s}\nabla_A^{(l)}\phi\big{)}\rangle\bigg{\vert} 
 &\leq
 \int C(g)K(\vert\vert\phi\vert\vert_{\infty})^{3}
 \vert \nabla_A^{(l)}\big{(}\gamma^{s}\nabla_A^{(l)}\phi\big{)}\vert \\
&\leq
\int \sum_{n=0}^{l}C(g)K(\vert\vert\phi\vert\vert_{\infty})^{3}
 \vert \nabla^{(n)}\gamma^{s}\vert\vert\nabla_A^{(l-n)}\nabla_A^{(l)}\phi\vert. 
  \end{align*}
We estimate this latter integral by splitting the integrand up into two parts, 
$n$ even and $n$ odd.  
The proof then proceeds analogously to what was done in the proof of 
lemma \ref{smoothing_spinor_3}. For details,  
see the proofs for 
\eqref{smoothing_spinor_3c} and  \eqref{smoothing_spinor_3d}. 
 
\end{proof}


Using the above lemmas, we can prove the following local $L^2$-derivative estimate.

\begin{thm}\label{spinor_smoothing_final}
Let $(\phi(t), A(t))$ be a solution to the higher order Seiberg-Witten flow.
Assume $Q(\vert\vert F_A\vert\vert_{\infty}) = \sup_{t \in [0,T)}\vert\vert F_A\vert\vert_{\infty} < \infty$, and 
let $K(\vert\vert\phi\vert\vert_{\infty}) 
= max\{1, \sup_{t\in [0,T)}\vert\vert\phi\vert\vert_{\infty}\}$.
Suppose $\gamma$ is a bump function, and
$s \geq 2(k+l)$.
Then
\begin{equation}
\frac{\partial}{\partial t}\vert\vert \gamma^{s/2}\nabla_A^{(l)}\phi\vert\vert^2_{L^2} 
\leq 
-\lambda\vert\vert\gamma^{s/2}\nabla_A^{k+1}\nabla_A^{(l)}\phi\vert\vert^2_{L^2}
+
C_s\big{(}Q(\vert\vert F_A\vert\vert_{\infty}), K(\vert\vert\phi\vert\vert_{\infty}), g, \gamma\big{)}
\vert\vert\phi\vert\vert^2_{L^2, \gamma>0}
\end{equation} 
where $1 \leq \lambda < 2$
\end{thm} 
\begin{proof}
Taking $0 < \epsilon = \epsilon_1 = \epsilon_2 = \epsilon_3 = \tilde{\epsilon}_3 = \ldots = 
\epsilon_{8} = \tilde{\epsilon}_{8} = \epsilon_9$ in lemmas 
\ref{smoothing_spinor_1}, \ref{smoothing_spinor_2}, \ref{smoothing_spinor_3}, \ref{smoothing_spinor_4}, 
\ref{smoothing_spinor_5}, \ref{smoothing_spinor_6}, and then using proposition \ref{derivative_L^2_spinor}. 
We see that we have a bound of the form
\begin{align*}
\frac{\partial}{\partial t}\vert\vert \gamma^{s/2}\nabla_A^{(l)}\phi\vert\vert^2_{L^2} 
&\leq
\big{(} -2 + C_1\big{(}Q(\vert\vert F_A\vert\vert_{\infty}), K(\vert\vert\phi\vert\vert_{\infty}), g, \gamma \big{)}
\epsilon\big{)}
\vert\vert\gamma^{s/2}\nabla_A^{k+1}\nabla_A^{(l)}\phi\vert\vert^2_{L^2} \\
&\hspace{1cm} +
C_2\big{(}\epsilon, Q(\vert\vert F_A\vert\vert_{\infty}), K(\vert\vert\phi\vert\vert_{\infty}), g, \gamma\big{)}
\vert\vert\phi\vert\vert^2_{L^2, \gamma>0}.
\end{align*}  
By choosing $\epsilon$ sufficiently small, we can make it so that
\begin{equation*}
1 \leq 
2 - C_1\big{(}Q(\vert\vert F_A\vert\vert_{\infty}), K(\vert\vert\phi\vert\vert_{\infty}), g, \gamma \big{)}
\epsilon 
< 2.
\end{equation*}
Taking $\lambda$ to be any such  
$2 - C_1\big{(}Q(\vert\vert F_A\vert\vert_{\infty}), K(\vert\vert\phi\vert\vert_{\infty}), g, \gamma \big{)}
\epsilon $, and defining
\begin{equation*}
C_s\big{(}Q(\vert\vert F_A\vert\vert_{\infty}), K(\vert\vert\phi\vert\vert_{\infty}), g, \gamma\big{)}
=
C_2\big{(}\epsilon, Q(\vert\vert F_A\vert\vert_{\infty}), K(\vert\vert\phi\vert\vert_{\infty}), g, \gamma\big{)},
\end{equation*}
we arrive at the statement of the theorem.

\end{proof}

The following corollary follows from integrating the above inequality in time.

\begin{cor}\label{spinor_smoothing_large_ time}
Suppose $(\phi(t), A(t))$ is a solution to the higher order Seiberg-Witten flow, on the time interval
$[0, T)$, where $T < \infty$, with the same assumptions as the above theorem. Then
\begin{align*}
\vert\vert\gamma^{s/2}\nabla_{A(t)}^{(l)}\phi(t)\vert\vert^2_{L^2} \leq 
TC
\sup_{t\in [0,T)}\bigg{(}\vert\vert\phi\vert\vert^2_{L^2, \gamma>0}\bigg{)}
\end{align*}
where $C$ depends on 
$C_s\big{(}Q(\vert\vert F_A\vert\vert_{\infty}), K(\vert\vert\phi\vert\vert_{\infty}), g, \gamma\big{)}$ and 
the initial condition $(\phi(0), A(0))$.
\end{cor}

\subsection{Estimates for derivatives of the curvature}

In this subsection, we establish local $L^2$- derivative estimates for the curvature form.

\begin{prop}\label{derivative_L^2_curvature}
Let $(\phi(t), A(t))$ be a solution to the higher order Seiberg-Witten flow. Then
\begin{align*}
\frac{\partial}{\partial t}\vert\vert \gamma^{s/2}\nabla_M^{(l)}F_A\vert\vert^2_{L^2} 
&\leq
\int 2Re\big{(}\langle (-1)^{k+1}\Delta_M^{k+1}\nabla_M^{(l)}F_{A}, \gamma^{s}\nabla_M^{(l)}F_A\rangle\big{)} 
+
2Re\bigg{(}\langle  \sum_{v = 0}^{2k+l}P_1[F_{A(t)}], \gamma^{s}\nabla_M^{(l)}F_A\rangle\bigg{)} \\
&\hspace{0.5cm} +
2Re\bigg{(} \langle -
2iIm\bigg{(} \sum_{i=1}^kC_i\nabla_M^{(l)}d\nabla_M^{*(i)}\langle\nabla_A^{(k)}\nabla_A\phi, 
\nabla_A^{(k-i)}\phi\rangle\bigg{)},  \gamma^{s}\nabla_M^{(l)}F_A\rangle\bigg{)}.
\end{align*}
\end{prop}

The above proposition immediately follows from corollary \ref{derivative_curvature}.

In order to obtain local $L^2$-derivative estimates for derivatives of the curvature form $F_A$, associated to the 
solution $(\phi(t), A(t))$. We will proceed as we did for the case of the spinor field. That is, we will
start by stating a string of lemmas that give estimates for the right hand side of the above
proposition. These estimates will then suffice to prove a general local estimate. Many of the 
proofs will follow the exact same techniques that was used in obtaining such estimates for the
spinor field. Due to this, we won't give details but rather refer the reader to those proofs. 

\begin{lem}\label{smoothing_curvature_1}
Assume $\sup_{t \in [0,T)}\vert\vert F_A\vert\vert_{\infty} < \infty$, and
suppose $\gamma$ is a bump function, and
$s \geq 2(k+l)$.
Then for 
$\epsilon_1, \epsilon_2, \epsilon_3, \tilde{\epsilon}_{3}, \epsilon_4, \tilde{\epsilon}_4 > 0$ sufficiently small,   
we have the following estimate
\begin{align*}
&\hspace{0.5cm}
\int 2Re\big{(}\langle (-1)^{k+1}\Delta_M^{k+1}\nabla_M^{(l)}F_{A}, \gamma^{s}\nabla_M^{(l)}F_A\rangle\big{)} \\
&\leq
\bigg{(} -2 + C(g, \gamma)(\epsilon_1 + \epsilon_2) + C(g)\epsilon_3 + 
C(g, \gamma)\tilde{\epsilon}_3 \\
&\hspace{1cm} 
C(g, \gamma)\big{(}\sup_{t \in [0,T)}\vert\vert F_A\vert\vert_{\infty}\big{)}(\epsilon_4 + 
 \tilde{\epsilon}_4)\bigg{)}
 \vert\vert\gamma^{s/2}\nabla_M^{k+1}\nabla_M^{(l)}F_A\vert\vert^2_{L^2} \\
&\hspace{0.5cm} +
\bigg{(}\frac{C(\epsilon_2, g, \gamma)}{\epsilon_1^2} + C(\epsilon_3, g) + 
C(\tilde{\epsilon}_3, g, \gamma) \\
&\hspace{0.5cm} +
 C(g, \gamma)\big{(}\sup_{t \in [0,T)}\vert\vert F_A\vert\vert_{\infty}\big{)}(C(\epsilon_4, g, \gamma) + 
  C(\tilde{\epsilon}_4, g, \gamma))\bigg{)}
  \vert\vert F_A\vert\vert^2_{L^2, \gamma>0}
\end{align*}
where $C(g), C(g, \gamma), C(\epsilon_2, g, \gamma), C(\epsilon_3, g, \gamma), C(\tilde{\epsilon}_{3}, g, \gamma), 
C(\epsilon_4, g, \gamma), C(\tilde{\epsilon}_4, g, \gamma)$ are constants that do not depend on $t \in [0, T)$.
\end{lem}

The proof of the above lemma is exactly analogous to how we proved lemma \ref{smoothing_spinor_1}. One simply
replaces the term $\phi$, in lemma \ref{smoothing_spinor_1}, with $F_A$, and then proceeds in exactly the same way.
Therefore, we won't give details of the proof, and refer the interested reader to
lemma \ref{smoothing_spinor_1}.

\begin{lem}\label{smoothing_curvature_2}
Suppose $\gamma$ is a bump function, and
$s \geq 2(k+l)$.
For $\epsilon_{5}, \tilde{\epsilon}_{5} > 0$ sufficiently small, we have the following estimate
\begin{align*}
&\hspace{0.5cm}
\int
2Re\bigg{(}\langle  \sum_{v = 0}^{2k+l}P_1^{(v)}[F_{A}], \gamma^{s}\nabla_M^{(l)}F_A\rangle\bigg{)} \\
&\leq
\big{(}C(g, \gamma)\epsilon_{5} + C\tilde{\epsilon}_{5}\big{)}
\vert\vert\gamma^{s/2}\nabla_M^{k+1}\nabla_M^{(l)}F_A\vert\vert^2_{L^2} 
+
\frac{C(\tilde{\epsilon}_{5}, g, \gamma)}{\epsilon_5^2}\vert\vert F_A\vert\vert^2_{L^2, \gamma>0}.
\end{align*}
\end{lem}
\begin{proof}
We start by writing
\begin{align*}
\int
\langle  \sum_{v = 0}^{2k+l}P_1^{(v)}[F_{A}], \gamma^{s}\nabla_M^{(l)}F_A\rangle 
=
\int
\sum_{v = 0}^{k+l+1}\langle P_1^{(v)}[F_{A}], \gamma^{s}\nabla_M^{(l)}F_A\rangle 
+
 \sum_{v = k+l+2}^{2k+l}\langle P_1^{(v)}[F_{A}], \gamma^{s}\nabla_M^{(l)}F_A\rangle. 
\end{align*}
Estimating the first time on the right, we have
\begin{align*}
\int \sum_{v = 0}^{2k+l}\big{\vert} \langle P_1^{(v)}[F_{A}], \gamma^{s}\nabla_M^{(l)}F_A\rangle  \big{\vert} 
&\leq
\int \sum_{v = 0}^{2k+l}C(g)\vert\nabla_M^{(v)}F_A\vert\vert \gamma^{s}\nabla_M^{(l)}F_A\vert \\
&=
\int \sum_{v = 0}^{2k+l}C(g)\vert\gamma^{s/2}\nabla_M^{(v)}F_A\vert\vert \gamma^{s/2}\nabla_M^{(l)}F_A\vert.
\end{align*}
For $v = k+l+1$, by applying Young's inequality, we obtain
\begin{align*}
\int C(g)\vert\gamma^{s/2}\nabla_M^{(k+l+1)}F_A\vert\vert \gamma^{s/2}\nabla_M^{(l)}F_A\vert
&\leq
\int C(g)\big{(}\epsilon_{5}\vert\gamma^{s/2}\nabla_M^{(v)}F_A\vert^2 + 
\frac{1}{\epsilon_{5}} \vert \gamma^{s/2}\nabla_M^{(l)}F_A\vert^2\big{)} \\
&=
C(g)\epsilon_5\vert\vert\gamma^{s/2}\nabla_M^{k+1}\nabla_M^{(l)}F_A\vert\vert^2_{L^2} +
\frac{C(g)}{\epsilon_5}\vert\vert\gamma^{s/2}\nabla_M^{(l)}F_A\vert\vert^2_{L^2}. 
\end{align*}
For $\epsilon_5$ sufficiently small, we have that $\frac{C(g)}{\epsilon_5} \geq 1$. Therefore, applying
lemma \ref{interp_3} to the term 
$\frac{C(g)}{\epsilon_5}\vert\vert\gamma^{s/2}\nabla_M^{(l)}F_A\vert\vert^2_{L^2}$, we get the following
estimate
\begin{align*}
C(g)\epsilon_5\vert\vert\gamma^{s/2}\nabla_M^{k+1}\nabla_M^{(l)}F_A\vert\vert^2_{L^2} +
\frac{C(g)}{\epsilon_5}\vert\vert\gamma^{s/2}\nabla_M^{(l)}F_A\vert\vert^2_{L^2}
&\leq
\big{(}C(g)\epsilon_5 + 
\tilde{\epsilon}_5\big{)}\vert\vert\gamma^{s/2}\nabla_M^{k+1}\nabla_M^{(l)}F_A\vert\vert^2_{L^2} \\
&\hspace{0.5cm}+
\frac{C(\tilde{\epsilon}_5, g)}{\epsilon_5^2}\vert\vert F_A\vert\vert^2_{L^2, \gamma>0}.
\end{align*}
For the case that $0\leq v\leq k+l$, we have
\begin{align*}
\int C(g)\vert\gamma^{s/2}\nabla_M^{(v)}F_A\vert\vert \gamma^{s/2}\nabla_M^{(l)}F_A\vert
&\leq
\int C(g)\big{(} 
\vert\gamma^{s/2}\nabla_M^{(v)}F_A\vert^2 + \vert\gamma^{s/2}\nabla_M^{(l)}F_A\vert^2\big{)} \\
&\leq
C(g)\epsilon_5
\vert\vert\gamma^{s/2}\nabla_M^{k+1}\nabla_M^{(l)}F_A\vert\vert^2_{L^2} + 
C(\epsilon_{5}, g)\vert\vert F_A\vert\vert^2_{L^2, \gamma>0}
\end{align*}
where in order to get the second line, we have applied lemma \ref{interp_3}
to both terms on the right hand side of the first line.

The next step is to estimate the term 
$\int \sum_{v = k+l+2}^{2k+l}\langle P_1^{(v)}[F_{A}], \gamma^{s}\nabla_M^{(l)}F_A\rangle$. We write this as
\begin{align*}
\int \sum_{v = k+l+2}^{2k+l}\langle P_1^{(v)}[F_{A}], \gamma^{s}\nabla_M^{(l)}F_A\rangle 
&=
\int \sum_{j = 1}^{k-1}\langle P_1^{(k+l+1+j)}[F_{A}], \gamma^{s}\nabla_M^{(l)}F_A\rangle \\
&=
\int
\sum_{j = 1}^{k-1}(-1)^j\langle P_1^{(k+l+1)}[F_{A}], P_1^{(j)}(\gamma^{s}\nabla_M^{(l)}F_A)\rangle
\end{align*}
where the second equality follows from integrating by parts.

We estimate
\begin{align*}
&\hspace{0.5cm}
\int \big{\vert}\langle P_1^{(k+l+1)}[F_{A}], P_1^{(j)}(\gamma^{s}\nabla_M^{(l)}F_A)\rangle\big{\vert} \\
&\leq
\int C(g) \big{\vert}\nabla_M^{(k+l+1)}F_{A}\big{\vert} 
\big{\vert}\nabla_M^{(j)}(\gamma^{s}\nabla_M^{(l)}F_A)\big{\vert} \\
&\leq
\int \sum_{i=0}^j C(g)
\big{\vert}\nabla_M^{(k+l+1)}F_{A}\big{\vert} 
\big{\vert}\nabla^{(i)}(\gamma^{s})\otimes\nabla_M^{(j-i)}\nabla_M^{(l)}F_A)\big{\vert} \\
&\leq
\int \sum_{i=0}^j C(g, \gamma)
\big{\vert}\gamma^{s/2}\nabla_M^{(k+l+1)}F_{A}\big{\vert}
\big{\vert}\gamma^{(s-2i)/2}\nabla_M^{(l+j-i)}F_{A}\big{\vert} \\
&\leq
\int \sum_{i=0}^j C(g, \gamma)
\big{(} \epsilon_5\big{\vert}\gamma^{s/2}\nabla_M^{(k+l+1)}F_{A}\big{\vert}^2 + 
\frac{1}{\epsilon_5}\big{\vert}\gamma^{(s-2i)/2}\nabla_M^{(l+j-i)}F_{A}\big{\vert}^2\big{)} \\
&=
C(g, \gamma)\epsilon_5\vert\vert\gamma^{s/2}\nabla_M^{k+1}\nabla_M^{(l)}F_A\vert\vert^2_{L^2} 
+
\int \sum_{i=0}^j \frac{C(g, \gamma)}{\epsilon_5}\big{\vert}\gamma^{(s-2i)/2}\nabla_M^{(l+j-i)}F_{A}\big{\vert}^2.
\end{align*}
For $\epsilon_5$ sufficiently small, we have that $\frac{C(g, \gamma)}{\epsilon_5} \geq 1$, so we can 
apply lemma \ref{interp_3} to obtain
\begin{align*}
&\hspace{0.5cm}
C(g, \gamma)\epsilon_5\vert\vert\gamma^{s/2}\nabla_M^{k+1}\nabla_M^{(l)}F_A\vert\vert^2_{L^2} 
+
\int \sum_{i=0}^j \frac{C(g, \gamma)}{\epsilon_5}\big{\vert}\gamma^{(s-2i)/2}\nabla_M^{(l+j-i)}F_{A}\big{\vert}^2 \\
&\leq
C(g, \gamma)\epsilon_5\vert\vert\gamma^{s/2}\nabla_M^{k+1}\nabla_M^{(l)}F_A\vert\vert^2_{L^2} + 
C\tilde{\epsilon}_5\vert\vert\gamma^{s/2}\nabla_M^{k+1}\nabla_M^{(l)}F_A\vert\vert^2_{L^2} +
\frac{C(\tilde{\epsilon}_5, g, \gamma)}{\epsilon_5^2}\vert\vert F_A\vert\vert^2_{L^2, \gamma>0}.
\end{align*}
Putting the estimates obtained for $0 \leq v \leq k+l+1$ with those obtained for 
$k+l+2 \leq v \leq 2k+l$, we see that we get the statement of the lemma.

\end{proof}


\begin{lem}\label{smoothing_curvature_3}
Assume $\sup_{t \in [0,T)}\vert\vert F_A\vert\vert_{\infty} < \infty$, and let 
$K(\vert\vert\phi\vert\vert_{\infty}) 
= max\{1, \sup_{t\in [0,T)}\vert\vert\phi\vert\vert_{\infty}\}$.
Suppose $\gamma$ is a bump function, and
$s \geq 2(k+l)$.
Then for $\epsilon_6 > 0$ sufficiently small, we have the following estimate
\begin{align*}
&\hspace{0.5cm}
\int
2Re\bigg{(} \langle -
2iIm\bigg{(} \sum_{i=1}^kC_i\nabla_M^{(l)}d\nabla_M^{*(i)}\langle\nabla_A^{(k)}\nabla_A\phi, 
\nabla_A^{(k-i)}\phi\rangle\bigg{)},  \gamma^{s}\nabla_M^{(l)}F_A\rangle\bigg{)} \\
&\leq
C(g)\bigg{(}\sup_{t \in [0,T)}\vert\vert F_A\vert\vert_{\infty} 
\bigg{)}K(\vert\vert\phi\vert\vert_{\infty})\epsilon_6
\vert\vert\gamma^{s/2}\nabla_A^{k+1}\nabla_A^{(l)}\phi\vert\vert^2_{L^2} 
+
C(\epsilon_{6}, g)K(\vert\vert\phi\vert\vert_{\infty})\bigg{(}\sup_{t \in [0,T)}\vert\vert 
F_A\vert\vert_{\infty}\bigg{)}
\vert\vert\phi\vert\vert^2_{L^2, \gamma>0}.
\end{align*}
\end{lem}
\begin{proof}
By applying integration by parts, we have
\begin{align*}
&\hspace{0.5cm}
\int
\langle -
2iIm\bigg{(} \sum_{i=1}^kC_i\nabla_M^{(l)}d\nabla_M^{*(i)}\langle\nabla_A^{(k)}\nabla_A\phi, 
\nabla_A^{(k-i)}\phi\rangle\bigg{)},  \gamma^{s}\nabla_M^{(l)}F_A\rangle \\
&=
\int
2i(-1)^l
\langle 
Im\bigg{(} \sum_{i=1}^kC_i\nabla_M^{*(l)}
\big{(}\gamma^{s}\nabla_M^{(l)}d\nabla_M^{*(i)}\langle\nabla_A^{(k)}\nabla_A\phi, 
\nabla_A^{(k-i)}\phi\rangle\big{)}\bigg{)}, F_A\rangle.
\end{align*}
We can then bound
\begin{align*}
&\hspace{0.5cm}
\bigg{\vert}
\int
\langle -
2iIm\bigg{(} \sum_{i=1}^kC_i\nabla_M^{(l)}d\nabla_M^{*(i)}\langle\nabla_A^{(k)}\nabla_A\phi, 
\nabla_A^{(k-i)}\phi\rangle\bigg{)},  \gamma^{s}\nabla_M^{(l)}F_A\rangle
\bigg{\vert} \\
&\leq
\int 2
\big{\vert}
\langle 
Im\bigg{(} \sum_{i=1}^kC_i\nabla_M^{*(l)}
\big{(}\gamma^{s}\nabla_M^{(l)}d\nabla_M^{*(i)}\langle\nabla_A^{(k)}\nabla_A\phi, 
\nabla_A^{(k-i)}\phi\rangle\big{)}\bigg{)}, F_A\rangle
\big{\vert} \\
&\leq
\int 
\sum_{i=1}^kC_i \big{\vert}
\nabla_M^{*(l)}
\big{(}\gamma^{s}\nabla_M^{(l)}d\nabla_M^{*(i)}\langle\nabla_A^{(k)}\nabla_A\phi, 
\nabla_A^{(k-i)}\phi\rangle\big{)}\big{\vert}
\bigg{(}
\sup_{t \in [0,T)}\vert\vert F_A\vert\vert_{\infty}
\bigg{)}.
\end{align*}
The next step is to estimate the term
$\int 
\sum_{i=1}^kC_i \big{\vert}
\nabla_M^{*(l)}
\big{(}\gamma^{s}\nabla_M^{(l)}d\nabla_M^{*(i)}\langle\nabla_A^{(k)}\nabla_A\phi, 
\nabla_A^{(k-i)}\phi\rangle\big{)}\big{\vert}$. This estimate follows exactly what was done 
in the proof of lemma \ref{smoothing_spinor_5}. We refer the reader to that lemma for the details. 

The statement of the lemma then easily follows.

\end{proof}


In general, we cannot obtain a local estimate for the curvature term alone, like we did for the spinor
field in theorem \ref{spinor_smoothing_final},
due to the term
$\vert\vert\gamma^{s/2}\nabla_A^{k+1}\nabla_A^{(l)}\phi\vert\vert^2_{L^2}$ that appears in the above lemma.
In fact, as the higher order Seiberg-Witten flow is a coupled system this shouldn't be surprising.

We can however prove a local estimate for the sum 
$\vert\vert\gamma^{s/2}\nabla_A^{(l)}\phi\vert\vert^2_{L^2} +
\vert\vert\gamma^{s/2}\nabla_M^{(l)}F_A\vert\vert^2_{L^2}$, as we now show.

\begin{thm}\label{smoothing_estimate_sum}
Let $(\phi(t), A(t))$ be a solution to the higher order Seiberg-Witten flow.
Assume $Q(\vert\vert F_A\vert\vert_{\infty}) = \sup_{t \in [0,T)}\vert\vert F_A\vert\vert_{\infty} < \infty$, and 
let $K(\vert\vert\phi\vert\vert_{\infty}) 
= max\{1, \sup_{t\in [0,T)}\vert\vert\phi\vert\vert_{\infty}\}$.
Suppose $\gamma$ is a bump function, and
$s \geq 2(k+l)$.
Then
\begin{align*}
\frac{\partial}{\partial t}\bigg{(}\vert\vert \gamma^{s/2}\nabla_A^{(l)}\phi\vert\vert^2_{L^2} 
+
\vert\vert \gamma^{s/2}\nabla_M^{(l)}F_A\vert\vert^2_{L^2}\bigg{)} 
&\leq 
-\lambda\bigg{(}\vert\vert\gamma^{s/2}\nabla_A^{k+1}\nabla_A^{(l)}\phi\vert\vert^2_{L^2}
+ \vert\vert\gamma^{s/2}\nabla_M^{k+1}\nabla_M^{(l)}F_A\vert\vert^2_{L^2}\bigg{)} \\
&\hspace{0.5cm}
+
C\big{(}Q(\vert\vert F_A\vert\vert_{\infty}), K(\vert\vert\phi\vert\vert_{\infty}), g, \gamma\big{)}\bigg{(}
\vert\vert\phi\vert\vert^2_{L^2, \gamma>0} 
+
\vert\vert F_A\vert\vert^2_{L^2, \gamma>0}\bigg{)}
\end{align*} 
where $1 \leq \lambda < 2$.
\end{thm}
\begin{proof}
We start by taking 
$0 < \epsilon = \epsilon_1 = \epsilon_2 = \epsilon_3 = \tilde{\epsilon}_3 = \ldots = \epsilon_{6}$ in lemmas
\ref{smoothing_curvature_1}, \ref{smoothing_curvature_2}, \ref{smoothing_curvature_3}. We can then obtain 
an estimate for the curvature
\begin{align*}
\frac{\partial}{\partial t}\vert\vert \gamma^{s/2}\nabla_M^{(l)}F_A\vert\vert^2_{L^2} 
&\leq
\big{(} -2 + \tilde{C}_1\big{(}Q(\vert\vert F_A\vert\vert_{\infty}), K(\vert\vert\phi\vert\vert_{\infty}), g, \gamma 
\big{)}\epsilon\big{)}
\vert\vert\gamma^{s/2}\nabla_M^{k+1}\nabla_M^{(l)}F_A\vert\vert^2_{L^2} \\
&\hspace{0.5cm} +
\tilde{C}_2\big{(}\epsilon, Q(\vert\vert F_A\vert\vert_{\infty}), K(\vert\vert\phi\vert\vert_{\infty}), g, 
\gamma\big{)}
\vert\vert F_A\vert\vert^2_{L^2, \gamma>0} \\
&\hspace{0.5cm} +
\tilde{C}_3\big{(}Q(\vert\vert F_A\vert\vert_{\infty}), K(\vert\vert\phi\vert\vert_{\infty}), g
\big{)}\epsilon
\vert\vert\gamma^{s/2}\nabla_A^{k+1}\nabla_A^{(l)}\phi\vert\vert^2_{L^2} \\
&\hspace{0.5cm} +
\tilde{C}_4\big{(}\epsilon, Q(\vert\vert F_A\vert\vert_{\infty}), K(\vert\vert\phi\vert\vert_{\infty}), g
\big{)}
\vert\vert\phi\vert\vert^2_{L^2, \gamma>0}
\end{align*}
where the last two terms involving $\phi$ come from lemma \ref{smoothing_curvature_3}.

We then apply the same argument to the spinor field $\phi$. Namely, take
$0 < \epsilon = \epsilon_1 = \epsilon_2 = \epsilon_3 = \tilde{\epsilon}_3 = \ldots = 
\epsilon_{8} = \tilde{\epsilon}_{8} = \epsilon_9$ in lemmas 
\ref{smoothing_spinor_1}, \ref{smoothing_spinor_2}, \ref{smoothing_spinor_3}, \ref{smoothing_spinor_4}, 
\ref{smoothing_spinor_5}, \ref{smoothing_spinor_6}. 
We then have the estimate
\begin{align*}
\frac{\partial}{\partial t}\vert\vert \gamma^{s/2}\nabla_A^{(l)}\phi\vert\vert^2_{L^2} 
&\leq
\big{(} -2 + 
\tilde{C}_5\big{(}Q(\vert\vert F_A\vert\vert_{\infty}), K(\vert\vert\phi\vert\vert_{\infty}), g, \gamma \big{)}
\epsilon\big{)}
\vert\vert\gamma^{s/2}\nabla_A^{k+1}\nabla_A^{(l)}\phi\vert\vert^2_{L^2} \\
&\hspace{1cm} +
\tilde{C}_6\big{(}\epsilon, Q(\vert\vert F_A\vert\vert_{\infty}), K(\vert\vert\phi\vert\vert_{\infty}), g, 
\gamma\big{)}
\vert\vert\phi\vert\vert^2_{L^2, \gamma>0}.
\end{align*}
Combining the two estimates together, we obtain
\begin{align*}
&\hspace{0.5cm}
\frac{\partial}{\partial t}\vert\vert \gamma^{s/2}\nabla_A^{(l)}\phi\vert\vert^2_{L^2}
+
\frac{\partial}{\partial t}\vert\vert \gamma^{s/2}\nabla_M^{(l)}F_A\vert\vert^2_{L^2} \\
&\leq
\bigg{(} -2 + 
\tilde{C}_5\big{(}Q(\vert\vert F_A\vert\vert_{\infty}), K(\vert\vert\phi\vert\vert_{\infty}), g, \gamma \big{)}
\epsilon +
\tilde{C}_3\big{(}Q(\vert\vert F_A\vert\vert_{\infty}), K(\vert\vert\phi\vert\vert_{\infty}), g
\big{)}\epsilon\bigg{)}
\vert\vert\gamma^{s/2}\nabla_A^{k+1}\nabla_A^{(l)}\phi\vert\vert^2_{L^2} \\
&\hspace{0.5cm} +
\bigg{(} -2 + \tilde{C}_1\big{(}Q(\vert\vert F_A\vert\vert_{\infty}), K(\vert\vert\phi\vert\vert_{\infty}), g, \gamma 
\big{)}\epsilon\bigg{)}
\vert\vert\gamma^{s/2}\nabla_M^{k+1}\nabla_M^{(l)}F_A\vert\vert^2_{L^2} \\
&\hspace{0.5cm} +
\bigg{(}
\tilde{C}_4\big{(}\epsilon, Q(\vert\vert F_A\vert\vert_{\infty}), K(\vert\vert\phi\vert\vert_{\infty}), g
\big{)} +
\tilde{C}_6\big{(}\epsilon, Q(\vert\vert F_A\vert\vert_{\infty}), K(\vert\vert\phi\vert\vert_{\infty}), g, 
\gamma\big{)}
\bigg{)}\vert\vert\phi\vert\vert^2_{L^2, \gamma>0} \\
&\hspace{0.5cm} +
\tilde{C}_2\big{(}\epsilon, Q(\vert\vert F_A\vert\vert_{\infty}), K(\vert\vert\phi\vert\vert_{\infty}), g, 
\gamma\big{)}
\vert\vert F_A\vert\vert^2_{L^2, \gamma>0}. 
\end{align*}

By choosing $\epsilon$ sufficiently small, we can make it so that
\begin{equation*}
1\leq 
\bigg{(} 2 - 
\tilde{C}_5\big{(}Q(\vert\vert F_A\vert\vert_{\infty}), K(\vert\vert\phi\vert\vert_{\infty}), g, \gamma \big{)}
\epsilon -
\tilde{C}_3\big{(}Q(\vert\vert F_A\vert\vert_{\infty}), K(\vert\vert\phi\vert\vert_{\infty}), g
\big{)}\epsilon\bigg{)} < 2
\end{equation*}
and
\begin{equation*}
1 \leq
\bigg{(} 2 - \tilde{C}_1\big{(}Q(\vert\vert F_A\vert\vert_{\infty}), K(\vert\vert\phi\vert\vert_{\infty}), g, \gamma 
\big{)}\epsilon\bigg{)} < 2.
\end{equation*}
For this value of $\epsilon$, we define $\lambda$ to be the minimum of these two constants
\begin{align*}
\lambda = min\bigg{\{}2 - 
\tilde{C}_5\big{(}Q(\vert\vert F_A\vert\vert_{\infty}), K(\vert\vert\phi\vert\vert_{\infty}), g, \gamma \big{)}
\epsilon -
\tilde{C}_3\big{(}Q(\vert\vert F_A\vert\vert_{\infty}), K(\vert\vert\phi\vert\vert_{\infty}), g
\big{)}\epsilon, \\
2 - \tilde{C}_1\big{(}Q(\vert\vert F_A\vert\vert_{\infty}), K(\vert\vert\phi\vert\vert_{\infty}), g, \gamma 
\big{)}\epsilon \bigg{\}}.
\end{align*}
and define
\begin{align*}
&\hspace{0.5cm}
C\big{(}Q(\vert\vert F_A\vert\vert_{\infty}), K(\vert\vert\phi\vert\vert_{\infty}), g, \gamma\big{)} \\
&= 
max\bigg{\{}
\tilde{C}_4\big{(}\epsilon, Q(\vert\vert F_A\vert\vert_{\infty}), K(\vert\vert\phi\vert\vert_{\infty}), g
\big{)} +
\tilde{C}_6\big{(}\epsilon, Q(\vert\vert F_A\vert\vert_{\infty}), K(\vert\vert\phi\vert\vert_{\infty}), g, 
\gamma\big{)}, \\
&\hspace{2cm}
\tilde{C}_2\big{(}\epsilon, Q(\vert\vert F_A\vert\vert_{\infty}), K(\vert\vert\phi\vert\vert_{\infty}), g, 
\gamma\big{)}
\bigg{\}}.
\end{align*}
The theorem then follows.

\end{proof}

The following corollary is a simple consequence of integrating the inequality, in the above theorem, in time.

\begin{cor}\label{smoothing_estimate_sum_large_time}
Suppose $(\phi(t), A(t))$ is a solution to the higher order Seiberg-Witten flow, on the time interval
$[0, T)$, where $T < \infty$. Assume the conditions of the above theorem, then
\begin{align*}
\vert\vert\gamma^{s/2}\nabla_{A(t)}^{(l)}\phi(t)\vert\vert^2_{L^2} + 
\vert\vert\gamma^{s/2}\nabla_M^{(l)} F_{A(t)}\vert\vert^2_{L^2} \leq 
TC_l
\sup_{t\in [0,T)}\bigg{(}\vert\vert\phi\vert\vert^2_{L^2, \gamma>0} + 
\vert\vert F_A\vert\vert^2_{L^2, \gamma>0}\bigg{)}
\end{align*}
where $C_l$ depends on 
$C\big{(}Q(\vert\vert F_A\vert\vert_{\infty}), K(\vert\vert\phi\vert\vert_{\infty}), g, \gamma\big{)}$ and on
the initial condition $(\phi(0), A(0))$.
\end{cor}

Using this corollary, we can then obtain the following proposition, which will be useful when we want
to find obstructions to long time existence.

\begin{prop}\label{smoothing_estimate_sum_large_time_2}
Suppose $(\phi(t), A(t))$ is a solution to the higher order Seiberg-Witten flow on the time interval $[0, T)$, with
$T < \infty$. 
Assume $Q(\vert\vert F_A\vert\vert_{\infty}) = \sup_{t \in [0,T)}\vert\vert F_A\vert\vert_{\infty} < \infty$, and 
let $K(\vert\vert\phi\vert\vert_{\infty}) 
= max\{1, \sup_{t\in [0,T)}\vert\vert\phi\vert\vert_{\infty}\}$.
Suppose $\gamma$ is a bump function, and
$s \geq 2(2k+l)$.
Then
\begin{equation*}
\sup_{M \times [0, T)}\big{(} \vert\gamma^{s/2}\nabla_{A(t)}^{(l)}\phi(t)\vert^2 + 
\vert\gamma^{s/2}\nabla_M^{(l)} F_{A(t)}\vert^2\big{)} \leq 
T\tilde{C}_l\sup_{t\in [0,T)}\bigg{(}\vert\vert\phi\vert\vert^2_{L^2, \gamma>0} + 
\vert\vert F_A\vert\vert^2_{L^2, \gamma>0}\bigg{)}
\end{equation*}
where $\tilde{C}_l$ depends on 
$C\big{(}Q(\vert\vert F_A\vert\vert_{\infty}), K(\vert\vert\phi\vert\vert_{\infty}), g, \gamma\big{)}$ and on
the initial condition $(\phi(0), A(0))$, where 
$C\big{(}Q(\vert\vert F_A\vert\vert_{\infty}), K(\vert\vert\phi\vert\vert_{\infty}), g, \gamma\big{)}$ is
the constant coming from theorem \ref{smoothing_estimate_sum}.
\end{prop}
\begin{proof}
We start by noting that, by the Sobolev embedding theorem, we have an embedding
$W^{k,2} \subseteq C^0$ provided $k > n/2$. Therefore, fixing $t \in [0, T)$, we have
\begin{align*}
\sup_{M}\vert\gamma^{s/2}\nabla_{A(t)}^{(l)}\phi(t)\vert &\leq
\sum_{j=0}^kC_{k,2}\vert\vert\nabla_{A(t)}^{(j)}\big{(}\vert\gamma^{s/2}\nabla_{A(t)}^{(l)}
\phi(t)\vert\big{)}\vert\vert_{L^2} \\
&\leq
\sum_{j=0}^k C_{k,2}C(\gamma)\vert\vert\gamma^{(s-2j)/2}\nabla_{A(t)}^{(j+l)}
\phi(t)\vert\vert_{L^2} 
\end{align*}
where in order to take the derivative $\nabla_{A(t)}^{(j)}$ inside the absolute value, we have
applied Kato's inequality, and $C_{k,2}$ is the Sobolev constant.

A similar computation gives
\begin{align*}
\sup_{M}\vert\gamma^{s/2}\nabla_{M}^{(l)}F_{A(t)}\vert &\leq
\sum_{j=0}^kC_{k,2}\vert\vert\nabla_{M}^{(j)}\big{(}\vert\gamma^{s/2}\nabla_{M}^{(l)}
F_{A(t)}\vert\big{)}\vert\vert_{L^2} \\
&\leq 
\sum_{j=0}^kC_{k,2}C(\gamma)\vert\vert\gamma^{(s-2j)/2}\nabla_{M}^{(j+l)}
F_{A(t)}\vert\vert_{L^2}
\end{align*}
Combining these two inequalities, we obtain
\begin{align*}
&\hspace{0.5cm}
\sup_{M}\big{(}\vert\gamma^{s/2}\nabla_{A(t)}^{(l)}\phi(t)\vert^2 + 
\vert\gamma^{s/2}\nabla_{M}^{(l)}F_{A(t)}\vert^2\big{)} \\
&\leq
\sum_{j=0}^k C_{k,2}C(\gamma)\big{(}
\vert\vert\gamma^{(s-2j)/2}\nabla_{A(t)}^{(j+l)}
\phi(t)\vert\vert^2_{L^2} +
\vert\vert\gamma^{(s-2j)/2}\nabla_{M}^{(j+l)}
F_{A(t)}\vert\vert^2_{L^2}\big{)}.
\end{align*} 

We now want to apply corollary \ref{smoothing_estimate_sum_large_time}. 
In order to do this, we observe that we are assuming $s \geq 2(2k + l)$,
which implies $s-2j \geq 2(j+l)$, so we are free
to apply corollary \ref{smoothing_estimate_sum_large_time}. In doing so, we obtain
\begin{align*}
&\hspace{0.5cm}
\sup_{M}\big{(}\vert\gamma^{s/2}\nabla_{A(t)}^{(l)}\phi(t)\vert^2 + 
\vert\gamma^{s/2}\nabla_{M}^{(l)}F_{A(t)}\vert^2\big{)} \\
&\leq
\sum_{j=0}^k C_{k,2}C(\gamma)\big{(}
\vert\vert\gamma^{(s-2j)/2}\nabla_{A(t)}^{(j+l)}
\phi(t)\vert\vert^2_{L^2} +
\vert\vert\gamma^{(s-2j)/2}\nabla_{M}^{(j+l)}
F_{A(t)}\vert\vert^2_{L^2}\big{)} \\
&\leq
C_{k,2}C(\gamma)T\bigg{(}\sum_{j=0}^kC_{j+l}\bigg{)}
\sup_{t\in [0,T)}\bigg{(}\vert\vert\phi\vert\vert^2_{L^2, \gamma>0} + 
\vert\vert F_A\vert\vert^2_{L^2, \gamma>0}\bigg{)}
\end{align*}
where $C_{j+l}$ are the constants coming from corollary \ref{smoothing_estimate_sum_large_time}. 

Defining $\tilde{C}_l = C_{k,2}C(\gamma)\bigg{(}\sum_{j=0}^kC_{j+l}\bigg{)}$, gives the result.

\end{proof}

\subsection{Estimates of Bernstein-Bando-Shi type}
In this subsection we will obtain estimates of Bernstein-Bando-Shi type, using the results obtained from 
the previous two subsections.

For the next theorem we will be making use of the constant 
$C\big{(}Q(\vert\vert F_A\vert\vert_{\infty}), K(\vert\vert\phi\vert\vert_{\infty}), g, \gamma\big{)}$ defined
in theorem \ref{smoothing_estimate_sum}. To make the notation a little bit easier, we will denote this constant
by $C$.


\begin{thm}\label{BBS_main}
Suppose $(\phi(t), A(t))$ is a solution to the higher order Seiberg-Witten flow on the time interval $[0, T)$, with 
$\sup_{t \in [0,T)}\vert\vert F_A\vert\vert_{\infty} < \infty$.  
Let $K = max\{1, C\}$, and suppose $T < \frac{1}{K}$. 
Suppose $\gamma$ is a bump function, and
$s \geq 2(2k+l)$.
Then for each $l \in \mathbb{N}$, there exists a positive
constant $C_l = C_l(dim M, K, g, \gamma, s)$ such that
\begin{align*}
\vert\vert\gamma^{s/2}\nabla_{A(t)}^{(l)}\phi(t)\vert\vert^2_{L^2} + 
\vert\vert\gamma^{s/2}\nabla_M^{(l)} F_{A(t)}\vert\vert^2_{L^2} \leq 
C_l\frac{\sup_{t\in [0,T)}\bigg{(}\vert\vert\phi\vert\vert^2_{L^2, \gamma>0} + 
\vert\vert F_A\vert\vert^2_{L^2, \gamma>0}\bigg{)}}{t^{\frac{l}{k+1}+1}}.
\end{align*} 
\end{thm}
\begin{proof}
Define 
\begin{equation*}
G(t) = \sum_{m=0}^l a_mt^m\big{(}
\vert\vert\gamma^{s/2}\nabla_{A(t)}^{(k+1)m}\phi(t)\vert\vert^2_{L^2} + 
\vert\vert\gamma^{s/2}\nabla_M^{(k+1)m} F_{A(t)}\vert\vert^2_{L^2}\big{)}
\end{equation*}
where $a_0 = 1$, and $a_m$ for $1 \leq m \leq l$ will be determined.

Differentiating $G$, and applying theorem \ref{smoothing_estimate_sum}, we obtain
\begin{align*}
\frac{\partial G}{\partial t} &= \sum_{m=1}^lma_mt^{m-1}
\big{(}
\vert\vert\gamma^{s/2}\nabla_{A(t)}^{(k+1)m}\phi(t)\vert\vert^2_{L^2} + 
\vert\vert\gamma^{s/2}\nabla_M^{(k+1)m} F_{A(t)}\vert\vert^2_{L^2}\big{)} \\
&\hspace{1cm} +
\sum_{m=0}^{l}a_mt^m\frac{\partial}{\partial t}
\big{(}
\vert\vert\gamma^{s/2}\nabla_{A(t)}^{(k+1)m}\phi(t)\vert\vert^2_{L^2} + 
\vert\vert\gamma^{s/2}\nabla_M^{(k+1)m} F_{A(t)}\vert\vert^2_{L^2}\big{)} \\
&\leq 
\sum_{m=0}^{l-1}(m+1)a_{m+1}t^{m}
\big{(}
\vert\vert\gamma^{s/2}\nabla_{A(t)}^{(k+1)(m+1)}\phi(t)\vert\vert^2_{L^2} + 
\vert\vert\gamma^{s/2}\nabla_M^{(k+1)(m+1)} F_{A(t)}\vert\vert^2_{L^2}\big{)} \\
&\hspace{1cm} +
\sum_{m=0}^la_mt^m\bigg{(}-\lambda\big{(}\vert\vert\gamma^{s/2}\nabla_{A(t)}^{(k+1)(m+1)}\phi(t)\vert\vert^2_{L^2} + 
\vert\vert\gamma^{s/2}\nabla_M^{(k+1)(m+1)} F_{A(t)}\vert\vert^2_{L^2}\big{)} \\
&\hspace{4cm} +
K\big{(}\vert\vert\phi(t)\vert\vert^2_{L^2, \gamma>0} + \vert\vert F_{A(t)}\vert\vert^2_{L^2, \gamma>0}\big{)}\bigg{)} 
\\
&= 
-\lambda a_lt^l\big{(}\vert\vert\gamma^{s/2}\nabla_{A(t)}^{(k+1)(l+1)}\phi(t)\vert\vert^2_{L^2} + 
\vert\vert\gamma^{s/2}\nabla_M^{(k+1)(l+1)} F_{A(t)}\vert\vert^2_{L^2}\big{)} \\
&\hspace{1cm} +
\sum_{m=0}^{l-1}\big{(}(m+1)a_{m+1} - \lambda a_m\big{)}t^m\big{(}
\vert\vert\gamma^{s/2}\nabla_{A(t)}^{(k+1)(m+1)}\phi(t)\vert\vert^2_{L^2} + 
\vert\vert\gamma^{s/2}\nabla_M^{(k+1)(m+1)} F_{A(t)}\vert\vert^2_{L^2}\big{)} \\
&\hspace{1cm} +
K\sum_{m=0}^la_mt^m
\big{(}\vert\vert\phi(t)\vert\vert^2_{L^2, \gamma>0} + \vert\vert F_{A(t)}\vert\vert^2_{L^2, \gamma>0}\big{)}.
\end{align*}
Define $a_l = 1$, and then define $a_m$, for $0\leq m\leq l-1$, recursively so that
\begin{equation*}
(m+1)a_{m+1} - \lambda a_m \leq 0.
\end{equation*}
We then have
\begin{align*}
\frac{\partial G}{\partial t} &\leq
K\sum_{m=0}^la_mt^m
\big{(}\vert\vert\phi(t)\vert\vert^2_{L^2, \gamma>0} + \vert\vert F_{A(t)}\vert\vert^2_{L^2, \gamma>0}\big{)} \\
&\leq
K\tilde{C}_{(k+1)l}
\big{(}\vert\vert\phi(t)\vert\vert^2_{L^2, \gamma>0} + \vert\vert F_{A(t)}\vert\vert^2_{L^2, \gamma>0}\big{)}
\end{align*}
where to get the second inequality we just note that, by assumption $t < T < 1/K \leq 1$, and where we have
taken $\tilde{C}_{(k+1)l} = \sum_{m=0}^la_m$.

Integrating the above gives
\begin{align*}
G(t) - G(0) &\leq K\tilde{C}_{(k+1)l}\int_{0}^t
\big{(}\vert\vert\phi(t)\vert\vert^2_{L^2, \gamma>0} + \vert\vert F_{A(t)}\vert\vert^2_{L^2, \gamma>0}\big{)}ds \\
&\leq 
K\tilde{C}_{(k+1)l}\bigg{(}\sup_{t\in [0,T)}
\big{(}\vert\vert\phi\vert\vert^2_{L^2, \gamma>0} + \vert\vert F_A\vert\vert^2_{L^2, \gamma>0}\big{)}\bigg{)}t
\end{align*}
which implies
\begin{align*}
G(t) &\leq K\tilde{C}_{(k+1)l}\bigg{(}\sup_{t\in [0,T)}
\big{(}\vert\vert\phi\vert\vert^2_{L^2, \gamma>0} + \vert\vert F_A\vert\vert^2_{L^2, \gamma>0}\big{)}\bigg{)}t +
G(0) \\
&\leq
\tilde{C}_{(k+1)l}\bigg{(}\sup_{t\in [0,T)}
\big{(}\vert\vert\phi\vert\vert^2_{L^2, \gamma>0} + \vert\vert F_A\vert\vert^2_{L^2, \gamma>0}\big{)}\bigg{)} +
G(0) \\
&\leq 
C_{(k+1)l}\bigg{(}\sup_{t\in [0,T)}
\big{(}\vert\vert\phi\vert\vert^2_{L^2, \gamma>0} + \vert\vert F_A\vert\vert^2_{L^2, \gamma>0}\big{)}\bigg{)}
\end{align*}
where to get the second inequality, we have used the fact that $t < T < 1/K$, and to get the third
inequality, we have used $G(0) = 
\vert\vert\phi(0)\vert\vert^2_{L^2, \gamma>0} + \vert\vert F_{A(0)}\vert\vert^2_{L^2, \gamma>0}$, 
and defined $C_{(k+1)q} = \tilde{C}_{(k+1)l} + 1$.

Using the fact that, $t^l\big{(} \vert\vert\gamma^{s/2}\nabla_{A(t)}^{(k+1)l}\phi(t)\vert\vert^2_{L^2} + 
\vert\vert\gamma^{s/2}\nabla_M^{(k+1)l} F_{A(t)}\vert\vert^2_{L^2}\big{)} \leq G(t)$, we obtain
\begin{align*}
t^l\big{(} \vert\vert\gamma^{s/2}\nabla_{A(t)}^{(k+1)l}\phi(t)\vert\vert^2_{L^2} + 
\vert\vert\gamma^{s/2}\nabla_M^{(k+1)l} F_{A(t)}\vert\vert^2_{L^2}\big{)} \leq
C_{(k+1)l}\bigg{(}\sup_{t\in [0,T)}
\big{(}\vert\vert\phi\vert\vert^2_{L^2, \gamma>0} + \vert\vert F_A\vert\vert^2_{L^2, \gamma>0}\big{)}\bigg{)}
\end{align*}
which implies
\begin{align*}
\vert\vert\gamma^{s/2}\nabla_{A(t)}^{(k+1)l}\phi(t)\vert\vert^2_{L^2} + 
\vert\vert\gamma^{s/2}\nabla_M^{(k+1)l} F_{A(t)}\vert\vert^2_{L^2} &\leq
C_{(k+1)l}\frac{\bigg{(}\sup_{t\in [0,T)}
\big{(}\vert\vert\phi\vert\vert^2_{L^2, \gamma>0} + \vert\vert F_A\vert\vert^2_{L^2, \gamma>0}\big{)}\bigg{)}}{t^l} \\
&\leq
C_{(k+1)l}\frac{\bigg{(}\sup_{t\in [0,T)}
\big{(}\vert\vert\phi\vert\vert^2_{L^2, \gamma>0} + \vert\vert F_A\vert\vert^2_{L^2, \gamma>0}\big{)}\bigg{)}}{t^{l+1}}. 
\end{align*}
This proves the lemma for the case of $(k+1)l$, and more generally the case of $(k+1)r$ for any $r \geq 0$.

In the general case, write $l = (k+1)r + w$, where $1 \leq w \leq k$. Then
\begin{align*}
&\hspace{0.5cm}
\vert\vert\gamma^{s/2}\nabla_{A(t)}^{(k+1)r + w}\phi(t)\vert\vert^2_{L^2} + 
\vert\vert\gamma^{s/2}\nabla_M^{(k+1)r + w} F_{A(t)}\vert\vert^2_{L^2} \\
&\leq 
\vert\vert\gamma^{s/2}\nabla_{A(t)}^{(k+1)(r+1)}\phi(t)\vert\vert^2_{L^2} + 
\vert\vert\gamma^{s/2}\nabla_M^{(k+1)(r+1)} F_{A(t)}\vert\vert^2_{L^2} 
+ 
C_1\big{(} \vert\vert\phi(t)\vert\vert^2_{L^2, \gamma>0} + \vert\vert F_{A(t)}\vert\vert^2_{L^2, \gamma>0} 
\big{)} \\
&\leq
C_{(k+1)(r+1)}\frac{\bigg{(}\sup_{t\in [0,T)}
\big{(}\vert\vert\phi\vert\vert^2_{L^2, \gamma>0} + 
\vert\vert F_A\vert\vert^2_{L^2, \gamma>0}\big{)}\bigg{)}}{t^{l+1}} +
C_1\big{(} \vert\vert\phi(t)\vert\vert^2_{L^2, \gamma>0} + \vert\vert F_{A(t)}\vert\vert^2_{L^2, \gamma>0} \big{)} \\
&\leq
C_{l}\frac{\bigg{(}\sup_{t\in [0,T)}
\big{(}\vert\vert\phi\vert\vert^2_{L^2, \gamma>0} + 
\vert\vert F_A\vert\vert^2_{L^2, \gamma>0}\big{)}\bigg{)}}{t^{\frac{l}{k+1}+1}} 
\end{align*}
where to get the first inequality, we have used lemma \ref{interp_3}
with $\epsilon = 1$. To get the second inequality, we have applied the theorem to the case $(k+1)r$, and to get the 
third inequality we have defined $C_{l} = C_{(k+1)(r+1)} + C_1$.

\end{proof}


\begin{prop}\label{BBS_small_time}
Under the same assumptions as theorem \ref{BBS_main}, we have
\begin{equation*}
\sup_{M}\big{(} \vert\gamma^{s/2}\nabla_{A(t)}^{(l)}\phi(t)\vert^2 + 
\vert\gamma^{s/2}\nabla_M^{(l)} F_{A(t)}\vert^2\big{)} \leq 
B_l\sup_{t\in [0,T)}\bigg{(}\vert\vert\phi\vert\vert^2_{L^2, \gamma>0} + 
\vert\vert F_A\vert\vert^2_{L^2, \gamma>0}\bigg{)}
\end{equation*}
where $B_l = B_l(t, dim M, K, g, \gamma, s)$.
\end{prop}
\begin{proof}
We start by noting that by the Sobolev embedding theorem we have an embedding
$W^{k,2} \subseteq C^0$ provided $k > n/2$. Therefore, fixing $t \in [0, T)$, we have
\begin{align*}
\sup_{M}\vert\gamma^{s/2}\nabla_{A(t)}^{(l)}\phi(t)\vert &\leq
\sum_{j=0}^kC_{k,2}\vert\vert\nabla_{A(t)}^{(j)}\big{(}\vert\gamma^{s/2}\nabla_{A(t)}^{(l)}
\phi(t)\vert\big{)}\vert\vert_{L^2} \\
&\leq
\sum_{j=0}^k C_{k,2}C(\gamma)\vert\vert\gamma^{(s-2j)/2}\nabla_{A(t)}^{(j+l)}
\phi(t)\vert\vert_{L^2} 
\end{align*}
where in order to take the derivative $\nabla_{A(t)}^{(j)}$ inside the absolute value, we have
applied Kato's inequality.

A similar computation gives
\begin{align*}
\sup_{M}\vert\gamma^{s/2}\nabla_{M}^{(l)}F_{A(t)}\vert &\leq
\sum_{j=0}^kC_{k,2}\vert\vert\nabla_{M}^{(j)}\big{(}\vert\gamma^{s/2}\nabla_{M}^{(l)}
F_{A(t)}\vert\big{)}\vert\vert_{L^2} \\
&\leq 
\sum_{j=0}^kC_{k,2}C(\gamma)\vert\vert\gamma^{(s-2j)/2}\nabla_{M}^{(j+l)}
F_{A(t)}\vert\vert_{L^2}
\end{align*}
Combining these two inequalities we obtain
\begin{align*}
&\hspace{0.5cm}
\sup_{M}\big{(}\vert\gamma^{s/2}\nabla_{A(t)}^{(l)}\phi(t)\vert^2 + 
\vert\gamma^{s/2}\nabla_{M}^{(l)}F_{A(t)}\vert^2\big{)} \\
&\leq
\sum_{j=0}^k C_{k,2}C(\gamma)\big{(}
\vert\vert\gamma^{(s-2j)/2}\nabla_{A(t)}^{(j+l)}
\phi(t)\vert\vert^2_{L^2} +
\vert\vert\gamma^{(s-2j)/2}\nabla_{M}^{(j+l)}
F_{A(t)}\vert\vert^2_{L^2}\big{)}
\end{align*} 
where we have used the general fact that if $a_1, \ldots, a_n$ are positive numbers, then
$(a_1 + \ldots + a_n)^2 \leq C(a_1^2 + \ldots + a_n^2)$. We have absorbed the constant $C$ into $C(\gamma)$.

We now want to apply theorem \ref{BBS_main}. In order to do this, we observe that we are assuming 
$s \geq 2(2k + l)$, which implies $s-2j \geq 2(j+l)$, so we are free
to apply theorem \ref{BBS_main}. In doing so, we obtain
\begin{align*}
&\hspace{0.5cm}
\sup_{M}\big{(}\vert\gamma^{s/2}\nabla_{A(t)}^{(l)}\phi(t)\vert^2 + 
\vert\gamma^{s/2}\nabla_{M}^{(l)}F_{A(t)}\vert^2\big{)} \\
&\leq
\sum_{j=0}^k C_{k,2}C(\gamma)\big{(}
\vert\vert\gamma^{(s-2j)/2}\nabla_{A(t)}^{(j+l)}
\phi(t)\vert\vert^2_{L^2} +
\vert\vert\gamma^{(s-2j)/2}\nabla_{M}^{(j+l)}
F_{A(t)}\vert\vert^2_{L^2}\big{)} \\
&\leq
C_{k,2}C(\gamma)\bigg{(}\sum_{j=0}^k\frac{C_{j+l}}{t^{\frac{j+l}{k+1}+1}}\bigg{)}
\sup_{t\in [0,T)}\bigg{(}\vert\vert\nabla_A^{(l)}\phi\vert\vert^2_{L^2, \gamma>0} + 
\vert\vert\nabla_M^{(l)} F_A\vert\vert^2_{L^2, \gamma>0}\bigg{)}
\end{align*}
where $C_{j+l}$ are the constants coming from \ref{BBS_main}.

Defining $B_l = C_{k,2}C(\gamma)\bigg{(}\sum_{j=0}^k\frac{C_{j+l}}{t^{\frac{j+l}{k+1}+1}}\bigg{)}$, we obtain
the statement of the corollary.

\end{proof}

\subsection{Obstructions to long time existence}

The estimates from the previous subsections can now be used to study obstructions to long time existence.
The purpose of this subsection is to show that the only obstruction to extending a solution past the
maximal time is curvature blow up.

\begin{prop}\label{extend_connection}
Let $A(t)$ denote a sequence of time dependent unitary connections, defined on some time interval  
$[0, T)$, with $T < \infty$. Suppose we have uniform bounds
\begin{equation*}
\sup_{M\times [0, T)}\bigg{\vert} \nabla_{M}^{(p)}\frac{\partial A(t)}{\partial t}\bigg{\vert}
\leq C_p
\end{equation*}
for some positive constants $C_p$.

Then $lim_{t \rightarrow T}A(t)$ exists, is smooth, and the sequence $\{A(t)\}$ converges to this
limit connection in every $C^m$-norm, $m \geq 0$. We remind the reader that we view
$A(t) \in i\Lambda(M)$, so this convergence is in the sense of 1-forms.
\end{prop}
\begin{proof}
We define $A_T = A(0) + \int_{0}^T \frac{\partial A(t)}{\partial t}dt$. The uniform bounds,
in the assumption of the theorem, imply that the integral on the right is absolutely convergent. Hence
$A_T$, as defined, is well defined and exists.

We then compute
\begin{align*}
\vert A(t) - A_T\vert &= \bigg{\vert} A(t) - A(0) - 
\int_{0}^T \frac{\partial A(t)}{\partial t}dt \bigg{\vert} \\
&=
\bigg{\vert} \int_{0}^t \frac{\partial A(s)}{\partial s}ds - \int_{0}^T \frac{\partial A(t)}
{\partial t}dt\bigg{\vert} \\
&=
\bigg{\vert}\int_{t}^T \frac{\partial A(s)}{\partial s}ds\bigg{\vert} \\
&\leq
\int_{t}^T C_0ds \\
&= (T-t)C_0.
\end{align*}
It follows that $lim_{t \rightarrow T}\big{\vert}A(t) - A_T\big{\vert} \rightarrow 0$, which
implies that $\{A(t)\}$ converges to $A_T$ uniformly. This in turn implies that $A_T$ is
continuous.

The next step is to show that the limit connection $A_T$ is smooth. We have
\begin{align*}
\nabla_{M}^{(p)}(A_T) &= 
\nabla_{M}^{(p)}\bigg{(} A(0) + \int_{0}^{T}\frac{\partial \nabla_{A(t)}}{\partial t}dt\bigg{)} \\
&= \nabla_M^{(p)}A(0) + \nabla_{M}^{(p)}\int_{0}^{T}\frac{\partial A(t)}{\partial t}dt \\
&= \nabla_M^{(p)}A(0) +  \int_{0}^{T}\nabla_{M}^{(p)}\frac{\partial \nabla_{A(t)}}{\partial t}dt 
\end{align*}
where we are able to take $\nabla_{M}^{(p)}$ into the integral, because 
$\frac{\partial A(t)}{\partial t}$ has uniformly bounded derivatives, by the assumption of the theorem.
It follows that $A_T$ is smooth.

Finally, we show that $\{A(t)\}$ converges to $A_T$ in $C^m$. We compute
\begin{align*}
\big{\vert}\nabla_{M}^{(p)}(A_T) - \nabla_{M}^{(p)}(A(t))\big{\vert} &= 
\bigg{\vert} \nabla_{M}^{(p)}(A(0)) + 
\int_{0}^T\nabla_{M}^{(p)}\frac{\partial A(t)}{\partial t}dt - \nabla_{M}^{(p)}(A(t)) 
\bigg{\vert} \\
&=
\bigg{\vert} -\int_{0}^t\nabla_{M}^{(p)}\frac{\partial A(s)}{\partial t}dt  
+ \int_{0}^T\nabla_{M}^{(p)}\frac{\partial A(t)}{\partial t}dt \bigg{\vert} \\
&\leq
(T - t)C_p.
\end{align*}
It follows that as $t \rightarrow T$, 
$\nabla_{M}^{(p)}(A(t)) \rightarrow \nabla_{M}^{(p)}(A_T)$ uniformly. This proves the result.

\end{proof} 

We have an analogous proposition for time dependent spinor fields. As the proof is exactly the same as that
given above, we won't give the proof.

\begin{prop}\label{extend_spinor}
Let $\phi(t)$ denote a sequence of time dependent spinor fields, and $A(t)$ denote a sequence of time dependent
unitary connections,
defined on some time interval  
$[0, T)$, with $T < \infty$. Suppose we have uniform bounds
\begin{equation*}
\sup_{M\times [0, T)}\bigg{\vert} \nabla_{A(t)}^{(p)}\frac{\partial \phi(t)}{\partial t}\bigg{\vert}
\leq C_p
\end{equation*}
for some positive constants $C_p$.

Then $\lim_{t \rightarrow T}\phi(t)$ exists, is smooth, and the sequence $\{\phi(t)\}$ converges to this
limit spinor in every $C^m$-norm, $m \geq 0$. 
\end{prop}

With these two propositions we can now show that the only obstruction to long time existence is curvature 
blow up.

\begin{thm}\label{curvature_blowup}
Suppose $(\phi(t), A(t))$ is a solution to the higher order Seiberg-Witten flow on the maximal time interval 
$[0, T)$, with $T < \infty$. Then
\begin{equation*}
\sup_{M \times [0, T)}\big{\vert}F_{A(t)}\big{\vert} = \infty
\end{equation*}
\end{thm}
\begin{proof}
Suppose on the contrary that $\sup_{M\times [0,T)}\big{\vert}F_{A(t)}\big{\vert} \leq C < \infty$.

Then by proposition \ref{smoothing_estimate_sum_large_time_2} we have uniform derivative bounds 
\begin{align}
\sup_{M\times [0, T)}\big{\vert} \nabla_{M}^{(l)}F_A{(t)}\big{\vert} &\leq C_l \label{CO_1}\\
\sup_{M\times [0, T)}\big{\vert} \nabla_{A(t)}^{(l)}\phi{(t)}\big{\vert} &\leq C_l. \label{CO_2} 
\end{align}
Looking at the second equation in the higher order Seiberg-Witten flow, we have
\begin{align*}
\frac{\partial A}{\partial t} &= (-1)^{k+1}d^*\Delta^{(k)}F_A - \sum_{v=0}^{2k-1}P_1^{(v)}[F_A] - 
2iIm\big{(}\sum_{i=0}^{k}C_i\nabla_M^{*(i)}\langle \nabla_{A}^{(k)}\nabla_{A}\phi, 
\nabla_{A}^{(k-i)}\phi\rangle\big{)}. 
\end{align*}
The two terms $(-1)^{k+1}d^*\Delta^{(k)}F_A$ and $\sum_{v=0}^{2k-1}P_1^{(v)}[F_A]$ both have uniform
derivative bounds coming from \eqref{CO_1}. Furthermore, appealing to lemma \ref{product_estimate}
and \eqref{CO_2}, we see that the term
$\nabla_M^{*(i)}\langle \nabla_{A}^{(k)}\nabla_{A}\phi, \nabla_{A}^{(k-i)}\phi\rangle$ also has uniform
derivative bounds. It follows that $\frac{\partial A}{\partial t}$ has uniform derivative bounds, and hence
by proposition \ref{extend_connection} we can define a smooth limit connection $\lim_{t\rightarrow T}A(t) = A(T)$.

Looking at the first equation in the higher order Seiberg-Witten flow, we have
\begin{align*}
\frac{\partial\phi}{\partial t} &= -\nabla_A^{*(k+1)}\nabla_A^{(k+1)}\phi - \frac{1}{4}(S + |\phi|^2)\phi.
\end{align*}
From \eqref{CO_2}, we see that the two terms on the right of the above equation have uniform derivative bounds.
This implies $\frac{\partial\phi}{\partial t}$ has uniform derivative bounds. Applying proposition 
\ref{extend_spinor}, we get a smooth limiting spinor $\lim_{t\rightarrow T}\phi(t) = \phi(T)$.

We can then apply short time existence with the initial condition $(\phi(T), A(T))$, and extend the solution
$(\phi(t), A(t))$ past the time $T$. However, this contradicts the maximality of $T$. Therefore we must in fact
have that 
$\sup_{M \times [0, T)}\big{\vert}F_{A(t)}\big{\vert} = \infty$, which completes the proof.

\end{proof}

\section{Finite time solutions}\label{blowup_analysis}

In the previous section, theorem \ref{curvature_blowup} showed us that the obstruction to extending
a solution past the maximal time is the curvature $F_A$ blowing up. In this section, we want to show
that under such circumstances one can still obtain information about the singularity present in the flow
through a blow up solution.

We start with some basic properties on scaling a connection and a spinor field.

\begin{defn}
Given a time dependent connection $\nabla_t$, with connection coefficient $\Gamma$. We define the
$\lambda$-scaled connection, $(\nabla_t)^{\lambda}$, to be the connection with connection coefficient
$\Gamma^{\lambda}$, defined by
\begin{equation*}
\Gamma^{\lambda}(x, t) = \lambda\Gamma(\lambda x, \lambda^{2(k+1)}t).
\end{equation*}
\end{defn}

\begin{defn}
Given a time dependent spinor field $\phi$, we define the $\lambda$-scaled spinor field $\phi^{\lambda}$ by,
$\phi^{\lambda}(x,t) = \lambda \phi(\lambda x, \lambda^{2(k+1)}t)$.
\end{defn}

These definitions will be employed while working in a local coordinate chart, and in cases were $\lambda$ is
sufficiently small, so that the dilation $\lambda x$ makes sense within the chart.

We will primarily focus on $\lambda$-scaled unitary connections, $A^{\lambda}$, on the line bundle 
$\mathcal{L}^2$, where $A^{\lambda}(x,t) = \lambda A(\lambda x, \lambda^{2(k+1)}t)$. Recall that associated to
a unitary connection $A$ on $\mathcal{L}^2$, we had the connection $\nabla_A$ defined on the spinor bundle.
Locally, $\nabla_A = d + (\omega + AI)$, where $\omega$ comes from the Levi-Civita connection on $M$. 
Given the scaled connection, $A^{\lambda}$, the connection $\nabla_{A^{\lambda}}$ will denote the scaled
version of $\nabla_{A}$. We are abusing notation slightly as locally, 
$\nabla_{A^{\lambda}} = d + (\omega^{\lambda} + A^{\lambda}I)$, and we point out to the reader that this is not
equal to $d + (\omega + A^{\lambda}I)$. Furthermore, we will also be dealing with scaled versions of the
Levi-Civita connection. We will denote the $\lambda$-scaled Levi-Civita connection by $\nabla_M^{\lambda}$.

Observe that because $F_{A^{\lambda}} = dA^{\lambda}$, we have that 
$F_{A^{\lambda}}(x, t)= \lambda^2F_{A}(\lambda x, \lambda^{2(k+1)}t)$, so the curvature scales quadratically 
in $\lambda$.

We now want to understand how the derivative terms in the 
higher order Seiberg-Witten equations scale. We start by computing time derivatives of the scaled connection and 
spinor field
\begin{align*}
\frac{\partial A^{\lambda}}{\partial t}(x, t) &= 
\lambda^{2k + 3}\frac{\partial A}{\partial t}(\lambda x, \lambda^{2(k+1)}t) \\ 
\frac{\partial \phi^{\lambda}}{\partial t}(x, t)  &=
\lambda^{2k + 3}\frac{\partial \phi}{\partial t}(\lambda x, \lambda^{2(k+1)}t).
\end{align*}  

We want to show that this scaling by $\lambda^{2k+3}$ holds for the derivative terms on the right hand side of the 
higher order Seiberg-Witten flow. 

The term $\nabla_{A}^{*(k+1)}\nabla_{A}^{(k+1)}\phi$ scales as
$\nabla_{A^{\lambda}}^{*(k+1)}\nabla_{A^{\lambda}}^{(k+1)}\phi^{\lambda}(x, t)
= \lambda^{2k+3}\nabla^{*(k+1)}_{A}\nabla^{(k+1)}_{A}\phi(\lambda x, \lambda^{2(k+1)})$. 

We know that the term
\begin{equation*}
(-1)^{k+1}d^*\Delta^{(k)}F_A - \sum_{v=0}^{2k-1}P_1^{(v)}[F_A]
\end{equation*}
can be written as $d^*\nabla_M^{*(k)}\nabla_M^{(k)}F_A$ (see \eqref{SW_3}). 
The term $d^*\nabla_M^{*(k)}\nabla_M^{(k)}F_A$ scales as
\begin{equation*}
d^*(\nabla_M^{\lambda})^{*(k)}(\nabla_M^{\lambda})^{(k)}F_{A^{\lambda}}(x, t)
=\lambda^{2k+3}d^*\nabla_M^{*(k)}\nabla_M^{(k)}F_A(\lambda x, \lambda^{2(k+1)}t).
\end{equation*}
It follows that
\begin{equation*}
(-1)^{k+1}d^*(\Delta^{\lambda})^{(k)}F_{A^{\lambda}} - \sum_{v=0}^{2k-1}P_1^{(v)}[F_{A^{\lambda}}] =
 \lambda^{2k+3}(-1)^{k+1}d^*\Delta^{(k)}F_{A} - \sum_{v=0}^{2k-1}P_1^{(v)}[F_{A}].
\end{equation*}
Finally, if we look at the term 
$\nabla_M^{*(i)}\langle \nabla_{A}^{(k)}\nabla_{A}\phi, 
\nabla_{A}^{(k-i)}\phi\rangle\big{)}$ it is easy to see that 
\begin{equation*}
(\nabla_M^{\lambda})^{*(i)}\langle \nabla_{A^{\lambda}}^{(k)}\nabla_{A^{\lambda}}\phi^{\lambda}, 
\nabla_{A^{\lambda}}^{(k-i)}\phi^{\lambda}\rangle\big{)}\big{\rangle} =
\lambda^{2k+3}\nabla_M^{*(i)}\langle \nabla_{A}^{(k)}\nabla_{A}\phi, 
\nabla_{A}^{(k-i)}\phi\rangle\big{)}\big{\rangle}.
\end{equation*}
From this discussion, we immediately get the following proposition

\begin{prop}\label{scaled_system}
Let $(\phi(t), A(t))$ be a solution to the higher order Seiberg-Witten flow on $[0, T)$. Then
$(\phi^{\lambda}, A^\lambda)$ is a solution to the following scaled system
\begin{align}
\frac{\partial \phi^{\lambda}}{\partial t} &= 
-\nabla_{A^{\lambda}}^{*(k+1)}\nabla_{A^{\lambda}}^{(k+1)}\phi^{\lambda} - 
\frac{\lambda^{2k}}{4}(\lambda^2S + \vert\phi^{\lambda}\vert^2)\phi^{\lambda} \label{scaled_system_1}\\
\frac{\partial A^{\lambda}}{\partial t} &=
(-1)^{k+1}d^*(\Delta^{\lambda})^{(k)}F_{A^{\lambda}} - \sum_{v=0}^{2k-1}P_1^{(v)}[F_{A^{\lambda}}] 
-
2iIm\big{(}\sum_{i=0}^{k}C_i(\nabla_M^{\lambda})^{*(i)}\langle \nabla_{A^{\lambda}}^{(k)}\nabla_{A^{\lambda}}\phi^{\lambda}, 
\nabla_{A^{\lambda}}^{(k-i)}\phi^{\lambda}\rangle\big{)} \label{scaled_system_2}
\end{align}
on the time interval $[0, \frac{1}{\lambda^{2(k+1)}}T)$.
\end{prop}

We will call the above scaled system a generalised higher order Seiberg-Witten flow. 

We will now show that in the case that the curvature form 
is blowing up, as one approaches the maximal time, a blow up limit
can be extracted. The proof of the theorem will closely follow the proof of proposition 3.24 in
\cite{kelleher}, and the proof of lemma 4.6 in \cite{hong}.

\begin{thm}\label{blowup_limit}
Let $(\phi(t), A(t))$ be a solution to the higher order Seiberg-Witten flow, on some maximal time interval
$[0, T)$, with $T < \infty$.
Then there exists a blow up sequence $(\phi^{i}(t), A^i(t))$, that converges pointwise, upto gauge transformations, 
to a smooth solution $(\mathcal{\phi}^{\infty}(t), \mathcal{A}^{\infty}(t))$ 
of the higher order Seiberg-Witten flow, with
domain $\R^n \times (-\infty, 0)$. 
\end{thm}
\begin{proof}
By theorem \ref{curvature_blowup}, we must have that $\lim_{t\rightarrow T}\sup_{M}\vert F_{A}\vert = \infty$. 
Therefore, we can choose a sequence of times $t_i$, such that $t_i \rightarrow T$, and a sequence of points 
$x_i$, such that 
\begin{equation*}
\vert F_{A(t_i)}(x_i)\vert = \sup_{M \times [0, t_i]}\vert F_{A_t}\vert.
\end{equation*}
By compactness of $M$, we can assume $x_i \rightarrow x_{\infty}$.

Fix a chart $U$ about $x_{\infty}$ and, without loss of generality, assume that $U$ gets mapped to
$B_1(0) \subseteq \R^n$, with $x_{\infty}$ mapping to $0$.
We will be considering the behaviour of the solution for points $(x_i, t_i)$ for $i$ sufficiently large.
Therefore, using this chart, we can assume the points $x_i$ are in $\R^n$, and are converging to $0$.
 
We define 
\begin{align*}
A^i(x, t) &= \lambda_i^{\frac{1}{2(k+1)}}A(\lambda_i^{\frac{1}{2(k+1)}}x + x_i, \lambda_it + t_i) \\
\phi^i(x, t) &= \lambda_i^{\frac{1}{2(k+1)}}\phi(\lambda_i^{\frac{1}{2(k+1)}}x + x_i, \lambda_it + t_i) 
\end{align*}
where $\lambda_i$ are positive numbers to be determined. The domain of $(\phi^i, A^i)$ is 
$B_{\lambda^{\frac{-1}{2(k+1)}}}(x_i) \times [\frac{-t_i}{\lambda_i}, \frac{T-t_i}{\lambda_i}]$. 
Furthermore, it easy to see that the pair $(\phi^i, A^i)$ satisfy a generalised higher order Seiberg-Witten flow,
with scale factor $\lambda_i^{\frac{1}{2(k+1)}}$. In fact, by defining $(\phi^i, A^i)$ for 
times $t \leq \frac{-t_i}{\lambda_i}$ by $A^i(\frac{-t_i}{\lambda_i})$, and similarly for $\phi^i$, we can extend
the domain of $(\phi^i, A^i)$ to 
$B_{\lambda^{\frac{-1}{2(k+1)}}}(x_i) \times (-\infty, \frac{T-t_i}{\lambda_i}]$.

We then observe that, $F^i(x,t) = F_{A^i}(x,t) = \lambda_i^{\frac{1}{k+1}}
F_{A}(\lambda_i^{\frac{1}{2(k+1)}}x + x_i, \lambda_it + t_i)$, which implies
\begin{align*}
\sup_{t \in [\frac{-t_i}{\lambda_i}, \frac{T-t_i}{\lambda_i}]}\vert F^i(x, t)\vert &=
\vert  \lambda_i^{\frac{1}{k+1}}\vert 
\sup_{t \in  [\frac{-t_i}{\lambda_i}, \frac{T-t_i}{\lambda_i}]}
\vert F_{A}(\lambda_i^{\frac{1}{2(k+1)}}x + x_i, \lambda_it + t_i)\vert \\
&= \vert  \lambda_i^{\frac{1}{k+1}}\vert \sup_{t \in [0, t_i]}\vert F_{A}(x,t)\vert \\
&= \vert  \lambda_i^{\frac{1}{k+1}}\vert \vert F(x_i, t_i)\vert.
\end{align*}
Therefore, defining $\lambda_i = \vert F_A(x_i, t_i)\vert$ we find
\begin{align}
\sup_{t \in [\frac{-t_i}{\lambda_i}, 0]}\vert F_{A^i}(x)\vert = 1. \label{blowup_1}
\end{align}
We thus see that the sequence $A^i$ represents a blow up sequence. We now have to show that we can extract
an actual blow up limit. Before we show how to do this, we point out to the reader that, 
by definition, $\frac{1}{\lambda_i^{\frac{1}{2(k+1)}}} \rightarrow \infty$, as $i \rightarrow \infty$. This means
that the domains, 
$B_{\lambda^{\frac{-1}{2(k+1)}}}(x_i) \times (-\infty, \frac{T-t_i}{\lambda_i}]$, will expand to
$\R^n \times (-\infty, 0)$.

We also observe that at each time $t \leq 0$ in the domain of definition of $F_{A^i}$, we have uniform derivative
bounds. To see this, take $y \in \R^n$, and take $i$ large enough so that
$B_{2r}(y) \times [t-1, t]$ is in the domain of definition of $(\phi^i, A^i)$ for some $r > 0$. Then take
a bump function $\gamma$, supported in $B_{2r}(y)$, so that $\gamma = 1$ on $B_{r}(y)$. 
Since $\sup\vert F_{A^i(t)}\vert = 1$, where the $\sup$ is
taken over the domain of definition of $A^i$, we have that $\sup\vert \gamma^{s}F_{A^i(t)}\vert \leq 1$.
Applying
proposition \ref{smoothing_estimate_sum_large_time_2},
we then see that there exists $C_l$ so that
\begin{align}
\sup_{B_r(y)}\vert \nabla_M^{(l)}F_{A^i(t)}\vert \leq 
\sup_{B_{2r}(y)}\vert\gamma^{s/2}\nabla_M^{(l)}F_{A^i(t)}\vert  \leq C_l. \label{blowup_2}
\end{align}
If we had another point $\tilde{y}$, then we could apply the same argument to $B_{2r}(\tilde{y})$, and obtain
the exact same uniform derivative bound. This means we have uniform bounds for  
$\vert \nabla_M^{(l)}F_{A^i(t)}\vert$ for all $i$, and all $l$.

Like we did for the curvature above, we want to show that we have derivative bounds for
the connections $A^i$. With these bounds, we can then apply the Arzela-Ascoli theorem to extract a limit
connection, which will then serve as the blow up limit. In order to do this, we will need to change gauge, obtain
the bounds in that gauge, and then transform back.

Before we explain how to put the above remark into action, let us explain what is going on with 
the spinor fields $\phi^i$. We know that $\phi(t)$ is uniformly bounded along the flow by 
proposition \ref{spinor_bounded_along_flow}. Therefore, since $\lambda_i \rightarrow 0$ as 
$i \rightarrow \infty$, it follows that $\phi^i \rightarrow 0$ as $i \rightarrow \infty$. 
What this means is that, any blow up limit we can obtain from the blow up sequence $(\phi^i, A^i)$ 
will necessarily have the limit spinor field being $0$. Hence, we need only deal with $A^i$ when 
we want to extract a blow up limit.

Fix $r > 0$ sufficiently large, fix $\tau < 0$, and $m \in \mathbb{N}$. Then for all $i$ sufficiently large, we have 
that the domain of $A^i$ contains $B_{2r + m} \times [\tau - m, \frac{-1}{m}]$.
The $F_{A^i(t)}$ are all uniformly bounded by $1$. Therefore, we can find some $\delta > 0$ such that
\begin{align*}
\vert\vert F_{A^i(t)}\vert\vert_{L^{n/2}(B_\delta(y))}\vert\vert \leq \kappa_n
\end{align*} 
where $i$ is taken so that $B_{\delta}(y)$ is in the domain of $A^i$, and $y \in \R^n$ is in the domain of
$A^i$. The constant $\kappa_n$ comes from the statement of the Coloumb gauge theorem, see theorem 
\ref{coloumb_gauge}.
We then map
\begin{equation*}
B_{\delta}(y) \ni x \rightarrow \frac{x-y}{\delta} \in B_1(0)
\end{equation*}
i.e. we translate $B_{\delta}(y)$ to $B_{\delta}(0)$ and then scale by $\frac{1}{\delta}$. What we want to do
is use the Coloumb gauge theorem to get good bounds on the $A^i$. The problem is that the Coloumb gauge theorem, 
theorem \ref{coloumb_gauge}, 
requires a curvature bound of the above type on $B_1(0)$. Therefore, we need to scale everything by 
$\frac{1}{\delta}$.

We define $\delta$-scaled connections 
$\tilde{A}^i(x, t) = \delta A^i(\delta x + y, \delta^{2(k+1)}t)$ for $x \in B_1(0)$. It is easy to see then that
the associated curvature $F_{\tilde{A}^i}$ satisfy the bound  
\begin{align*}
\vert\vert F_{\tilde{A}^i(t)}\vert\vert_{L^{n/2}(B_1(0))}\vert\vert \leq \kappa_n.
\end{align*} 
Also note, that if we let $\tilde{\phi}^i$ denote the $\delta$-scaled spinor fields, then the pair
$(\tilde{\phi}^i, \tilde{A}^i)$, satisfy a generalised higher order Seiberg-Witten flow, with scaling term
$\delta$. Furthermore, $(\tilde{\phi}^i, \tilde{A}^i)$ is defined on 
$B_{1}(0) \times [\frac{\tau-m }{\delta^{2(k+1)}}, \frac{-1}{\delta{2(k+1)}m}]$.

We then apply the Coloumb gauge theorem, theorem \ref{coloumb_gauge}, to the connections 
$\tilde{A}^i(x, t)$, where $t \in  [\frac{\tau-m }{\delta^{2(k+1)}}, \frac{-1}{\delta{2(k+1)}m}]$. 
In doing so, we get connections 
$\tilde{\mathcal{A}}^i(x, t)$ defined on $B_1(0)$, and by (2) of the Coloumb gauge theorem,
we have that there exists $c_n$ such that
\begin{equation*}
\big{\vert}\big{\vert} \tilde{\mathcal{A}}^i(x, t)
\big{\vert}\big{\vert}_{C^{p,1}(B_1(0))} \leq c_n(t)
\end{equation*}
where $p \geq n/2$. By compactness of the interval 
$[\frac{\tau-m }{\delta^{2(k+1)}}, \frac{-1}{\delta{2(k+1)}m}]$, we can get a bound of the form
\begin{equation*}
\sup_{B_1(0)\times [\frac{\tau-m }{\delta^{2(k+1)}}, \frac{-1}{\delta{2(k+1)}m}]}
\big{\vert}\big{\vert} \tilde{\mathcal{A}}^i(x, t)
\big{\vert}\big{\vert}_{C^{p,1}(B_1(0))} \leq c_n(\tau-m).
\end{equation*}
Note that, because the curvature corresponding to a unitary connection is invariant under gauge transformations, we
have that the curvature corresponding to $\tilde{\mathcal{A}}^i$ is equal to $F_{\tilde{A}^i}$. Since
$\tilde{A}^i$ is just a scaled version of $A^i$ we have that  $F_{\tilde{A}^i}$ is just a
$\delta^2$ scaling of $F_{A^i}$. This means that the curvatures of $\tilde{\mathcal{A}}^i$ 
also have uniform derivative bounds, just like $F_{A^i}$ did. In this gauge, we denote the spinor fields
by $\tilde{\Phi}^i$.
 
We now want to map $B_1(0)$ back to $B_{\delta}(y)$ by mapping 
$B_1(0) \ni x \rightarrow \delta x + t \in B_{\delta}(y)$, and then scale $\tilde{\mathcal{A}}^i$ by defining
\begin{equation*}
\mathcal{A}^i(t, x) = \frac{1}{\delta}\tilde{\mathcal{A}}^i\big{(}\frac{x-y}{\delta}, \delta^{-2(k+1)}t\big{)}.
\end{equation*}    
We then have
\begin{equation*}
\sup_{B_{\delta}(y) \times [\tau - m, \frac{-1}{m}]}\vert\vert \mathcal{A}^i\vert\vert \leq \delta c_n(\tau -m).
\end{equation*} 
We denote the $\delta$ dilated $\tilde{\Phi}^i$, by $\Phi^i$.

Note that because of its construction, $\mathcal{A}^i$ is gauge equivalent to $A^i$, and $\Phi^i$ is gauge
equivalent to $\phi^i$. Therefore, the pair $(\Phi^i, \mathcal{A}^i)$ satisfy a generalised 
higher order Seiberg-Witten flow.

The connections $\mathcal{A}^i$ are defined on $B_{y}(\delta)$. However, taking any other point $\tilde{y}$, we can
run the same argument above and obtain a connection satisfying the same bounds on $B_{\delta}(\tilde{y})$.
What this means is that, if we take a collection of points $y_1, \ldots ,y_n$ so that
\begin{equation*}
B_{2r+m}(0) \supseteq \bigcup_{i=1}^nB_{\delta}(y_i) \supseteq B_{r+m}(0)
\end{equation*}
We then obtain connections $\mathcal{A}^i_1, \ldots , \mathcal{A}^i_n$ on each $B_{\delta}(y_i)$. As the
Coloumb gauge is defined on $B_{2r+m}(0)$, we can then apply theorem \ref{gauge_patching},
to obtain a single $\mathcal{A}^i$ that is defined
on all of $B_{r+m}(0)$.

This means we have a sequence of connections $\mathcal{A}^i$ admitting uniform $C^{p,1}$ bounds, for 
$p \geq n/2$, on $B_{r+m} \times [\tau-m, \frac{-1}{m}]$. We now want to show that for each $m$, we can extract
a limit connection, defined on $B_{r+m} \times [\tau-m, \frac{-1}{m}]$. 

Fix $p \geq n/2$, $m \in \mathbb{N}$, and $0 < \alpha < 1$. From the fact that we have uniform $C^{p,1}$ bounds
for $\mathcal{A}^i$ , and the fact
that $\alpha < 1$. We see that if we apply the Arzela-Ascoli theorem, we can extract a limit 
$\mathcal{A}^{m, \infty}_{p}$, which is defined on $B_{r+m}(0) \times [\tau - m , \frac{-1}{m}]$. 

If we took another $q > p \geq n/2$, and applied the above to obtain limits 
$\mathcal{A}^{m, \infty}_{q}$ and $\mathcal{A}^{m, \infty}_{p}$. Then we would in fact have that
$\mathcal{A}^{m, \infty}_{q} = \mathcal{A}^{m, \infty}_{p}$, as $C^{q,1} \subseteq \C^{p,1}$ as topological spaces.
Therefore, applying the above for each $p \geq n/2$, we get a limit $\mathcal{A}^{m, \infty}$ in $C^{\infty}$, 
defined on $B_{r+m} \times [\tau-m, \frac{-1}{m}]$, for each $m \in \mathbb{N}$.
The final step is to show that we can extract a limit defined on all of $\R^n \times (-\infty, 0)$. In order to
do this, we apply the same procedure as above, but then extract a diagonal limit.  
 
We start by denoting the sequence $\mathcal{A}^i$ on $B_{r+m} \times [\tau-m, \frac{-1}{m}]$ by 
$\mathcal{A}^{m, i}$. If we fix $p \geq n/2$ and $0 < \alpha < 1$, Arzela-Ascoli tells us that, passing to
a subsequence if necessary,
$\mathcal{A}^{m, i} \rightarrow \mathcal{A}^{m, \infty}_{p}$. Doing this for each $p \geq n/2$, we obtain
a limit $\mathcal{A}^{m, i} \rightarrow \mathcal{A}^{m, \infty}$ in $C^{\infty}$.

We then consider the diagonal sequence: 
$\mathcal{A}^{1, 1}, \mathcal{A}^{2, 2}, \ldots ,\mathcal{A}^{m, m},\ldots$. This sequence converges on
any compact subset of $\R^n \times (-\infty, 0)$ to a connection $\mathcal{A}^{\infty}$, which is the
required blow up limit of the $A^i$.

We remind the reader that we already handled the structure of the blow up limit of the $\phi^i$. Namely, we 
saw that the limit was just $0$. Together with the above, we see that our blow up limit is
$(0, \mathcal{A}^{\infty})$. It is also easy to see that this blow up limit satisfies the higher order
Seiberg-Witten flow on $\R^n \times (-\infty, 0)$. 

We also point out that, if we let $\mathcal{F}^{\infty}$ denote the curvature associated to
$\mathcal{A}^{\infty}$, then by \eqref{blowup_1} we have
\begin{align*}
\lim_{t\rightarrow 0}\sup_{\R^n}\vert \mathcal{F}^{\infty}(x, t)\vert = 1
\end{align*}
and that by \eqref{blowup_2}, $\mathcal{F}^{\infty}$ has uniform derivative bounds.

\end{proof}

\section{Long time existence results}\label{long_time}

We prove long time existence for solutions to the flow in sub-critical dimensions, and then show that
in the critical dimension, long time existence is obstructed by an $L^{k+2}$ curvature concentration
phenomenon.

\subsection{Long time existence for subcritical dimensions}

We start with the following proposition.

\begin{prop}\label{subcritical_L^p_estimate}
Let $dim M = n < 2p$, and suppose $(\phi(t), A(t))$ is a solution to the higher order Seiberg-Witten flow, on
$[0, T)$ where $T \leq \infty$. Assume $F_{A(t)} \in L^{\infty}([0, T); L^p(M))$, then
$F_{A(t)} \in  L^{\infty}([0, T); L^{\infty}(M))$. In particular, $T = \infty$.
\end{prop}
\begin{proof}
So as to obtain a contradiction, assume $\sup_{[0,T)}\vert\vert F_{A}\vert\vert_{\infty} = \infty$. As we did 
in theorem \ref{blowup_limit}, we can then construct a blowup sequence $(\phi^i, A^i)$, with blow up limit
$(\mathcal{\phi}^{\infty}, \mathcal{A}^{\infty})$. The curvature of $A^i$ was given by
\begin{align*}
F_{A^i} = \lambda_i^{\frac{1}{k+1}}F(\lambda_i^{\frac{1}{2(k+1)}}x + x_i, \lambda_it + t_i)
\end{align*}
where $\lambda_i = \vert F(x_i, t_i)\vert^{-(k+1)}$.

We also know, by \eqref{blowup_2}, that the limit curvature $\mathcal{F}^{\infty}$ satisfies
\begin{align*}
\vert\vert\mathcal{F}^{\infty}\vert\vert^p_{L^p} \neq 0.
\end{align*}
Applying Fatou's lemma we have
\begin{align*}
\vert\vert\mathcal{F}^{\infty}\vert\vert^p_{L^p} &\leq 
\liminf_{i \rightarrow \infty}\vert\vert F_{A^i}\vert\vert^p_{L^p} \\
&\leq \lim_{i \rightarrow \infty}\lambda_i^{\frac{2p-n}{2k+2}}\vert\vert F_{A}\vert\vert^p_{L^p}. 
\end{align*}
We know that, $\lambda_i \rightarrow 0$ as $i \rightarrow \infty$. Furthermore, because $2p > n$, by assumption, 
we have that the right hand side of the above inequality goes to zero. But this is a contradiction.

\end{proof}

Using this result, we can prove long time existence in the sub-critical dimension i.e. for $dim M < 2(k+2)$.

\begin{thm}\label{main_theorem_1}
Let $(\phi(0), A(0))$ be a given initial condition. Suppose $dim M < 2(k+2)$. Then there exists a unique solution
$(\phi(t), A(t))$, with initial condition $(\phi(0), A(0))$, that exists for all time $t > 0$.
\end{thm}
\begin{proof}
By short time existence, we have that a unique solution $(\phi(t), A(t))$ exists, with initial condition
$(\phi(0), A(0))$, on some maximal time interval $[0, T)$. If $T = \infty$, there is nothing to prove, so
assume $T < \infty$.

By the Sobolev embedding theorem, we have that $W^{k,2}$ embeds continuously into $L^p$ if
$\frac{1}{p} = \frac{1}{2} - \frac{k}{n}$. If also add the condition that $\frac{n}{2} < p$, then we 
must have $n < 2(k+2)$.

Applying the Sobolev embedding theorem we get
\begin{align*}
\vert\vert F_t\vert\vert_{L^p} &\leq C_{k, 2}\bigg{(} \sum_{j=0}^k\vert\vert\nabla^{(j)}
F_A\vert\vert^2_{L^2}\bigg{)} \\
&\leq CC_{k,2}\bigg{(} \vert\vert F_A\vert\vert^2_{L^2} + \vert\vert\nabla^{(k+1)}F_A\vert\vert^2_{L^2}
\bigg{)}.
\end{align*} 
where to obtain the second inequality, we have applied lemma \ref{interp_3}.

By lemmas \ref{energy-k} and \ref{SW_finite_energy}, we know that
the Seiberg-Witten energy and the higher order Seiberg-Witten energy are bounded along the flow. 
We then have that the left hand side of the
above inequality is bounded along the flow. 

Proposition \ref{subcritical_L^p_estimate} then implies, 
$F_{A(t)} \in  L^{\infty}([0, T); L^{\infty}(M))$. This means we can extend this solution
past $T$, but this contradicts maximality of $T$. Therefore we must in fact have that $T = \infty$.

\end{proof}

The above theorem can be seen as an analogue of the first part of theorem 7.8 in \cite{mantegazza}, and theorem A in
\cite{kelleher}, for the case of these higher order Seiberg-Witten functionals.

\subsection{Curvature concentration in the critical dimension}

As was seen in the above subsection, long time existence for the sub-critical dimensions is quite
straightforward to prove. Unfortunately, the above technique breaks down in the critical dimension. 
The main issue, as we will see shortly, is that in the critical dimension curvature can start to concentrate
in smaller and smaller balls, and this in turn obstructs one from being able obtain a solution for all time.

\begin{prop}\label{curvature_concentration}
Suppose $dim M = n = 2p$, and $(\phi(t), A(t))$ is a solution to the higher order Seiberg-Witten flow, on 
$[0, T)$, with $T < \infty$. If $x_0 \in M$ is such that,
\begin{equation*}
\limsup_{t \rightarrow T}\vert F_{A(t)}(x_0)\vert = \infty.
\end{equation*}
Then there exists some $\epsilon > 0$ such that, for all $r > 0$ we have
\begin{equation*}
\lim_{t \rightarrow T}\vert\vert F_{A_t}\vert\vert_{L^p(B_r(x_0))} \geq \epsilon.
\end{equation*}
\end{prop}
\begin{proof}
As in the proof of theorem \ref{blowup_limit}, we pick a sequence of times $t_i$ so that
$\sup_{M\times [0, t_i]}\vert F_{A}\vert = \vert F_A(x_0, t_i)\vert$, with $t_i \rightarrow T$.

We then let $(\phi^i, A^i)$ be the associated blowup sequence, and 
$(\mathcal{\phi}^{\infty}, \mathcal{A}^{\infty})$ the associated blowup limit, defined on 
$\R^n \times (-\infty, 0)$. Recall from theorem \ref{blowup_limit}, we saw that 
$\lim_{t\rightarrow 0}\vert \mathcal{F}^{\infty}(0, t)\vert = 1$. This means that we can find a 
$\delta > 0$ such that, for $(x,t) \in B_{\delta}(0) \times (-\delta, 0]$ we have
\begin{equation*}
\vert \mathcal{F}^{\infty}(x,t)\vert \geq \Lambda
\end{equation*}
where $\Lambda$ is any constant slightly less than $1$, for example take $\Lambda = 1 - \lambda$ for 
$\lambda > 0$ sufficiently small.

Using this we find
\begin{align*}
\lim_{t\rightarrow 0}\vert\vert \mathcal{F}^{\infty}\vert\vert^p_{L^p(B_{\delta}(0))} &= 
\lim_{t\rightarrow 0} \int_{B_{\delta}(0)}\vert \mathcal{F}^{\infty}(x,t)\vert^pdx \\
&\geq \Lambda^pVol(B_{\delta}(0)). 
\end{align*}
Now, fix $r > 0$. If $\lim_{t\rightarrow T}\vert\vert F_{A}\vert\vert_{L^p(B_{\delta}(x_0))} = \infty$, then
there is nothing to prove and we are done. Therefore, assume 
$\lim_{t\rightarrow T}\vert\vert F_{A}\vert\vert_{L^p(B_{\delta}(x_0))} < \infty$.

We compute
\begin{align*}
\vert\vert \mathcal{F}^{\infty}(x,t) \vert\vert^{p}_{L^p(B_{\delta}(0))} 
&= \int_{B_{\delta}(0)}\vert \mathcal{F}^{\infty}(x, t)\vert^pdx \\
&= \int_{B_{\delta}(0)}\lim_{i\rightarrow \infty} \vert F_{A^i}(x, t)\vert^pdx \\
&= \lim_{i\rightarrow \infty}\int_{B_{\delta}(0)} \bigg{\vert}
\lambda_i^{\frac{1}{2(k+1)}}F_{A}(\lambda_i^{\frac{1}{2(k+1)}}x + x_0, \lambda_it + t_i)\bigg{\vert}^pdx \\
&= \lim_{i\rightarrow \infty}\int_{B_{\delta\lambda_i^{1/(2k+2)}}(x_0)}
\lambda_i^{\frac{2p-n}{2(k+1)}}\vert F(z, \lambda_it + t_i)\vert^p dz \\
&= \lim_{i\rightarrow \infty}\int_{B_{\delta\lambda_i^{1/(2k+2)}}(x_0)}
\vert F(z, \lambda_it + t_i)\vert^p dz \\
&\leq \lim_{i\rightarrow \infty}\int_{B_{r}(x_0)}
\vert F(z, \lambda_it + t_i)\vert^p dz \\
&= \lim_{t\rightarrow T}\vert\vert F_{A}\vert\vert^p_{L^p(B_r(x_0))}.
\end{align*}
Therefore, we obtain 
\begin{align*}
\lim_{t\rightarrow 0}\vert\vert \mathcal{F}^{\infty}(x,t) \vert\vert^{p}_{L^p(B_{\delta}(0))} \leq 
\lim_{t\rightarrow T}\vert\vert F_{A}\vert\vert^p_{L^p(B_r(x_0))}
\end{align*}
which in turn gives
\begin{align*}
\Lambda (Vol(B_{\delta}(0)))^{1/p} \leq \lim_{t\rightarrow T}\vert\vert F_{A}\vert\vert^p_{L^p(B_r(x_0))}.
\end{align*}
Taking $\epsilon = \Lambda (Vol(B_{\delta}(0)))^{1/p}$ finishes the proof.

\end{proof}

We can now prove our second main theorem.

\begin{thm}\label{main_theorem_2}
Let $(\phi(0), A(0))$ be an initial condition, and suppose $dim M = 2(k+2)$. Then 
\begin{itemize}
\item[1.] there exists a unique
solution to the higher order Seiberg-Witten flow, on a maximal time interval $[0, T)$, with
$T \leq \infty$.

\item[2.] If $T < \infty$, then $\limsup_{t\rightarrow T}\vert\vert F_{A}\vert\vert_{\infty} = \infty$, 
and there exists $x_0 \in M$ satisfying the following $L^{k+2}$-curvature concentration phenomenon:
There exists $\epsilon > 0$, such that for all $r > 0$ we have
\begin{equation*}
\lim_{t \rightarrow T}\vert\vert F_{A_t}\vert\vert_{L^{k+2}(B_r(x_0))} \geq \epsilon.
\end{equation*}
Moreover, the number of points where such a concentration can occur is finite.
\end{itemize}
\end{thm}
\begin{proof}
The proof of 1. follows from short time existence. The first part of 2. follows from theorem 
\ref{curvature_blowup}, and the concentration of curvature phenomenon follows from 
proposition \ref{curvature_concentration}. Therefore, we need only prove that such a phenomenon can take 
place at most at a finite number of points.

To see this let $p = k+2$, and apply the Sobolev embedding theorem to get an embedding 
$W^{k,2} \subseteq L^p$. Then
\begin{align*}
\vert\vert F_A\vert\vert_{L^p} &\leq C_{k,2}\bigg{(}\sum_{j=0}^k\vert\vert \nabla_M^{(j)}\vert\vert_{L^2}\bigg{)} \\
&\leq S_{k,2}C(\vert\vert F_A\vert\vert_{L^2} + \vert\vert \nabla_M^{(k+1)}F_{A}\vert\vert_{L^2}
\end{align*}
where $C_{k,2}$ denotes the constant in the Sobolev inequality, and where
we get the second inequality by applying lemma \ref{interp_3}.

The right hand side of the above inequality is bounded in time by lemma \ref{SW_finite_energy},
which in turn implies the left hand side is 
bounded as $t \rightarrow T$. The result follows.

\end{proof}

The above theorem shows that in the critical dimension, long time existence is obstructed by the possibility of
the curvature form concentrating in smaller and smaller balls. This is analogous to what Struwe observed for the
Yang-Mills flow in dimension four (see theorem 2.3 in \cite{struwe}), and what Kelleher observed for the
higher order Yang-Mills flow in the critical dimension (see theorem B in \cite{kelleher}).

\section{concluding remarks}\label{conclusion}

Theorem \ref{main_theorem_1} tells us that, provided the order of derivatives, appearing in the higher order
Seiberg-Witten functional, is sufficiently large, solutions to the associated gradient flow do not hit any
finite time singularities. On the other hand, theorem \ref{main_theorem_2} tells us that if the dimension
of $M$ is equal to the critical dimension, then there is a possibility of finite time singularities, due
to the $L^{k+2}$ energy of the curvature form concentrating in smaller and smaller balls. The theorem in fact
proves that the points where this energy concentration can happen, must be finite in number. The question
then remains, is it possible that there are in fact no such points?

In the case of the Seiberg-Witten flow, the critical dimension is dimension four. Hong and Schrabrun show
that if long time existence is obstructed then again it is due to an energy concentration phenomenon, but this
time the energy is an $L^2$ energy. 
Using a rescaling argument, similar to what we did in \ref{blowup_limit}, they are able to show that one can 
extract a limiting curvature form. They then show, by using an $L^2$ energy estimate, that this implies the
limiting curvature form must be harmonic. 
Using the mean
value formula for harmonic forms, they are then able to derive a contradiction, and show that the $L^2$
energy of the curvature form cannot concentrate in smaller and smaller balls. 

The key point to note is that for them, everything is taking place in $L^2$. 
Therefore, the $L^2$ energy estimates they derive
are robust enough to obtain information about a limiting curvature form. In our case, we have that curvature
is potentially concentrating in $L^{k+2}$. This fact, that in these higher order flows 
curvature concentration takes place in higher $L^p$ spaces, makes the approach taken by Hong and
Schrabrun inadequate for these higher order flows. It becomes a challenge as to whether one can obtain
suitable $L^{k+2}$ estimates, that could possible lead to ruling out curvature concentration in the
critical dimension.

\section{Appendix}\label{appendix}

In the following appendix, we gather together various theorems from other resources that we will be using 
in the paper.

\subsection{Interpolation inequalities}

The following interpolation results will be used in section \ref{local_estimates}, when proving
local derivative estimates.

We will need the following theorem, which is theorem 5.4 in \cite{kuwert}.

\begin{thm}\label{interp_1}
Let $\phi$ be a section of a vector bundle $E$ over $M$, with connection $\nabla$, and let $\gamma$ be a 
bump function on $M$.
For $k \in \mathbb{N}$, $1 \leq i \leq k$ and $s \geq 2k$ we have the identity
\begin{equation*}
\bigg{(} \int_{M} \vert \nabla^{(i)}\phi\vert^{\frac{2k}{i}}\gamma^sd\mu\bigg{)}^{\frac{i}{2k}} \leq
C\vert\vert\phi\vert\vert^{1-\frac{i}{k}}_{\infty}\bigg{(}\bigg{(}
\int_M\vert\nabla^{(k)}\phi\vert^2\gamma^sd\mu\bigg{)}^{\frac{1}{2}} + \vert\vert\phi\vert\vert_{L^2, \gamma>0}
\bigg{)}^{\frac{i}{k}}
\end{equation*}
where $C = C(g, \gamma, s, n)$.
\end{thm}

An immediate corollary of the above is the following, see corollary 5.5 in \cite{kuwert}.

\begin{cor}\label{interp_2}
Under the same assumptions as the above theorem. Let $0 \leq i_1, \ldots ,i_r\leq k$, $i_1 + \ldots + i_r = 2k$, and
$s \geq 2k$. Then we have
\begin{align*}
\bigg{\vert}\int_M \nabla^{(i_1)}\phi * \ldots * \nabla^{(i_1)}\phi \gamma^sd\mu \bigg{\vert} &\leq
 C_0\int_M \vert\nabla^{(i_1)}\phi\vert \cdots  \vert\nabla^{(i_1)}\phi\vert \gamma^sd\mu \\
 &\leq
 C\vert\vert\phi\vert\vert^{r-2}_{\infty}\bigg{(}\int_M\vert\nabla^{(k)}\phi\vert^2\gamma^sd\mu + 
 \vert\vert\phi\vert\vert^2_{L^2, \gamma>0}\bigg{)}
\end{align*}
where $C_0 = C_0(g)$ depends only on the metric, and $C = C(n, k, r, s, g, \gamma)$.
\end{cor}

Finally, we will need the following interpolation result, see corollary 5.5 in \cite{kelleher}, and corollary
5.3 in \cite{kuwert}.

\begin{lem}\label{interp_3}
Let $E$ be a vector bundle over $M$, $\nabla$ a connection on $E$, and $\gamma$ a bump function on $M$. 
For $2 \leq p < \infty$, 
$l \in \mathbb{N}$, $s \geq lp$, there exists 
$C(\epsilon) = C(\epsilon, n, rank(E), p, l, g, \gamma, s) \in \R_{>0}$ such that for $\phi$ a smooth section
we have
\begin{equation*}
\big{\vert}\big{\vert}\gamma^{s/p}\nabla^{(l)}\phi\big{\vert}\big{\vert}_{L^p(M)} \leq
\epsilon  \big{\vert}\big{\vert}\gamma^{(s+jp)/p}\nabla^{(l+j)}\phi\big{\vert}\big{\vert}_{L^p(M)} + 
C(\epsilon)\vert\vert\phi\vert\vert_{L^p(M), \gamma > 0}.
\end{equation*}
In particular, for $p = 2$ and some constant $K \geq 1$, we have
\begin{equation*}
K\big{\vert}\big{\vert}\gamma^{s/2}\nabla^{(l)}\phi\big{\vert}\big{\vert}^2_{L^2(M)} \leq
\epsilon  \big{\vert}\big{\vert}\gamma^{(s+2j)/2}\nabla^{(l+j)}\phi\big{\vert}\big{\vert}^2_{L^2(M)} + 
C(\epsilon)K^2\vert\vert\phi\vert\vert^2_{L^2(M), \gamma > 0}.
\end{equation*}
\end{lem}

\subsection{Commutation formulae for connections}
	
During the study of the higher order Seiberg-Witten flow, there will be times when we need to 
switch derivatives, leading to the need for various commutation formulas.
We collect here various results on formulas for commuting connections. 

We start with the Weitzenb\"{o}ck identity, see theorem 9.4.1 in \cite{petersen}.

\begin{prop}[Weitzenb\"{o}ck identity]\label{connection_1}
Let $(M, g)$ be a Riemannian manifold with Levi-civita connection $\nabla_M$. We also denote by 
$\nabla_M$ the differential operator from $\Omega^P(M) \rightarrow  T^*M \otimes \Omega^p(M)$ induced
by the Levi-Civita connection. Let $\Delta_H = dd^* + d^*d$ denote
the Hodge Laplacian, and let $\nabla_M^*\nabla_M = \Delta_M$ denote the Bochner Laplacian. 
Given $\omega \in \Omega^p(M)$, we have
\begin{equation*}
\Delta_M\omega = \Delta_H\omega + Rm* \omega. 
\end{equation*}
\end{prop}

The following lemma tells us how to switch derivatives, see lemma 5.12 in \cite{kelleher}.
\begin{lem}\label{connection_2}
Let $E$ be a Hermitian vector bundle over a Riemannian manifold $(M, g)$, with metric compatible connection 
$\nabla$. Let $\phi$ denote a 
section of $E$. We have
\begin{align*}
\nabla_{i_k}\nabla_{i_{k-1}}\cdots\nabla_{i_1}\nabla_{j_1}\nabla_{j_2}\cdots\nabla_{j_k}\phi &=
\nabla_{i_k}\nabla_{j_k}\nabla_{i_{k-1}}\nabla_{j_{k-1}}\cdots\nabla_{i_1}\nabla_{j_1}\phi \\
&\hspace{0.5cm}+
\sum_{l=0}^{2k-2}\big{(}(\nabla_M^{(l)}Rm + \nabla^{(l)}F_{\nabla})*\nabla^{(2k-2-l)}\phi\big{)}.
\end{align*}
where $F_{\nabla}$ denotes the curvature associated to $\nabla$, and $Rm$ is the Riemannian curvature.
\end{lem}

A simple corollary of this lemma is the following.

\begin{cor}\label{connection_3}
Let $E$ be a Hermitian vector bundle over a Riemannian manifolds $(M,g)$, with metric compatible connection 
$\nabla$. Let $\Delta = 
\nabla^*\nabla$ denote the Bochner
Laplacian. Given a section $\phi$ of $E$, we have
\begin{equation*}
\nabla^{*(k)}\nabla^{(k)}\phi = \Delta^{(k)}\phi + 
\sum_{j=0}^{2k-2}\bigg{(}(\nabla_M^{(j)}Rm + \nabla^{(j)}F_{\nabla})*\nabla^{(2k-2-j)}\phi\bigg{)}.
\end{equation*}
\end{cor}

We will also need to commute derivatives with Laplacian terms. The following lemma shows us how to
do this, see corollary 5.15 in \cite{kelleher}.

\begin{lem}\label{connection_4}
Let $E$ be a Hermitian vector bundle over a Riemannian manifold $(M, g)$, with metric compatible connection 
$\nabla$. Let $\Delta = \nabla^*\nabla$ denote the Bochner
Laplacian, and let $\phi$ be a section of $E$. We have
\begin{equation*}
\nabla^{(n)}\Delta^{(k)}\phi = \Delta^{(k)}\nabla^{(n)}\phi + 
\sum_{j=0}^{2k+n-2}\bigg{(}(\nabla_M^{(j)}Rm + \nabla^{(j)}F_{\nabla})*\nabla^{(2k+n-2-j)}\phi\bigg{)}.
\end{equation*}
\end{lem}

Combining corollary \ref{connection_3} and lemma \ref{connection_4} we obtain

\begin{cor}\label{connection_5}
Let $E$ be a Hermitian vector bundle over a Riemannian manifold $(M, g)$, with metric compatible connection 
$\nabla$. Let $\Delta = \nabla^*\nabla$ denote the Bochner
Laplacian, and let $\phi$ be a section of $E$. We have
\begin{equation*}
\nabla^{(n)}\nabla^{*(k)}\nabla^{(k)}\phi = \Delta^{(k)}\nabla^{(n)}\phi + 
\sum_{j=0}^{2k+n-2}\bigg{(}(\nabla_M^{(j)}Rm + \nabla^{(j)}F_{\nabla})*\nabla^{(2k+n-2-j)}\phi\bigg{)}.
\end{equation*}
\end{cor}

We will also need the following integration by parts formula, see lemma 5.13 in \cite{kelleher}.

\begin{lem}\label{connection_6}
Let $E$ be a Hermitian vector bundle over a Riemannian manifold $(M, g)$, with metric compatible connection 
$\nabla$. Let $\Delta = \nabla^*\nabla$ denote the Bochner
Laplacian, and let $\phi$ and $\psi$ be sections of $E$. We have
\begin{equation*}
\int_M \langle \nabla^{(k)}\phi, \nabla^{(k)}\phi\rangle d\mu = 
\int_M (-1)^k\langle \phi, \Delta^{(k)}\psi\rangle d\mu + 
\bigg{\langle}\phi, \sum_{v=0}^{2k-2}\bigg{(}(\nabla_M^{(v)}Rm + \nabla^{(v)}F_{\nabla})*
\nabla^{(2k-2-v)}\phi\bigg{)}\bigg{\rangle}.    
\end{equation*}
\end{lem}

\subsection{Theorems from gauge theory}

The following two theorems from gauge theory will be used in section \ref{blowup_analysis}. We state them
here for the convenience of the reader.

The first theorem we will need is the Coloumb gauge theorem, theorem 1.3 in \cite{uhlenbeck}.

\begin{thm}[Coloumb gauge theorem]\label{coloumb_gauge}
Let $M = B_1(0) \subseteq \R^n$,  $E = B_1(0) \times \R^m$ be a trivial bundle over $M$, and $n \leq 2p$.  
Suppose $\nabla = d + A$ is a connection on $E$. Then there exists constants $\kappa(n) > 0$ and $c(n) < \infty$ 
such
that if $\vert\vert F_{\nabla}\vert\vert^{n/2}_{L^{n/2}} \leq \kappa(n)$, then $\nabla$ is gauge equivalent
to a connection $d + \tilde{A}$ where $\tilde{A}$ satisfies:
\begin{itemize}
\item[1.] $d^*\tilde{A} = 0$

\item[2.] $\vert\vert \tilde{A}\vert\vert_{C^{p,1}} \leq c(n)\vert\vert F_{\nabla}\vert\vert_{C^{p,0}}$.
\end{itemize}    
\end{thm}

The second theorem we will need is a theorem that allows us to glue together a sequence of connections
defined on small open sets, see corollary 4.4.8 \cite{kronheimer}.

\begin{thm}\label{gauge_patching}
Suppose $\{\nabla^i\}$ is a sequence of connections on $E$ over $M$ with the following property: For each
$x \in M$ there exists a neighbourhood $U_x$, and a subsequence $\{\nabla^{i_j}\}$ with corresponding 
sequence of gauge transformations $s_{i_j}$ defined over $M$ such that $s^*_{i_j}\nabla^{i_j}$ converges 
over $U_x$. Then there exists a single subsequence $\{\nabla^{i_{j_k}}\}$ defined over $M$ such that
$s_{i_{j_k}}^*\nabla^{i_{j_k}}$ converges over all of $M$.
\end{thm}

\section*{Acknowledgements}
The author wishes to acknowledge support from the Bejing International Centre for Mathematical Research, and the
Jin Guang Mathematical Foundation.










\end{document}